\documentclass[12pt,twoside]{amsart}

\usepackage{amsmath,amsthm,amscd,amssymb,mathrsfs,graphicx,amsfonts,mathrsfs}
\usepackage{amsfonts}
\usepackage{amssymb,enumerate}
\usepackage{amsthm}
\usepackage[all]{xy}
\usepackage{hyperref}
\allowdisplaybreaks[4] \footskip=12pt
\renewcommand{\uppercasenonmath}[1]{}

\numberwithin{equation}{section} \theoremstyle{plain}
\newtheorem*{thm*}{Main Theorem}
\newtheorem{thm}{Theorem}[section]

\newtheorem*{cor*}{Corollary}
\newtheorem{lem}[thm]{Lemma}
\newtheorem*{lem*}{Lemma}

\newtheorem*{fact*}{Fact}

\newtheorem*{nota*}{Notation}
\newtheorem{prop}[thm]{Proposition}
\newtheorem*{prop*}{Proposition}
\newtheorem{rem}[thm]{Remark}
\newtheorem*{rem*}{Remark}

\newtheorem*{observation*}{Observation}
\newtheorem{exa}[thm]{Example}
\newtheorem*{exa*}{Example}
\newtheorem{df}[thm]{Definition}
\newtheorem*{df*}{Definition}

\newtheorem*{conj*}{Conjecture}

\newcommand{\D}{\mathcal{D}}

\begin{document}
\footnote[0]{$^{\ast}$Corresponding author}
\begin{center}
{\large  \bf  Construction of $n$-cotorsion pairs in some abelian categories}

\vspace{0.5cm}  Wenjing Chen$^{\ast,1,2}$ and Junjie Zhang$^{1}$\\
1 College of Mathematics and Statistics, Northwest Normal University, Lanzhou 730070,
Gansu, China\\
2 Gansu Provincial Research Center for Basic Disciplines of Mathematics and Statistics, Lanzhou 730070, Gansu, China\\
E-mail addresses: chenwj@nwnu.edu.cn, 17797575425@163.com
\end{center}

\bigskip
\centerline { \textsc{Abstract}}
\leftskip10truemm \rightskip10truemm \noindent In this paper, we mainly investigate $n$-cotorsion pairs in two types of abelian categories. Firstly, we construct left (right) $n$-cotorsion pairs in comma categories by improving some isomorphisms between homology groups and combining some classes of objects in comma categories. Secondly, based on known homological formulas over trivial ring extensions, we study some classes of modules satisfying certain conditions over such rings, and we establish left (right) $n$-cotorsion pairs over trivial ring extensions through these classes. Then corresponding conclusions over Morita rings with zero bimodule homomorphisms are directly presented. The hereditary property of left (right) $n$-cotorsion pairs is investigated in both cases. Finally, we apply main results obtained in comma categories and over trivial ring extensions to formal triangular matrix rings. At the same time, via these applications, we found that relevant conditions in known results have been improved and the known results have been elevated.

\leftskip0truemm \rightskip0truemm
\bigskip

\section{Introduction}
The notion of cotorsion pairs was introduced by Salce in the category of abelian groups in \cite{L1979}. Cotorsion pairs, defined by replacing the \textbf{Hom} functor with the \textbf{Ext} functor, are the analogue of classical torsion pairs. On the one hand, cotorsion pairs are closely related to precovers and preenvelopes. On the other, cotorsion pairs play a crucial role in the study of relative homological algebra, model structures and recollements. Regarding cotorsion pairs, please refer to \cite{EEE-2000}, \cite{RJ2012}, \cite{M2002}, \cite{L2020}, \cite{L2024}, \cite{YJD2024} for more details.

Motivated by some properties satisfied by Gorenstein projective and Gorenstein injective modules over an Iwanaga-Gorenstein ring, Huerta, Mendoza and P\'{e}rez introduced the concepts of left $n$-cotorsion pairs and right $n$-cotorsion pairs in abelian categories and described a series of properties of left (resp. right) $n$-cotorsion pairs in \cite{MOM2021}. A 1-cotorsion pair is actually consistent with a complete cotorsion pair, defined by Salce in \cite{L1979}. By left or right $n$-cotorsion pairs, we can directly give some approximation classes. As a generalization of cotorsion pairs, $n$-cotorsion pairs soon attracted extensive attention from researchers. He and Zhou introduced the notion of $n$-cotorsion pairs in extriangulated categories with enough projective and injective objects in \cite{JP2022} and showed that there exists a one-to-one correspondence between $n$-cotorsion pairs and $(n+1)$-cluster tilting subcategories, which generalizes the work in \cite{MOM2021} since the notion of extriangulated categories is a common generalization of exact categories and triangulated categories. Recently, Long and Zhang have studied some special classes of modules over formal triangular matrix rings in \cite{TX2025} and constructed left (resp. right) $n$-cotorsion pairs over such rings by means of these classes of modules together with left (resp. right) $n$-cotorsion pairs over internal rings of formal triangular matrix rings. Analogous to cotorsion pairs, using recollement as a tool, one can also investigate $n$-cotorsion pairs, for example, see \cite{WJK2024} and \cite{JJ2025}.

In recent years, cotorsion pairs in various categories have become a significant and widely studied subject for scholars, refer to \cite{HP2019}, \cite{JH2022}, \cite{L2020}, \cite{L2024}, \cite{YJD2024}. These studies motivate us to further explore related topics concerning $n$-cotorsion pairs. Note that based on the study of cotorsion pairs, the study of $n$-cotorsion pairs can be carried out very well. In fact, a category of modules over a formal triangular matrix ring is a common example of both a comma category and a category of modules over a trivial ring extension, or to be more precise, a category of modules over a Morita ring with zero bimodule homomorphisms. Therefore, in this paper, we will conduct the study on $n$-cotorsion pairs from both the perspectives of comma categories and trivial ring extensions and provide the description of $n$-cotorsion pairs in the categories of modules over Morita rings with zero bimodule homomorphisms and formal triangular matrix rings, respectively.

Let $\mathcal{A}$ and $\mathcal{B}$ be abelian categories, $T:\mathcal{B}\rightarrow\mathcal{A}$ a right exact covariant functor, and $G:\mathcal{A}\rightarrow\mathcal{B}$ a left exact covariant functor. According to the work in \cite{RPI1975} and \cite{N1983}, there exist two abelian comma categories $(T\downarrow\mathcal{A})$ and $(\mathcal{B}\downarrow G)$. Examples of comma categories include but are not limited to: categories of modules or complexes over triangular matrix rings, morphism categories of abelian categories, and so on (see \cite[\text{Example}~2.2]{JH2022}). It should be noted that comma categories not only give rise to adjoint functors for constructing recollements in abelian categories and triangulated categories, but also are used in the study of Auslander-Reiten quivers and tilting modules, refer to \cite{XJ2022}, \cite{YDR2024}, \cite{N1983}, \cite{C2014}, \cite{WZ2022}. Via numerous favorable homological properties of comma categories and existing results on cotorsion pairs therein, we will also consider $n$-cotorsion pairs in comma categories.

The notion of trivial extension of a ring by a bimodule is an important extension of a ring. Let $R$ be an associative ring and $M$ an $R$-$R$-bimodule. The Cartesian product $R\times M$, with the natural addition and multiplication given by $(r_1,m_1)(r_2,m_2)=(r_1r_2,r_1m_2+m_1r_2)$, becomes a ring, where $r_1,r_2\in R$ and $m_1,m_2\in M$. This ring is called the trivial extension of the ring $R$ by the bimodule $M$, trivial ring extension for short, and denoted by $R \ltimes M$ in \cite{RPI1975}. In fact, there exists a direct containment relationship as follows:
$$
\aligned
\{\text{formal~triangular~matrix~rings}\}
&\subseteq \{\text{Morita~rings~with~zero~bimodule~homomorphisms}\}\\
&\subseteq\{\text{trivial~ring~extensions}\}.
\endaligned
$$
Thus, trivial ring extensions play a crucial role in ring theory. Many scholars have focused on trivial ring extensions and modules over such rings, and studied them from different aspects, such as homological dimension theory, silting theory, cotorsion theory and model structure theory, refer to \cite{RPI1975}, \cite{L2023}, \cite{L2024}, \cite{HK2020}, \cite{IJ1973}. More specifically, Mao investigated how to construct cotorsion pairs and Hovey triples over trivial ring extensions in detail in \cite{L2024}.
Naturally, we will also consider $n$-cotorsion pairs in categories of modules over trivial ring extensions.

The main aim of the paper is to investigate how to construct left (resp. right) $n$-cotorsion pairs in comma categories and categories of modules over trivial ring extensions. Under this motivation, we investigate some special classes of objects and describe left (resp. right) $n$-cotorsion pairs in comma categories and categories of modules over trivial ring extensions, respectively. We also apply the foregoing results obtained to formal triangular matrix rings. However, we are aware that our results have improved to some extent.

The paper is organized as follows:

In Section 2, we recall some notions and basic facts.

In Section 3, we first discuss precovers and preenvelopes in the comma categories $(T\downarrow\mathcal{A})$ and $(\mathcal{B}\downarrow G)$ (see Propositions 3.1, 3.2 and 3.3). Next, we improve some isomorphisms between homology groups (see Lemmas 3.5 and 3.6). Then, using these improved isomorphisms together with some classes of objects in the comma categories, we construct left (resp. right) $n$-cotorsion pairs in the comma categories (see Theorems 3.11 and 3.12). Finally, we further investigate the heredity of left (resp. right) $n$-cotorsion pairs in the comma categories (see Propositions 3.16 and 3.17).

In Section 4, based on the existing homology formulas over trivial ring extensions, we study classes of modules satisfying certain conditions over such rings, and construct left (resp. right) $n$-cotorsion pairs over trivial ring extensions by means of these classes of  modules (see Theorems 4.4 and 4.6). In this case, we investigate the heredity of left (resp. right) $n$-cotorsion pairs (see Propositions 4.9 and 4.10). Since Morita rings with zero bimodule homomorphisms are a special case of trivial ring extensions, as an application, we apply the previous results obtained to these rings (see Propositions 4.11, 4.12, 4.13 and 4.14).

In Section 5, since the category of modules over a formal triangular matrix ring is a special case of comma categories and categories of modules over trivial ring extensions, we apply the conclusions from Sections 3 and 4 to formal triangular matrix rings. As a consequence of these applications, we improve the relevant conditions in the existing results and the existing results have been elevated.

\section{Preliminaries}
Throughout the paper, we always assume that $\mathcal{A}$ and $\mathcal{B}$ are abelian categories (not necessarily with enough projective and injective objects), and that all classes of objects in $\mathcal{A}$ and $\mathcal{B}$ contain the zero object and are closed under isomorphisms. In this paper, all rings are associative rings with identity and all modules are unitary. For any ring $R$, we use $R$-Mod (Mod-$R$) to denote the category of all left (right) $R$-modules. Let $\mathcal{C}$ and $\mathcal{D}$ be classes of objects in $\mathcal{A}$, and $i\geqslant 1$ an integer. The notation $\mathrm{Ext}
^i_{\mathcal{A}}(\mathcal{C},\mathcal{D})=0$ means that $\mathrm{Ext}^i_{\mathcal{A}}(C,D)=0$ for any $C\in\mathcal{C}$ and any $D\in\mathcal{D}$. The right $i$-th orthogonal class of $\mathcal{C}$ is defined by $\mathcal{C}^{\bot_{i}}:=\{N\in\mathcal{A}\mid\mathrm{Ext}_{\mathcal{A}}
^{i}(C,N)=0,~\forall~C\in\mathcal{C}\}$. Dually, the left $i$-th orthogonal class of $\mathcal{C}$ is given by ${^{\bot_{i}}\mathcal{C}}:=\{M\in\mathcal{A}\mid
\mathrm{Ext}_{\mathcal{A}}^{i}(M,C)=0,~\forall~C\in\mathcal{C}\}$. In particular, we set $\mathcal{C}^{\bot_{1}}:=\mathcal{C}^{\bot}$ and ${^{\bot_{1}}\mathcal{C}}
=:{^{\bot}\mathcal{C}}$ when $i=1$. Moreover, for any positive integer $m$, we define $\mathcal{C}^{\bot_{[1,m]}}
:=\bigcap\limits_{i=1}\limits^{m}\mathcal{C}^{\bot_{i}}$ and ${^{\bot_{[1,m]}}\mathcal{C}}:=\bigcap\limits
_{i=1}\limits^{m}{^{\bot_{i}}\mathcal{C}}$.
Let $\mathcal{X}$ be a class of right $R$-modules and $\mathcal{Y}$ a class of left $R$-modules. Similarly, the notation $\mathrm{Tor}_i^R(\mathcal{X},\mathcal{Y})=0$ means that $\mathrm{Tor}_{i}^R(X,Y)=0$ for any $X\in\mathcal{X}$ and any $Y\in\mathcal{Y}$.

{\bf 2.1 Resolving class and coresolving class}

Let $\mathcal{C}$ and $\mathcal{D}$ be classes of objects in $\mathcal{A}$. We say that $\mathcal{C}$ is a resolving class if it is closed under extensions and kernels of epimorphisms. Dually, $\mathcal{D}$ is called a coresolving class if it is closed under extensions and cokernels of monomorphisms. Furthermore, assume that $\mathcal{A}$ has enough projective objects. A resolving class $\mathcal{C}$ is said to be a projectively resolving class if it contains all projective objects in $\mathcal{A}$. Assume that $\mathcal{A}$ has enough injective objects. A coresolving class $\mathcal{D}$ is said to be an injectively coresolving class if it contains all injective objects in $\mathcal{A}$.

{\bf 2.2 Resolution dimension and coresolution dimension}

Let $\mathcal{C}$ be a class of objects of $\mathcal{A}$. Given an object $A\in\mathcal{A}$ and a nonnegative integer $m$, a $\mathcal{C}$-coresolution of $A$ of length $m$ is an exact sequence
$$0\rightarrow A\rightarrow C^0\rightarrow C^1\rightarrow\cdots\rightarrow C^{m-1}\rightarrow C^{m}\rightarrow 0$$
in $\mathcal{A}$, where $C_k\in\mathcal{C}$ for every $0\leqslant k\leqslant m$. The coresolution dimension of $A$ with respect to $\mathcal{C}$ (or the $\mathcal{C}$-coresolution dimension of $A$), denoted by $\mathrm{coresdim}_{\mathcal{C}}(A)$, is defined as the smallest nonnegative integer $m$ such that $A$ has a $\mathcal{C}$-coresolution of length m. If such $m$ does not exist, we set $\mathrm{coresdim}_{\mathcal{C}}(A):=\infty$. Dually, we have the concepts of the $\mathcal{C}$-resolution of $A$ of length $m$ and the resolution dimension of $A$ with respect to $\mathcal{C}$ (or the $\mathcal{C}$-resolution dimension of $A$), denoted by $\mathrm{resdim}_{\mathcal{C}}(A)$. Based on these two homological dimensions, we shall subsequently consider the following classes of objects in $\mathcal{A}$:
$$\mathcal{C}_{m}^{\wedge}
:=\{A\in\mathcal{A}\mid\mathrm{resdim}_{\mathcal{C}}(A)\leqslant m\}\quad\text{and}\quad\mathcal{C}_{m}^{\vee}:=\{A\in \mathcal{A}\mid\mathrm{coresdim}_{\mathcal{C}}(A)\leqslant m\}.$$

{\bf 2.3 $n$-Cotorsion pair}

\begin{df}\label{prop:2.4}{\rm{(\cite[\text{Definition}~2.1]{MOM2021})
Let $\mathcal{C}$ and $\mathcal{D}$ be classes of objects in $\mathcal{A}$, and $n$ a positive integer. We say that $(\mathcal{C}, \mathcal{D})$ is a \emph{left $n$-cotorsion pair} in $\mathcal{A}$ if the following conditions are satisfied:

$\mathrm{(1)}$ $\mathcal{C}$ is closed under direct summands.

$\mathrm{(2)}$ $\mathrm{Ext}^{i}_{\mathcal{A}}(\mathcal{C},\mathcal{D})=0$ for every $1\leqslant i\leqslant n$.

$\mathrm{(3)}$ For every object $A\in\mathcal{A}$, there exists a short exact sequence $0\rightarrow K\rightarrow C \rightarrow A\rightarrow 0$, where $C\in\mathcal{C}$ and $K\in\mathcal{D}_{n-1}^{\wedge}$.

\noindent Dually, we say that $(\mathcal{C},\mathcal{D})$ is a \emph{right $n$-cotorsion pair} in $\mathcal{A}$ if condition (2) above is satisfied, with $\mathcal{D}$ closed under direct summands, and if every object of $\mathcal{A}$ can be embedded into an object of $\mathcal{D}$ with cokernel in $\mathcal{C}_{n-1}^{\vee}$. Finally, $\mathcal{C}$ and $\mathcal{D}$ form an \emph{$n$-cotorsion pair} $(\mathcal{C},\mathcal{D})$ in $\mathcal{A}$ if $(\mathcal{C}, \mathcal{D})$ is both a left and right $n$-cotorsion pair in $\mathcal{A}$.
}}
\end{df}

\begin{exa}\label{prop:2.4}{\rm{(\cite[\text{Example}~5.1]{MOM2021})
Assume that $R$ is an $n$-Iwanaga-Gorenstein ring with some integer $n\geqslant0$, that is, $R$ is both left and right noetherian and has self-injective dimension at most $n$ on both the left and the right. Then $($the class of all Gorenstein projective left $R$-modules, the class of all projective left $R$-modules$)$ is a left $n$-cotorsion pair and $($the class of all injective left $R$-modules, the class of all Gorenstein injective left $R$-modules$)$ is a right $n$-cotorsion pair in $R$-Mod. In particular, when $n=0$, a $0$-Iwanaga-Gorenstein ring is just a quasi-Frobenius ring. In this case, the left $n$-cotorsion pair above coincides with the classical projective cotorsion pair, and the right $n$-cotorsion pair above coincides with the classical injective cotorsion pair.
}}
\end{exa}

{\bf 2.4 Comma category}

Now, we recall some notions and facts in comma categories, which can be found in \cite{RPI1975} and \cite{N1983}.

\begin{df}\label{prop:2.4}{\rm{(See \cite{RPI1975}, \cite{N1983})
Let $T:\mathcal{B}\rightarrow\mathcal{A}$ be a right exact covariant functor. Then the \emph{comma category} $(T\downarrow\mathcal{A})$ is defined as follows:

$\mathrm{(1)}$ The objects are triplets $\left(\begin{smallmatrix}A\\B\end{smallmatrix}\right)_{\varphi}$ with $A\in\mathcal{A}$, $B\in\mathcal{B}$ and $\varphi:T(B)\rightarrow A$ being a morphism in $\mathcal{A}$.

$\mathrm{(2)}$ A morphism $\left(\begin{smallmatrix}a\\b\end{smallmatrix}\right):\left(\begin{smallmatrix}A\\B \end{smallmatrix}\right)_{\varphi}\rightarrow
\left(\begin{smallmatrix}A'\\B'\end{smallmatrix}\right)_{\varphi'}$ is given by two morphisms $a:A\rightarrow A'$ in $\mathcal{A}$ and $b:B\rightarrow B'$ in $\mathcal{B}$ such that the diagram
$$\small\xymatrix{
      & T(B) \ar[d]_{\varphi} \ar[r]^{T(b)} & T(B') \ar[d]^{\varphi'}  & \\
      & A  \ar[r]_{a} &  A'  &  \\   }\vspace*{2mm}$$
is commutative, that is, $\varphi'T(b)=a\varphi$.
}}
\end{df}

\begin{df}\label{prop:2.4}{\rm{(See \cite{RPI1975}, \cite{N1983})
Let $G:\mathcal{A}\rightarrow\mathcal{B}$ be a left exact covariant functor. Then the \emph{comma category} $(\mathcal{B}\downarrow G)$ is defined as follows:

$\mathrm{(1)}$ The objects are triplets $\left(\begin{smallmatrix}A\\B\end{smallmatrix}\right)_{\phi}$ with $A\in\mathcal{A}$, $B\in\mathcal{B}$ and $\phi:B\rightarrow G(A)$ being a morphism in $\mathcal{B}$.

$\mathrm{(2)}$ A morphism $\left(\begin{smallmatrix}a\\b\end{smallmatrix}\right):\left(\begin{smallmatrix}A\\B \end{smallmatrix}\right)_{\phi}\rightarrow
\left(\begin{smallmatrix}A'\\B'\end{smallmatrix}\right)_{\phi'}$ is given by two morphisms $a:A\rightarrow A'$ in $\mathcal{A}$ and $b:B\rightarrow B'$ in $\mathcal{B}$ such that the diagram
$$\small\xymatrix{
      & B \ar[d]_{\phi} \ar[r]^{b} & B' \ar[d]^{\phi'}  & \\
      & G(A)  \ar[r]_{G(a)} &  G(A')  &  \\   }\vspace*{2mm}$$
is commutative, that is, $\phi'b=G(a)\phi$.
}}
\end{df}

If there is no possible confusion, we sometimes omit the morphism $\varphi$ (resp. $\phi$). Recall from \cite{RPI1975} that the comma category $(T\downarrow\mathcal{A})$ is indeed an abelian category since the functor $T$ is assumed to be right exact. Similarly, the comma category $(\mathcal{B}\downarrow G)$ is also an abelian category.
A sequence
$$0\longrightarrow\left(\begin{smallmatrix}A'\\B'\end{smallmatrix}\right)_{\varphi^{'}}\overset{\left(
\begin{smallmatrix}a'\\b'\end{smallmatrix}\right)}\longrightarrow\left(\begin{smallmatrix}A\\B
\end{smallmatrix}\right)_{\varphi}\overset{\left(\begin{smallmatrix}a''\\b''\end{smallmatrix}\right)}
\longrightarrow\left(\begin{smallmatrix}A''\\B''\end{smallmatrix}\right)_{\varphi^{''}}\longrightarrow0$$
in $(T\downarrow\mathcal{A})$ is exact if and only if both the sequence
$0\rightarrow A'\overset{a'}\rightarrow A\overset{a''}\rightarrow A''\rightarrow0$
in $\mathcal{A}$ and the sequence
$0\rightarrow B'\overset{b'}\rightarrow B\overset{b''}\rightarrow B''\rightarrow0$
in $\mathcal{B}$ are exact. Of course, an analogous fact holds for the comma category $(\mathcal{B}\downarrow G)$.
Next, we give two typical examples of comma categories, which can also be found in some references.

\begin{exa}\label{prop:2.4}{\rm{Let $R$ and $S$ be two rings, $M$ an $R$-$S$-bimodule, and $\Lambda=\left(
\begin{smallmatrix}R&M\\0&S\end{smallmatrix}\right)$ a formal triangular matrix ring.

$\mathrm{(1)}$ If we define $T\cong M\otimes_{S}-:S\text{-}\mathrm{Mod}\rightarrow R\text{-}\mathrm{Mod}$, then we get that $\Lambda\text{-}\mathrm{Mod}$ is equivalent to the comma category $(T\downarrow R\text{-}\mathrm{Mod})$.

$\mathrm{(2)}$ If we define $G\cong \mathrm{Hom}_{R}(M,-):R\text{-}\mathrm{Mod}\rightarrow S\text{-}\mathrm{Mod}$, then we get that $\Lambda\text{-}\mathrm{Mod}$ is equivalent to the comma category $(S\text{-}\mathrm{Mod}\downarrow G)$.
}}
\end{exa}

In the sense of category equivalences, the above example demonstrates that the categories $(T\downarrow R\text{-}\mathrm{Mod})$, $\Lambda\text{-}\mathrm{Mod}$ and $(S\text{-}\mathrm{Mod}\downarrow G)$ are mutually equivalent.

\begin{rem}\label{prop:2.4}{\rm{For the comma category $(T\downarrow\mathcal{A})$ and the comma category $(\mathcal{B}\downarrow G)$, we give the following functors and facts.

$\mathrm{(1)}$ We have a functor $\mathbf{p}:\mathcal{A}\times\mathcal{B}\rightarrow(T\downarrow\mathcal{A})$ with $\mathbf{p}(A,B)
=\left(\begin{smallmatrix}A\oplus T(B)\\B\end{smallmatrix}\right)_{\left(\begin
{smallmatrix}0\\1\end{smallmatrix}\right)}$ and $\mathbf{p}(a,b)=\left(\begin{smallmatrix}a\oplus T(b)\\b\end{smallmatrix}\right)$, where $(A,B)$ is an object in $\mathcal{A}\times\mathcal{B}$, $(a,b):(A,B)\rightarrow (A',B')$ is a morphism in $\mathcal{A}\times\mathcal{B}$, and $1$ is the identity morphism on $T(B)$. It follows from \cite{JH2022} that $\mathbf{p}(A,B)=\mathbf{p}(A,0)\oplus\mathbf{p}(0,B)$. Moreover, $\mathbf{p}$ preserves projective objects if $\mathcal{A}$ and $\mathcal{B}$ have enough projective objects.

$\mathrm{(2)}$ We have a functor $\mathbf{h}:\mathcal{A}\times\mathcal{B}\rightarrow(\mathcal{B}\downarrow G)$ with $\mathbf{h}(A,B)
=\left(\begin{smallmatrix}A\\B\oplus G(A)\end{smallmatrix}\right)_{(0,1)}$ and $\mathbf{h}(a,b)=\left(\begin{smallmatrix}a\\b\oplus G(a)\end{smallmatrix}\right)$, where $(A,B)$ is an object in $\mathcal{A}\times\mathcal{B}$, $(a,b):(A,B)\rightarrow (A',B')$ is a morphism in $\mathcal{A}\times\mathcal{B}$, and $1$ is the identity morphism on $G(A)$. It follows from \cite{YJD2024} that $\mathbf{h}(A,B)=\mathbf{h}(A,0)\oplus
\mathbf{h}(0,B)$. Moreover, $\mathbf{h}$ preserves injective objects if $\mathcal{A}$ and $\mathcal{B}$ have enough injective objects.

$\mathrm{(3)}$ If the abelian categories $\mathcal{A}$ and $\mathcal{B}$ have enough projective objects, then the comma category $(T\downarrow\mathcal{A})$ has enough projective objects. If the abelian categories $\mathcal{A}$ and $\mathcal{B}$ have enough injective objects, then the comma category $(\mathcal{B}\downarrow G)$ has enough injective objects. As a matter of fact, for any $\left(\begin{smallmatrix}A\\B\end{smallmatrix}\right)_{\varphi}
\in(T\downarrow\mathcal{A})$, since $\mathcal{B}$ has enough projective objects, there exists an epimorphism $Q\overset{b}\rightarrow B\rightarrow0$ in $\mathcal{B}$, where $Q$ is a projective object in $\mathcal{B}$. Since $T:\mathcal{B}\rightarrow\mathcal{A}$ is a right exact functor, there exists an epimorphism $T(Q)\overset{T(b)}\rightarrow T(B)\rightarrow0$ in $\mathcal{A}$. Because $\mathcal{A}$ has enough projective objects, there exists an epimorphism $P\overset{a}\rightarrow A\rightarrow0$ in $\mathcal{A}$, where $P$ is a projective object in $\mathcal{A}$. Thus, there exists an epimorphism
$$\left(\begin{smallmatrix}P\oplus T(Q)\\Q\end{smallmatrix}\right)_{\left(\begin{smallmatrix}0\\1
\end{smallmatrix}\right)}\overset{\left(\begin{smallmatrix}(a,{\varphi} T(b))\\b\end{smallmatrix}\right)}\longrightarrow\left(\begin{smallmatrix}A\\B\end{smallmatrix}\right)_{\varphi}
\longrightarrow0$$
in $(T\downarrow\mathcal{A})$. By $\mathrm{(1)}$, $\mathbf{p}(P,Q)
=\left(\begin{smallmatrix}P\oplus T(Q)\\Q\end{smallmatrix}\right)_{\left(\begin{smallmatrix}0\\1\end{smallmatrix}\right)}$ is a projective object in $(T\downarrow\mathcal{A})$. Therefore, $(T\downarrow\mathcal{A})$ has enough projective objects. Similarly, we can get that $(\mathcal{B}\downarrow G)$ has enough injective objects.

$\mathrm{(4)}$ If we define $\mathbf{q}:(T\downarrow\mathcal{A})\rightarrow\mathcal{A}\times\mathcal{B}$ with $\mathbf{q}\left(\begin{smallmatrix}A\\B\end{smallmatrix}\right)_{\varphi}=(A,B)$ and $\mathbf{q}\left(\begin{smallmatrix}a\\b\end{smallmatrix}\right)=(a,b)$, for any object $\left(\begin{smallmatrix}A\\B\end{smallmatrix}\right)_{\varphi}$ and any morphism $\left(\begin{smallmatrix}a\\b\end{smallmatrix}\right):\left(\begin{smallmatrix}A\\B \end{smallmatrix}\right)_{\varphi}\rightarrow
\left(\begin{smallmatrix}A'\\B'\end{smallmatrix}\right)_{\varphi'}$ in $(T\downarrow\mathcal{A})$, then by \cite[Remark~2.4~(2)]{JH2022}, $(\mathbf{p},\mathbf{q})$ is an adjoint pair.

$\mathrm{(5)}$ If we define $\mathbf{g}:(\mathcal{B}\downarrow G)\rightarrow\mathcal{A}\times\mathcal{B}$ with $\mathbf{g}\left(\begin{smallmatrix}A\\B\end{smallmatrix}\right)_{\phi}=(A,B)$ and $\mathbf{g}\left(\begin{smallmatrix}a\\b\end{smallmatrix}\right)=(a,b)$, for any object $\left(\begin{smallmatrix}A\\B\end{smallmatrix}\right)_{\phi}$ and any morphism $\left(\begin{smallmatrix}a\\b\end{smallmatrix}\right):\left(\begin{smallmatrix}A\\B \end{smallmatrix}\right)_{\phi}\rightarrow
\left(\begin{smallmatrix}A'\\B'\end{smallmatrix}\right)_{\phi'}$ in $(\mathcal{B}\downarrow G)$, then by \cite[Remark~3.2~(2)]{YJD2024}, $(\mathbf{g},\mathbf{h})$ is an adjoint pair.
}}
\end{rem}

{\bf 2.5 Trivial ring extension}

Let $R\ltimes M$ be the trivial extension of a ring $R$ by an $R$-$R$-bimodule $M$. Recall from \cite{RPI1975} that the category $R\ltimes M$-Mod is isomorphic to the category $\Xi$, whose objects are pairs $(X,f)$ with $X$ a left $R$-module and $f\in\mathrm{Hom}_{R}(M\otimes_{R}X,X)$ such that the composition $M\otimes_{R}M\otimes_{R}X\overset{M\otimes_{R}f}\longrightarrow M\otimes_{R}X\overset{f}\longrightarrow X$ is 0, and morphisms $(X,f)\rightarrow(Y,g)$ are morphisms $\gamma$ with $\gamma\in\mathrm{Hom}_{R}(X,Y)$ such that the following diagram
$$\small\xymatrix{
      & M\otimes_{R}X \ar[d]_{f} \ar[r]^{M\otimes_{R}\gamma} & M\otimes_{R}Y \ar[d]_{g}  & \\
      & X  \ar[r]^{\gamma} &  Y  &  \\   }\vspace*{2mm}$$
is commutative, that is, $g(M\otimes_{R}\gamma)=\gamma f$.
A sequence
$$0\rightarrow(X_{1},f_{1})\overset{\gamma_{1}
}\rightarrow(X_{2},f_{2})\overset{\gamma_{2}}
\rightarrow(X_{3},f_{3})\rightarrow0$$
in $\Xi$ is exact if and only if the sequence $0\rightarrow X_{1}\overset{\gamma_{1}}\rightarrow X_{2}\overset{\gamma_{2}}\rightarrow X_{3}\rightarrow0$ in $R$-Mod is exact.
Dually, the category $R\ltimes M$-Mod is also isomorphic to the category $\Upsilon$, whose objects are pairs $[X,\alpha]$ with $X$ a left $R$-module and $\alpha\in\mathrm{Hom}_{R}(X,\mathrm{Hom}_{R}(M,X))$ such that the composition $X\overset{\alpha}\longrightarrow\mathrm{Hom}_{R}(M,X)\overset{\mathrm{Hom}_{R}(M,\alpha)}
\longrightarrow\mathrm{Hom}_{R}(M,\mathrm{Hom}_{R}(M,X))$ is 0, and morphisms $[X,\alpha]\rightarrow
[Y,\beta]$ are morphisms $\varphi$ with $\varphi\in\mathrm{Hom}_{R}(X,Y)$ such that the following diagram
$$\small\xymatrix@C=2cm@R=0.9cm{
      & X \ar[d]_{\alpha} \ar[r]^{\varphi} & Y \ar[d]_{\beta}  & \\
      & \mathrm{Hom}_{R}(M,X)  \ar[r]^{\mathrm{Hom}_{R}(M,\varphi)} &  \mathrm{Hom}_{R}(M,Y)  &  \\   }\vspace*{2mm}$$
is commutative, that is, $\beta\varphi=\mathrm{Hom}_{R}(M,\varphi)\alpha$.
A sequence
$$0\rightarrow[X_{1},\alpha_{1}]\overset{\varphi_{1}}\rightarrow[X_{2},\alpha_{2}]
\overset{\varphi_{2}}\rightarrow[X_{3},\alpha_{3}]\rightarrow0$$
in $\Upsilon$ is exact if and only if the sequence $0\rightarrow X_{1}\overset{\varphi_{1}}\rightarrow X_{2}\overset{\varphi_{2}}\rightarrow X_{3}\rightarrow0$ in $R$-Mod is exact.
We will identify the category $R\ltimes M\text{-}\mathrm{Mod}$ with the category $\Xi$ and the category $\Upsilon$.

\begin{rem}\label{prop:2.4}{\rm{For the category $\Xi$ and the category $\Upsilon$, we give the following functors.

$(1)$ The functor $\mathbf{T}:R$-Mod~$\rightarrow\Xi$ is defined by $\mathbf{T}(X)=(X\oplus(M\otimes_{R}X),\mu)$ with $$\mu=\left(\begin{smallmatrix}0&0\\1&0\end{smallmatrix}\right):(M\otimes_{R}X)\oplus(M\otimes_{R}M\otimes_{R}X)
\rightarrow X\oplus(M\otimes_{R}X)$$
for every object $X\in R\text{-}\mathrm{Mod}$,
and $\mathbf{T}(f)=\left(\begin{smallmatrix}f&0\\0&M\otimes_{R}f\end{smallmatrix}\right)$
for every morphism $f:X\rightarrow Y$ in $R$-Mod.

$(2)$ The functor $\mathbf{U}:\Xi\rightarrow R\text{-}\mathrm{Mod}$ is defined by $\mathbf{U}(X,f)=X$ for every object $(X, f)\in\Xi$, and $\mathbf{U}(\alpha)=\alpha$ for every morphism $\alpha:(X,f)\rightarrow(Y,g)$ in $\Xi$.

$(3)$ The functor $\mathbf{Z}:R$-Mod~$\rightarrow\Xi$ is defined by $\mathbf{Z}(X)=(X,0)$
for every object $X\in R\text{-}\mathrm{Mod}$, and $\mathbf{Z}(\alpha)=\alpha$ for every morphism $\alpha:X\rightarrow Y$ in $R$-Mod.

$(4)$ The functor $\mathbf{H}:R$-Mod~$\rightarrow\Upsilon$ is defined by $\mathbf{H}(X)=[\mathrm{Hom}_{R}(M,X)\oplus X,\vartheta]$ with
$$\vartheta=\left(\begin{smallmatrix}0&0\\1&0\end{smallmatrix}\right):\mathrm{Hom}_{R}(M,X)\oplus X\rightarrow
\mathrm{Hom}_{R}(M,\mathrm{Hom}_{R}(M,X))\oplus\mathrm{Hom}_{R}(M,X)$$
for every object $X\in R\text{-}\mathrm{Mod}$,
and $\mathbf{H}(\beta)=\left(\begin{smallmatrix}\mathrm{Hom}_{R}(M,\beta)
&0\\0&\beta\end{smallmatrix}\right)$ for every morphism $\beta:X\rightarrow Y$ in $R$-Mod.

$(5)$ The functor $\mathbf{U}':\Upsilon\rightarrow R$-Mod is defined by $\mathbf{U}'[X,\beta]=X$ for every object $[X,\beta]\in\Upsilon$, and $\mathbf{U}'(\alpha)=\alpha$ for every morphism $\alpha:[X,\beta]\rightarrow[Y,\gamma]$ in $\Upsilon$.

$(6)$ The functor $\mathbf{Z}':R$-Mod~$\rightarrow\Upsilon$ is defined by $\mathbf{Z}'(X)=[X,0]$ for every object $X\in R\text{-}\mathrm{Mod}$, and $\mathbf{Z}'(\alpha)=\alpha$ for every morphism $\alpha:X\rightarrow Y$ in $R$-Mod.
}}
\end{rem}

Note that the functor $\mathbf{T}$ is right exact, the functor $\mathbf{H}$ is left exact, and the functors $\mathbf{U}$, $\mathbf{U}'$, $\mathbf{Z}$ and $\mathbf{Z}'$ are exact. By \cite[\text{Proposition}~1.3]{RPI1975}, $(\mathbf{T},\mathbf{U})$ and $(\mathbf{U}', \mathbf{H})$ are adjoint pairs with $\mathbf{UZ}=\mathrm{id}_{R\text{-}\mathrm{Mod}}$.

\section{$n$-Cotorsion pairs over Comma categories}

The goal of this section is to construct left $n$-cotorsion pairs and right $n$-cotorsion pairs in comma categories by improving some isomorphisms between homology groups and combining some classes of objects in comma categories.

Throughout this section, let abelian categories $\mathcal{A}$ and $\mathcal{B}$ have enough projective and injective objects, $\mathcal{X}$ and $\mathcal{M}$ be classes of objects in $\mathcal{A}$, and $\mathcal{Y}$ and $\mathcal{N}$ be classes of objects in $\mathcal{B}$. Now, we define two special classes in the comma categories $(T\downarrow\mathcal{A})$ and $(\mathcal{B}\downarrow G)$, respectively.

(1)~$\left(\begin{smallmatrix}\mathcal{X}\\ \mathcal{Y}\end{smallmatrix}\right):=\{\left(\begin{smallmatrix}X\\Y\end{smallmatrix}\right)_{\varphi}\in
(T\downarrow\mathcal{A})\mid X\in\mathcal{X},~Y\in\mathcal{Y}\}$.

(2)~$\left(\begin{smallmatrix}\mathcal{M}\\ \mathcal{N}\end{smallmatrix}\right):=\{\left(\begin{smallmatrix}M\\N
\end{smallmatrix}\right)_{\phi}\in(\mathcal{B}\downarrow G)\mid M\in\mathcal{M},~N\in\mathcal{N}\}$.

(3)~$\mathfrak{B}_{\mathcal{Y}}^{\mathcal{X}}:=\{\left(\begin{smallmatrix}X\\Y\end{smallmatrix}\right)_{\varphi}\in
(T\downarrow\mathcal{A})\mid Y\in\mathcal{Y},~\varphi~\text{is a monomorphism and}~\mathrm{Coker}\varphi\in\mathcal{X}\}$.

(4)~$\mathfrak{D}_{\mathcal{N}}^{\mathcal{M}}:=\{\left(\begin{smallmatrix}M\\N\end{smallmatrix}\right)_{\phi}\in
(\mathcal{B}\downarrow G)\mid M\in\mathcal{M},~\phi~\text{is an epimorphism and}~\mathrm{Ker}\phi\in\mathcal{N}\}$.

{\bf 3.1 Approximation theory}

In this subsection, we discuss the connections between precovers and preenvelopes in the abelian categories $\mathcal{A}$, $\mathcal{B}$ and those in the comma categories $(T\downarrow\mathcal{A})$ and $(\mathcal{B}\downarrow G)$, respectively. The first three conclusions are an independent interest. The fourth conclusion is a simple generalization of the existing result.

Let $\mathcal{C}$ be a class of objects of $\mathcal{A}$. A morphism $f:C\rightarrow A$ is said to be a $\mathcal{C}$-precover (or a right $\mathcal{C}$-approximation) of $A$ if $C\in\mathcal{C}$ and if for any morphism $g:C'\rightarrow A$ with $C'\in\mathcal{C}$, there exists a morphism $h:C'\rightarrow C$ such that the following diagram
$$\xymatrix{
                &         C'  \ar[ld]_{h}\ar[d]^{g}     \\
  C \ar[r]_{f} & A }$$
is commutative, that is, $g=fh$. If, in addition, $f:C\rightarrow A$ is an epimorphism and $\mathrm{Ker}f\in\mathcal{C}^{\perp}$, then $f$ is called a special $\mathcal{C}$-precover of $A$. The class $\mathcal{C}$ is said to be
(special) precovering if every object of $\mathcal{A}$ has a (special) $\mathcal{C}$-precover. Dually, we have the notions of $\mathcal{C}$-preenvelopes (or left $\mathcal{C}$-approximations), special $\mathcal{C}$-preenvelopes and (special) preenveloping classes.

\begin{prop}\label{prop:2.4}{\it{$(1)$ If $\left(\begin{smallmatrix}f_1\\f_2\end{smallmatrix}\right):\left(\begin{smallmatrix}X\\Y\end{smallmatrix}
\right)_{\alpha}\rightarrow\left(\begin{smallmatrix}A\\B\end{smallmatrix}\right)_{\varphi}$ is an $\left(\begin
{smallmatrix}\mathcal{X}\\ \mathcal{Y}\end{smallmatrix}\right)$-precover of $\left(\begin{smallmatrix}A\\B\end{smallmatrix}\right)_{\varphi}$, then $f_{1}:X\rightarrow A$ is an $\mathcal{X}$-precover of $A$, and $f_{2}:Y\rightarrow B$ is a $\mathcal{Y}$-precover of $B$ in the case $T(\mathcal{Y})\subseteq
\mathcal{X}$.

$(2)$ If $\left(\begin{smallmatrix}f_1\\f_2\end{smallmatrix}\right):\left(\begin{smallmatrix}X\\Y\end{smallmatrix}
\right)_{\alpha}\rightarrow\left(\begin{smallmatrix}A\\B\end{smallmatrix}\right)_{\varphi}$ is a $\mathfrak{B}
_{\mathcal{Y}}^{\mathcal{X}}$-precover of $\left(\begin{smallmatrix}A\\B\end{smallmatrix}\right)_{\varphi}$, then $f_{2}:Y\rightarrow B$ is a $\mathcal{Y}$-precover of $B$, and $f_{1}:X\rightarrow A$ is an $\mathcal{X}$-precover of $A$ in the case $X\in\mathcal{X}$.
}}
\end{prop}
\begin{proof} $(1)$ Let $f:C\rightarrow A$ be any morphism in $\mathcal{A}$ with $C\in\mathcal{X}$. Then we get the morphism $\left(\begin{smallmatrix}f\\0\end{smallmatrix}\right):
\left(\begin{smallmatrix}C\\0\end{smallmatrix}\right)_{0}
\rightarrow\left(\begin{smallmatrix}A\\B\end{smallmatrix}\right)_{\varphi}$ in $(T\downarrow\mathcal{A})$ with $\left(\begin{smallmatrix}C\\0\end{smallmatrix}\right)_{0}\in\left(\begin
{smallmatrix}\mathcal{X}\\ \mathcal{Y}\end{smallmatrix}\right)$. There exists a morphism $\left(\begin{smallmatrix}g\\0\end{smallmatrix}\right):\left(\begin{smallmatrix}C\\0\end{smallmatrix}
\right)_{0}\rightarrow\left(\begin{smallmatrix}X\\Y\end{smallmatrix}\right)_{\alpha}$ such that $\left(
\begin{smallmatrix}f_1\\f_2\end{smallmatrix}\right)\left(\begin{smallmatrix}g\\0\end{smallmatrix}\right)=
\left(\begin{smallmatrix}f\\0\end{smallmatrix}\right)$. Hence $f_{1}g=f$. So $f_{1}:X\rightarrow A$ is an $\mathcal{X}$-precover of $A$ in $\mathcal{A}$.

Let $h:D\rightarrow B$ be any morphism in $\mathcal{B}$ with $D\in\mathcal{Y}$. Then we get the morphism $\left(
\begin{smallmatrix}\varphi T(h)\\h\end{smallmatrix}\right):\left(\begin{smallmatrix}T(D)\\D\end{smallmatrix}
\right)_{1}\rightarrow\left(\begin{smallmatrix}A\\B\end{smallmatrix}\right)_{\varphi}$ in $(T\downarrow
\mathcal{A})$. Since $T(\mathcal{Y})\subseteq
\mathcal{X}$, $\left(\begin{smallmatrix}T(D)\\D\end{smallmatrix}\right)_{1}\in\left(\begin
{smallmatrix}\mathcal{X}\\ \mathcal{Y}\end{smallmatrix}\right)$. There exists a morphism $\left(\begin{smallmatrix}k
\\l\end{smallmatrix}\right):\left(\begin{smallmatrix}T(D)\\D\end{smallmatrix}\right)_{1}\rightarrow
\left(\begin{smallmatrix}X\\Y\end{smallmatrix}\right)_{\alpha}$ such that $\left(\begin{smallmatrix}f_1\\f_2
\end{smallmatrix}\right)\left(\begin{smallmatrix}k\\l\end{smallmatrix}\right)=\left(\begin{smallmatrix}
\varphi T(h)\\h\end{smallmatrix}\right)$. Thus $f_{2}l=h$. So $f_{2}:Y\rightarrow B$ is a $\mathcal{Y}$-precover of $B$ in $\mathcal{B}$.

$(2)$ Let $b:F\rightarrow B$ be any morphism in $\mathcal{B}$ with $F\in\mathcal{Y}$. Then we get the morphism $\left(\begin{smallmatrix}\varphi T(b)\\b\end{smallmatrix}\right):\left(\begin{smallmatrix}T(F)\\F
\end{smallmatrix}\right)_{1}\rightarrow\left(\begin{smallmatrix}A\\B\end{smallmatrix}\right)_{\varphi}$ in $(T\downarrow\mathcal{A})$. Since $\left(\begin{smallmatrix}T(F)\\F\end{smallmatrix}\right)_{1}\in\mathfrak{B}
_{\mathcal{Y}}^{\mathcal{X}}$, there exists a morphism $\left(\begin{smallmatrix}c
\\d\end{smallmatrix}\right):\left(\begin{smallmatrix}T(F)\\F\end{smallmatrix}\right)_{1}\rightarrow
\left(\begin{smallmatrix}X\\Y\end{smallmatrix}\right)_{\alpha}$ such that $\left(\begin{smallmatrix}f_1\\f_2
\end{smallmatrix}\right)\left(\begin{smallmatrix}c\\d\end{smallmatrix}\right)=\left(\begin{smallmatrix}
\varphi T(b)\\b\end{smallmatrix}\right)$. Thus $f_{2}d=b$. So $f_{2}:Y\rightarrow B$ is a $\mathcal{Y}$-precover of $B$ in $\mathcal{B}$.

Let $a:E\rightarrow A$ be any morphism in $\mathcal{A}$ with $E\in\mathcal{X}$. Then we get the morphism $\left(\begin{smallmatrix}a\\0\end{smallmatrix}\right):
\left(\begin{smallmatrix}E\\0\end{smallmatrix}\right)_{0}
\rightarrow\left(\begin{smallmatrix}A\\B\end{smallmatrix}\right)_{\varphi}$ in $(T\downarrow\mathcal{A})$ with $\left(\begin{smallmatrix}E\\0\end{smallmatrix}\right)_{0}\in\mathfrak{B}
_{\mathcal{Y}}^{\mathcal{X}}$. There exists a morphism $\left(\begin{smallmatrix}e\\0\end{smallmatrix}\right):\left(\begin{smallmatrix}E\\0\end{smallmatrix}
\right)_{0}\rightarrow\left(\begin{smallmatrix}X\\Y\end{smallmatrix}\right)_{\alpha}$ such that $\left(
\begin{smallmatrix}f_1\\f_2\end{smallmatrix}\right)\left(\begin{smallmatrix}e\\0\end{smallmatrix}\right)=
\left(\begin{smallmatrix}a\\0\end{smallmatrix}\right)$. Hence $f_{1}e=a$. So $f_{1}:X\rightarrow A$ is an $\mathcal{X}$-precover of $A$ in $\mathcal{A}$.
\end{proof}

Dually, we have the following lemma.

\begin{prop}\label{prop:2.4}{\it{
$(1)$ If $\left(\begin{smallmatrix}g_1\\g_2\end{smallmatrix}\right):\left(\begin{smallmatrix}A\\B
\end{smallmatrix}\right)_{\varphi}\rightarrow\left(\begin{smallmatrix}M\\N\end{smallmatrix}\right)_{\phi}$ is an $\left(\begin{smallmatrix}\mathcal{M}\\ \mathcal{N}\end{smallmatrix}
\right)$-preenvelope of $\left(\begin{smallmatrix}A\\B
\end{smallmatrix}\right)_{\varphi}$, then $g_{2}:B\rightarrow N$ is an $\mathcal{N}$-preenvelope of $B$, and $g_{1}:A\rightarrow M$ is an $\mathcal{M}$-preenvelope of $A$ in case the $G(\mathcal{M})\subseteq\mathcal{N}$.

$(2)$ If $\left(\begin{smallmatrix}g_1\\g_2\end{smallmatrix}\right):\left(\begin{smallmatrix}A\\B
\end{smallmatrix}\right)_{\varphi}\rightarrow\left(\begin{smallmatrix}M\\N\end{smallmatrix}\right)_{\phi}$ is a $\mathfrak{D}_{\mathcal{N}}^{\mathcal{M}}$-preenvelope of $\left(\begin{smallmatrix}A\\B
\end{smallmatrix}\right)_{\varphi}$, then $g_{1}:A\rightarrow M$ is an $\mathcal{M}$-preenvelope of $A$, and $g_{2}:B\rightarrow N$ is an $\mathcal{N}$-preenvelope of $B$ in the case $N\in\mathcal{N}$.
}}
\end{prop}

\begin{prop}\label{prop:2.4}{\it{
$(1)$ $\left(\begin{smallmatrix}\mathcal{M}\\ \mathcal{N}\end{smallmatrix}\right)$ is a precovering class in $(\mathcal{B}\downarrow G)$ if and only if $\mathcal{M}$ is a precovering class in $\mathcal{A}$ and $\mathcal{N}$ is a precovering class in $\mathcal{B}$.

$(2)$ $\left(\begin{smallmatrix}\mathcal{X}\\ \mathcal{Y}\end{smallmatrix}\right)$ is a preenveloping class in $(T\downarrow\mathcal{A})$ if and only if $\mathcal{X}$ is a preenveloping class in $\mathcal{A}$ and $\mathcal{Y}$ is a preenveloping class in $\mathcal{B}$.
}}
\end{prop}
\begin{proof} The proof of $(1)$ is given below, and the proof of $(2)$ is dual.

``~$\Rightarrow$~'' Let $A$ be any object in $\mathcal{A}$, $\left(\begin{smallmatrix}f_1\\0
\end{smallmatrix}\right):\left(\begin{smallmatrix}E\\F\end{smallmatrix}\right)_{\varphi}
\rightarrow
\left(\begin{smallmatrix}A\\0\end{smallmatrix}\right)_{0}$ be an $\left(\begin{smallmatrix}\mathcal{M}\\ \mathcal{N}\end{smallmatrix}\right)$-precover of $\left(\begin{smallmatrix}A\\0\end{smallmatrix}\right)_{0}$ in $(\mathcal{B}\downarrow G)$, and $f:H\rightarrow A$ be any morphism in $\mathcal{A}$ with $H\in\mathcal{M}$. Then we get the morphism $\left(\begin{smallmatrix}f\\0
\end{smallmatrix}\right):\left(\begin{smallmatrix}H\\0\end{smallmatrix}\right)_{0}\rightarrow
\left(\begin{smallmatrix}A\\0\end{smallmatrix}\right)_{0}$ in $(\mathcal{B}\downarrow G)$ with $\left(\begin{smallmatrix}H\\0\end{smallmatrix}\right)_{0}\in
\left(\begin{smallmatrix}\mathcal{M}\\ \mathcal{N}\end{smallmatrix}\right)$. So there exists a morphism $\left(\begin{smallmatrix}h\\0\end{smallmatrix}\right):\left(\begin{smallmatrix}H\\0\end{smallmatrix}\right)
_{0}\rightarrow\left(\begin{smallmatrix}E\\F\end{smallmatrix}\right)_{\varphi}$ such that $\left(\begin
{smallmatrix}f_1\\0\end{smallmatrix}\right)\left(\begin{smallmatrix}h\\0\end{smallmatrix}\right)=
\left(\begin{smallmatrix}f\\0\end{smallmatrix}\right)$. Thus $f_{1}h=f$. Then $f_{1}:E\rightarrow A$ is an $\mathcal{M}$-precover of $A$ in $\mathcal{A}$, and hence $\mathcal{M}$ is a precovering class in $\mathcal{A}$.

Let $B$ be any object in $\mathcal{B}$, $\left(\begin{smallmatrix}0\\f_{2}
\end{smallmatrix}\right):\left(\begin{smallmatrix}C\\D\end{smallmatrix}\right)_{\phi}
\rightarrow
\left(\begin{smallmatrix}0\\B\end{smallmatrix}\right)_{0}$ be an $\left(\begin{smallmatrix}\mathcal{M}\\ \mathcal{N}\end{smallmatrix}\right)$-precover of $\left(\begin{smallmatrix}0\\B\end{smallmatrix}\right)_{0}$ in $(\mathcal{B}\downarrow G)$, and $\alpha:J\rightarrow B$ be any morphism in $\mathcal{B}$ with $J\in\mathcal{N}$. Then we get the morphism $\left(\begin{smallmatrix}0\\ \alpha
\end{smallmatrix}\right):\left(\begin{smallmatrix}0\\J\end{smallmatrix}\right)_{0}\rightarrow
\left(\begin{smallmatrix}0\\B\end{smallmatrix}\right)_{0}$ in $(\mathcal{B}\downarrow G)$ with $\left(\begin{smallmatrix}0\\J\end{smallmatrix}\right)_{0}\in\left(\begin{smallmatrix}\mathcal{M}\\ \mathcal{N}\end{smallmatrix}\right)$. So there is a morphism $\left(
\begin{smallmatrix}0\\ \eta\end{smallmatrix}\right):\left(\begin{smallmatrix}0\\J\end{smallmatrix}\right)
_{0}\rightarrow\left(\begin{smallmatrix}C\\D\end{smallmatrix}\right)_{\phi}$ such that $\left(\begin{smallmatrix}
0\\f_{2}\end{smallmatrix}\right)\left(\begin{smallmatrix}0\\ \eta\end{smallmatrix}\right)=\left(
\begin{smallmatrix}0\\ \alpha\end{smallmatrix}\right)$. Thus $f_{2}\eta=\alpha$. Then $f_{2}:D\rightarrow B$ is an $\mathcal{N}$-precover of $B$ in $\mathcal{B}$, and hence $\mathcal{N}$ is a precovering class in $\mathcal{B}$.

~``~$\Leftarrow$~'' Let $\left(\begin{smallmatrix}A\\B\end{smallmatrix}\right)_{\psi}$ be any object in $(\mathcal{B}\downarrow G)$ and $f_{1}:E\rightarrow A$ be an $\mathcal{M}$-precover of $A$ in $\mathcal{A}$. Then we obtain the following pullback:
$$\small\xymatrix{
      & H \ar[d]_{\beta} \ar[r]^{\alpha} & B \ar[d]^{\psi}  & \\
      & G(E)  \ar[r]_{G(f_{1})} &  G(A).  &  \\   }\vspace*{2mm}$$
Let $\gamma:K\rightarrow H$ be an $\mathcal{N}$-precover of $H$ in $\mathcal{B}$. We obtain the following commutative diagram:
$$\small\xymatrix{
      & K \ar[d]_{\beta\gamma} \ar[r]^{\alpha\gamma} & B \ar[d]^{\psi}  & \\
      & G(E)  \ar[r]_{G(f_{1})} &  G(A).  &  \\   }\vspace*{2mm}$$
Next, we prove that $\left(\begin{smallmatrix}f_{1}\\ \alpha\gamma\end{smallmatrix}\right):\left(
\begin{smallmatrix}E\\K\end{smallmatrix}\right)_{\beta\gamma}\rightarrow
\left(\begin{smallmatrix}
A\\B\end{smallmatrix}\right)_{\psi}$ is an $\left(\begin{smallmatrix}\mathcal{M}\\ \mathcal{N}\end{smallmatrix}\right)$-precover of $\left(\begin{smallmatrix}
A\\B\end{smallmatrix}\right)_{\psi}$ in $(\mathcal{B}\downarrow G)$. Let $\left(\begin{smallmatrix}
\xi\\ \eta\end{smallmatrix}\right):\left(\begin{smallmatrix}I\\J\end{smallmatrix}\right)_{\phi}
\rightarrow\left(\begin{smallmatrix}A\\B\end{smallmatrix}\right)_{\psi}$ be any morphism in $(\mathcal{B}\downarrow G)$ with $\left(\begin{smallmatrix}
I\\J\end{smallmatrix}\right)_{\phi}\in\left(\begin{smallmatrix}\mathcal{M}\\ \mathcal{N}\end{smallmatrix}
\right)$. There exists a morphism $h:I\rightarrow E$ such that $f_{1}h=
\xi$. Note that
$$\psi\eta=G(\xi)\phi=G(f_{1})G(h)\phi.$$
By the universal property of pullback, there exists a morphism $\theta:J\rightarrow H$ such that $\alpha\theta=\eta$ and $G(h)\phi=\beta\theta$. So there exists a morphism $\delta:
J\rightarrow K$ such that $\gamma\delta=\theta$. Thus we get the following commutative diagram:
$$\small\xymatrix{
      & J \ar[d]_{\phi} \ar[r]^{\delta} & K \ar[d]^{\beta\gamma}  & \\
      & G(I)  \ar[r]_{G(h)} &  G(E).  &  \\   }\vspace*{2mm}$$
Note that $\left(\begin{smallmatrix}f_{1}\\ \alpha\gamma\end{smallmatrix}\right)\left(\begin{smallmatrix}h\\ \delta\end{smallmatrix}\right)=\left(\begin{smallmatrix}\xi\\ \eta\end{smallmatrix}\right)$. Then $\left(
\begin{smallmatrix}f_{1}\\ \alpha\gamma\end{smallmatrix}\right):\left(\begin{smallmatrix}E\\K
\end{smallmatrix}\right)_{\beta\gamma}\rightarrow
\left(\begin{smallmatrix}A\\B\end{smallmatrix}\right)
_{\psi}$ is an $\left(\begin{smallmatrix}\mathcal{M}\\ \mathcal{N}\end{smallmatrix}\right)$-precover of $\left(\begin{smallmatrix}A\\B\end{smallmatrix}\right)
_{\psi}$ in $(\mathcal{B}\downarrow G)$, and hence $\left(\begin{smallmatrix}\mathcal{M}\\ \mathcal{N}\end{smallmatrix}\right)$ is a precovering class in $(\mathcal{B}\downarrow G)$.
\end{proof}

The following proposition is the version of \cite[\text{Lemma}~5.4]{L2020} in abelian categories, and it will play a crucial role in the proof of the main result. Since the proof of this proposition is similar to that of \cite[\text{Lemma}~5.4]{L2020}, we do not prove it.

\begin{prop}\label{prop:2.4}{\it{Let $0\rightarrow A_1\rightarrow A_2\rightarrow A_3\rightarrow0$ be a short exact sequence in $\mathcal{A}$.

$(1)$ Suppose that $\mathcal{H}$ is an injectively coresolving class in $\mathcal{A}$. If $A_1$ and $A_3$ admit special $\mathcal{H}$-preenvelopes, then $A_2$ admits a special $\mathcal{H}$-preenvelope. Moreover, there exists a commutative diagram with all rows and columns exact:
$$\small\xymatrix{
                   &0 \ar[d]_{}                       &0 \ar[d]_{}                         &0 \ar[d]_{}                 &  \\
  0  \ar[r]^{}     &A_1 \ar[d]_{} \ar[r]^{}           &A_2 \ar[d]_{} \ar[r]^{}             &A_3 \ar[d]_{} \ar[r]^{}        &0  \\
  0 \ar[r]^{}      &H_1 \ar[d]_{} \ar[r]^{}           &H_2 \ar[d]_{} \ar[r]^{}             &H_3 \ar[d]_{} \ar[r]^{}        &0  \\
  0  \ar[r]^{}     &B_1 \ar[d]_{} \ar[r]^{}           &B_2 \ar[d]_{} \ar[r]^{}              &B_3 \ar[d]_{} \ar[r]^{}       &0  \\
                   &0                                 &0                                   &0                           &
  &      \\   }\vspace*{2mm}$$
where $H_{i}\in\mathcal{H}$ and $B_{i}\in{^{\bot}\mathcal{H}}$ for $i=1,2,3$.

$(2)$ Suppose that $\mathcal{G}$ is a projectively resolving class in $\mathcal{A}$. If $A_1$ and $A_3$ admit special $\mathcal{G}$-precovers, then $A_2$ admits a special $\mathcal{G}$-precover. Moreover, there exists a commutative diagram with all rows and columns exact:
$$\small\xymatrix{
                   &0 \ar[d]_{}                       &0 \ar[d]_{}                         &0 \ar[d]_{}                 &  \\
  0  \ar[r]^{}     &E_1 \ar[d]_{} \ar[r]^{}           &E_2 \ar[d]_{} \ar[r]^{}             &E_3 \ar[d]_{} \ar[r]^{}        &0  \\
  0 \ar[r]^{}      &G_1 \ar[d]_{} \ar[r]^{}           &G_2 \ar[d]_{} \ar[r]^{}             &G_3 \ar[d]_{} \ar[r]^{}        &0  \\
  0  \ar[r]^{}     &A_1 \ar[d]_{} \ar[r]^{}           &A_2 \ar[d]_{} \ar[r]^{}             &A_3 \ar[d]_{} \ar[r]^{}       &0  \\
                   &0                                 &0                                   &0                           &
  &      \\   }\vspace*{2mm}$$
where $G_i\in\mathcal{G}$ and $E_i\in\mathcal{G}^{\bot}$ for $i=1,2,3$.
}}
\end{prop}

{\bf 3.2 Improvement of some isomorphisms between homology groups}

In this subsection, we give two fundamental and important lemmas, which not only refine the existing conclusions but also make a pivotal contribution to investigating the relationships among some classes of objects in comma categories.

Let $C:\mathcal{B}\rightarrow\mathcal{A}$ be a right exact covariant functor and $n\geqslant1$. For any object $B\in\mathcal{B}$, take the projective resolution
$$\cdots\rightarrow P_{n+1}\overset{d_{n+1}}\rightarrow P_{n}\overset{d_{n}}\rightarrow P_{n-1}\rightarrow\cdots\overset{d_{2}}\rightarrow P_{1}\overset{d_{1}}\rightarrow P_{0}\overset{\varepsilon}\rightarrow B\rightarrow0$$
 of $B$.
Applying the functor $C$ to the deleted projective resolution of $B$, one yields a complex
$$\cdots\rightarrow C(P_{n+1})\overset{C(d_{n+1})}\rightarrow C(P_{n})\overset{C(d_{n})}\rightarrow C(P_{n-1})\rightarrow\cdots\overset{C(d_{2})}\rightarrow C(P_{1})\overset{C(d_{1})}\rightarrow C(P_{0})\rightarrow0.$$
One defines the $n$-th cohomology group $(\mathrm{L}_{n}C)B=\mathrm{Ker}C(d_{n})/\mathrm{Im}C(d_{n+1})$.

Similarly, let $D:\mathcal{A}\rightarrow\mathcal{B}$ be a left exact covariant functor and $n\geqslant1$. For any object $A\in\mathcal{A}$, take the injective resolution
$$0\rightarrow A\overset{\upsilon}\rightarrow I^{0}\overset{d^{0}}\rightarrow I^{1}\overset{d^{1}}\rightarrow\cdots\rightarrow I^{n-1}\overset{d^{n-1}}\rightarrow I^{n}\overset{d^{n}}\rightarrow I^{n+1}\rightarrow\cdots$$
of $A$.
Applying the functor $D$ to the deleted injective resolution of $A$, one yields a complex
$$0\rightarrow D(I^{0})\overset{D(d^{0})}\rightarrow D(I^{1})\overset{D(d^{1})}\rightarrow\cdots\rightarrow D(I^{n-1})\overset{D(d^{n-1})}\rightarrow D(I^{n})\overset{D(d^{n})}\rightarrow D(I^{n+1})\rightarrow\cdots.$$
One defines the $n$-th cohomology group $(\mathrm{R}^{n}D)A=\mathrm{Ker}D(d^{n})/\mathrm{Im}D(d^{n-1})$.

In the comma categories $(T\downarrow\mathcal{A})$ and $(\mathcal{B}\downarrow G)$, we improve \cite[\text{Lemma}~3.1~(1)]{YDR2024} and \cite[\text{Lemma}~3.6~(1)]{YJDY2025} by means of dimension shifting and the adjoint pairs $(\mathbf{p},\mathbf{q})$ and $(\mathbf{g},\mathbf{h})$ in Remark 2.6 (4) and (5). Only the proof of Lemma 3.5 (2) is given below, and the proof of Lemma 3.6 (2) is similar.

\begin{lem}\label{prop:2.4}{\it{Let $n\geqslant1$ be an integer. Then the following statements hold.

$\mathrm{(1)}$ $\mathrm{Ext}_{(T\downarrow\mathcal{A})}^{n}(\left(\begin{smallmatrix}X\\0\end{smallmatrix}\right)_
{0},\left(\begin{smallmatrix}N_1\\N_2\end{smallmatrix}\right)_{\varphi})
\cong\mathrm{Ext}_{\mathcal{A}}^{n}(X,N_1)$.

$\mathrm{(2)}$ If $(\mathrm{L}_{j}T)Y=0$ for each $1\leqslant j\leqslant n$, then
$$\mathrm{Ext}_{(T\downarrow\mathcal{A})}^{j}(\left(\begin{smallmatrix}T(Y)\\Y\end{smallmatrix}\right)_{1},
\left(\begin{smallmatrix}X_1\\X_2\end{smallmatrix}\right)_{\psi})\cong\mathrm{Ext}_{\mathcal{B}}^{j}(Y,X_2).$$
}}
\end{lem}
\begin{proof} $\mathrm{(1)}$ This follows from \cite[\text{Lemma}~3.1~(2)]{YDR2024}.

$\mathrm{(2)}$ There is an exact sequence
$$0\rightarrow K_{n}\rightarrow P_{n-1}\rightarrow P_{n-2}\rightarrow\cdots\rightarrow P_{1}\rightarrow P_{0}\rightarrow Y\rightarrow0$$
in $\mathcal{B}$ with each $P_{i}$ projective object. Let $K_{n-1}=\mathrm{Ker}(P_{n-2}\rightarrow P_{n-3})$, $K_{1}=\mathrm{Ker}(P_{0}\rightarrow Y)$. Consider the short exact sequence $0\rightarrow K_{1}\rightarrow P_{0}\rightarrow Y\rightarrow0$. We get the induced exact sequence
$$(\mathrm{L}_{1}T)Y\rightarrow T(K_{1})\rightarrow T(P_{0})\rightarrow T(Y)\rightarrow0.$$
Since $(\mathrm{L}_{1}T)Y=0$, we obtain the short exact sequence $0\rightarrow T(K_{1})\rightarrow T(P_{0})\rightarrow T(Y)\rightarrow0$. Consider the short exact sequence $0\rightarrow K_{2}\rightarrow P_{1}\rightarrow K_{1}\rightarrow0$. We get the induced exact sequence
$$(\mathrm{L}_{1}T)K_{1}\rightarrow T(K_{2})\rightarrow T(P_{1})\rightarrow T(K_{1})\rightarrow0.$$
By dimension shifting, we have $(\mathrm{L}_{1}T)K_{1}\cong(\mathrm{L}_{2}T)Y=0$. So we get the short exact sequence
$$0\rightarrow T(K_{2})\rightarrow T(P_{1})\rightarrow T(K_{1})\rightarrow0.$$
Continuing this process, for the short exact sequence $0\rightarrow K_{n}\rightarrow P_{n-1}\rightarrow K_{n-1}\rightarrow0$, it induces an exact sequence
$$(\mathrm{L}_{1}T)K_{n-1}\rightarrow T(K_{n})\rightarrow T(P_{n-1})\rightarrow T(K_{n-1})\rightarrow0.$$
By dimension shifting, we have
$$(\mathrm{L}_{1}T)K_{n-1}\cong(\mathrm{L}_{2}T)K_{n-2}\cong(\mathrm{L}_{3}T)K_{n-3}\cong\cdots\cong
(\mathrm{L}_{n-1}T)K_{1}\cong(\mathrm{L}_{n}T)Y=0.$$
Thus, we obtain the short exact sequence $0\rightarrow T(K_{n})\rightarrow T(P_{n-1})\rightarrow T(K_{n-1})\rightarrow0$. Then there is an exact sequence
$$0\rightarrow T(K_{n})\rightarrow T(P_{n-1})\rightarrow T(P_{n-2})\rightarrow\cdots\rightarrow T(P_{1})\rightarrow T(P_{0})\rightarrow T(Y)\rightarrow0$$
in $\mathcal{A}$. Hence there is an exact sequence
$$0\rightarrow\textbf{p}(0,K_{n})\rightarrow\textbf{p}(0,P_{n-1})\rightarrow\textbf{p}(0,P_{n-2})\rightarrow\cdots
\rightarrow\textbf{p}(0,P_{1})\rightarrow\textbf{p}(0,P_{0})\rightarrow\textbf{p}(0,Y)\rightarrow0.$$
in $(T\downarrow\mathcal{A})$. Each $\textbf{p}(0,P_{i})$ is projective by Remark 2.6 (1). Let $X=\left(\begin{smallmatrix}X_1\\X_2\end{smallmatrix}\right)_{\psi}$ in $(T\downarrow\mathcal{A})$.
Note that
$$\mathrm{Hom}_{(T\downarrow\mathcal{A})}(\textbf{p}(0,P_{n-1}),
\left(\begin{smallmatrix}X_1\\X_2\end{smallmatrix}\right)_{\psi})\cong\mathrm{Hom}_{\mathcal{A}\times
\mathcal{B}}((0,P_{n-1}),(X_1,X_2)),$$
$$\mathrm{Hom}_{(T\downarrow\mathcal{A})}(\textbf{p}(0,K_{n}),\left
(\begin{smallmatrix}X_1\\X_2\end{smallmatrix}\right)_{\psi})\cong\mathrm{Hom}_{\mathcal{A}\times
\mathcal{B}}((0,K_{n}),(X_1,X_2)).$$
And, it is easy to prove that
$$\mathrm{Hom}_{\mathcal{A}\times
\mathcal{B}}((0,P_{n-1}),(X_1,X_2))\cong\mathrm{Hom}_{\mathcal{B}}(P_{n-1},X_2),$$
$$\mathrm{Hom}_{\mathcal{A}\times\mathcal{B}}((0,K_{n}),(X_1,X_2))\cong\mathrm{Hom}_{\mathcal{B}}(K_{n},X_2).$$
Then we have the following commutative diagram with exact rows:
$$\small\xymatrix{
& \mathrm{Hom}_{(T\downarrow\mathcal{A})}(\textbf{p}(0,P_{n-1}),X) \ar[d]_{\cong} \ar[r]^{} & \mathrm{Hom}_{(T\downarrow\mathcal{A})}(\textbf{p}(0,K_{n}),X) \ar[d]_{\cong}\ar[r]^{} & \mathrm{Ext}_{(T\downarrow\mathcal{A})}^{1}(\textbf{p}(0,K_{n-1}),X) \ar[d]_{}\ar[r]^{}&0\ar@{=}[d]  \\
& \mathrm{Hom}_{\mathcal{B}}(P_{n-1},X_2) \ar[r]^{} & \mathrm{Hom}_{\mathcal{B}}(K_{n},X_2) \ar[r]^{} & \mathrm{Ext}_{\mathcal{B}}^{1}(K_{n-1},X_2)\ar[r]^{} &0. &  \\   }\vspace*{2mm}$$
By five lemma, we obtain the following isomorphism
$$\mathrm{Ext}_{(T\downarrow\mathcal{A})}^{1}
(\textbf{p}(0,K_{n-1}),\left(\begin{smallmatrix}X_1\\X_2\end{smallmatrix}\right)_{\psi})\cong
\mathrm{Ext}_{\mathcal{B}}^{1}(K_{n-1},X_2).$$
Therefore,
$$
\aligned
\mathrm{Ext}_{(T\downarrow\mathcal{A})}
^{n}(\textbf{p}(0,Y),\left(\begin{smallmatrix}X_1\\X_2\end{smallmatrix}\right)_{\psi})&\cong\mathrm{Ext}_{
(T\downarrow\mathcal{A})}^{1}(\textbf{p}(0,K_{n-1}),\left(\begin{smallmatrix}X_1\\X_2\end{smallmatrix}\right)
_{\psi})\\
&\cong\mathrm{Ext}_{\mathcal{B}}^{1}(K_{n-1},X_2)\\
&\cong\mathrm{Ext}_{\mathcal{B}}^{n}(Y,X_2).
\endaligned
$$
Continuing this process, we have the following commutative diagram with exact rows:
$$\small\xymatrix{
& \mathrm{Hom}_{(T\downarrow\mathcal{A})}(\textbf{p}(0,P_{n-2}),X) \ar[d]_{\cong} \ar[r]^{} & \mathrm{Hom}_{(T\downarrow\mathcal{A})}(\textbf{p}(0,K_{n-1}),X) \ar[d]_{\cong}\ar[r]^{} & \mathrm{Ext}_{(T\downarrow\mathcal{A})}^{1}(\textbf{p}(0,K_{n-2}),X) \ar[d]_{}\ar[r]^{}&0\ar@{=}[d] \\
& \mathrm{Hom}_{\mathcal{B}}(P_{n-2},X_2) \ar[r]^{} & \mathrm{Hom}_{\mathcal{B}}(K_{n-1},X_2) \ar[r]^{} & \mathrm{Ext}_{\mathcal{B}}^{1}(K_{n-2},X_2)\ar[r]^{}&0. & \\   }\vspace*{2mm}$$
Then $\mathrm{Ext}_{(T\downarrow\mathcal{A})}^{1}(\textbf{p}(0,K_{n-2}),\left(\begin{smallmatrix}X_1\\X_2\end
{smallmatrix}\right)_{\psi})\cong\mathrm{Ext}_{\mathcal{B}}^{1}(K_{n-2},X_2)$, and hence
$$
\aligned
\mathrm{Ext}_{(T\downarrow\mathcal{A})}
^{n-1}(\textbf{p}(0,Y),\left(\begin{smallmatrix}X_1\\X_2\end{smallmatrix}\right)_{\psi})
&\cong\mathrm{Ext}_{(T\downarrow\mathcal{A})}^{1}(\textbf{p}(0,K_{n-2}),\left(\begin{smallmatrix}X_1\\X_2
\end{smallmatrix}\right)_{\psi})\\
&\cong\mathrm{Ext}_{\mathcal{B}}^{1}(K_{n-2},X_2)\\
&\cong\mathrm{Ext}_{\mathcal{B}}^{n-1}(Y,X_2).
\endaligned
$$
At last, we have the following commutative diagram with exact rows:
$$\small\xymatrix{
& \mathrm{Hom}_{(T\downarrow\mathcal{A})}(\textbf{p}(0,P_{0}),X) \ar[d]_{\cong} \ar[r]^{} & \mathrm{Hom}_{(T\downarrow\mathcal{A})}(\textbf{p}(0,K_{1}),X) \ar[d]_{\cong}\ar[r]^{} & \mathrm{Ext}_{(T\downarrow\mathcal{A})}^{1}(\textbf{p}(0,Y),X) \ar[d]_{}\ar[r]^{}&0\ar@{=}[d] \\
& \mathrm{Hom}_{\mathcal{B}}(P_{0},X_2) \ar[r]^{} & \mathrm{Hom}_{\mathcal{B}}(K_{1},X_2) \ar[r]^{} & \mathrm{Ext}_{\mathcal{B}}^{1}(Y,X_2)\ar[r]^{}&0. & \\   }\vspace*{2mm}$$
Then $\mathrm{Ext}_{(T\downarrow\mathcal{A})}^{1}(\textbf{p}(0,Y),\left(\begin{smallmatrix}X_1\\X_2\end
{smallmatrix}\right)_{\psi})\cong\mathrm{Ext}_{\mathcal{B}}^{1}(Y,X_2)$.
This completes the proof.
\end{proof}

\begin{lem}\label{prop:2.4}{\it{Let $n\geqslant1$ be an integer. Then the following statements hold.

$\mathrm{(1)}$ $\mathrm{Ext}_{(\mathcal{B}\downarrow G)}^{n}(\left(\begin{smallmatrix}M_1\\M_2
\end{smallmatrix}\right)_{\varphi},\left(\begin{smallmatrix}0\\Y
\end{smallmatrix}\right)_{0})\cong\mathrm{Ext}_{\mathcal{B}}^{n}(M_2,Y)$.

$\mathrm{(2)}$ If $(\mathrm{R}^{j}G)X=0$ for each $1\leqslant j\leqslant n$, then
$$\mathrm{Ext}_{(\mathcal{B}\downarrow G)}^{j}(\left(\begin{smallmatrix}Y_1\\Y_2\end{smallmatrix}\right)_{\psi},\left(\begin{smallmatrix}X\\G(X)\end{smallmatrix}
\right)_{1})\cong\mathrm{Ext}_{\mathcal{A}}^{j}(Y_1,X).$$
}}
\end{lem}

{\bf 3.3 Construct left (right) $n$-cotorsion pairs in comma categories}

In this subsection, via Lemma 3.5 and Lemma 3.6, we investigate the relationships among some classes of objects in comma categories $(T\downarrow\mathcal{A})$ and $(\mathcal{B}\downarrow G)$ associated with $\left(\begin{smallmatrix}\mathcal{X}\\ \mathcal{Y}\end{smallmatrix}\right)$, $\left(\begin{smallmatrix}\mathcal{M}\\ \mathcal{N}\end{smallmatrix}\right)$, $\mathfrak{B}_{\mathcal{Y}}^{\mathcal{X}}$ and $\mathfrak{D}_{\mathcal{N}}^{\mathcal{M}}$. We propose an approach to constructing left (resp. right) $n$-cotorsion pairs in the comma category $(\mathcal{B}\downarrow G)$ (resp. $(T\downarrow\mathcal{A})$) by means of left (resp. right) $n$-cotorsion pairs in abelian categories $\mathcal{A}$ and $\mathcal{B}$. At the same time, conversely, we also give an approach to constructing left (resp. right) $n$-cotorsion pairs in abelian categories $\mathcal{A}$ and $\mathcal{B}$ by virtue of left (resp. right) $n$-cotorsion pairs in the comma category $(\mathcal{B}\downarrow G)$ (resp. $(T\downarrow\mathcal{A})$).

\begin{lem}\label{prop:2.4}{\it{For any integer $n\geqslant1$, the following statements hold.

$\mathrm{(1)}$ If $(\mathrm{L}_{j}T)Y=0$ for each $1\leqslant j\leqslant n$ and any $Y\in\mathcal{Y}$, then
$\bigcap\limits_{i=1}\limits^{n}(\mathfrak{B}_{\mathcal{Y}}^{\mathcal{X}})^{\bot_{i}}=
\left(\begin{smallmatrix}\mathcal{X}^{\bot_{[1,n]}}\\ \mathcal{Y}^{\bot_{[1,n]}}\end{smallmatrix}\right)$.

$\mathrm{(2)}$ If $(\mathrm{R}^{j}G)M=0$ for each $1\leqslant j\leqslant n$ and any $M\in\mathcal{M}$, then $\bigcap
\limits_{i=1}\limits^{n}{^{\bot_{i}}(\mathfrak{D}_{\mathcal{N}}^{\mathcal{M}})}=
\left(\begin{smallmatrix}{^{\bot_{[1,n]}}\mathcal{M}}\\ {^{\bot_{[1,n]}}\mathcal{N}}\end{smallmatrix}\right)$.
}}
\end{lem}
\begin{proof} The proof of (1) is given below, and the proof of (2) is similar.

Let $\left(\begin{smallmatrix}A\\B\end{smallmatrix}\right)_{\varphi}\in\bigcap
\limits_{i=1}\limits^{n}(\mathfrak{B}_{\mathcal{Y}}^{\mathcal{X}})^{\bot_{i}}$,~$X\in\mathcal{X}$ and $Y\in
\mathcal{Y}$. Since $(\mathrm{L}_{j}T)Y=0$ for each $1\leqslant j\leqslant n$ and $Y\in\mathcal{Y}$, we have $\mathrm{Ext}_{(T\downarrow\mathcal{A})}^{j}
(\left(\begin{smallmatrix}T(Y)\\Y\end{smallmatrix}\right)_{1},\left(\begin{smallmatrix}A\\B\end{smallmatrix}
\right)_{\varphi})\cong\mathrm{Ext}_{\mathcal{B}}^{j}(Y,B)$ by Lemma 3.5 (2). Since $\left(\begin{smallmatrix}T(Y)\\Y\end{smallmatrix}\right)_{1}\in\mathfrak{B}_{\mathcal{Y}}^
{\mathcal{X}}$, we have $\mathrm{Ext}_{(T\downarrow\mathcal{A})}^{j}
(\left(\begin{smallmatrix}T(Y)\\Y\end{smallmatrix}\right)_{1},\left(\begin{smallmatrix}A\\B\end{smallmatrix}
\right)_{\varphi})=0$ for each $1\leqslant j\leqslant n$, which implies $\mathrm{Ext}_{\mathcal{B}}^{j}(Y,B)=0$. Thus $B\in\mathcal{Y}^{\bot_{[1,n]}}$. Meanwhile, since $\left(\begin{smallmatrix}X\\0\end{smallmatrix}\right)
_{0}\in\mathfrak{B}_{\mathcal{Y}}^{\mathcal{X}}$, we have $\mathrm{Ext}_{\mathcal{A}}^{j}(X,A)\cong
\mathrm{Ext}_{(T\downarrow\mathcal{A})}^{j}(\left(\begin{smallmatrix}X\\0\end{smallmatrix}\right)_{0},\left(
\begin{smallmatrix}A\\B\end{smallmatrix}\right)_{\varphi})=0$ for each $1\leqslant l\leqslant n$ by Lemma 3.5 (1). So $A\in\mathcal{X}^{\bot_{[1,n]}}$. Then $\left(\begin{smallmatrix}A\\B\end{smallmatrix}\right)
_{\varphi}\in\left(\begin{smallmatrix}\mathcal{X}^{\bot_{[1,n]}}\\ \mathcal{Y}^{\bot_{[1,n]}}
\end{smallmatrix}\right)$ and hence $\textstyle\bigcap\limits_{i=1}
\limits^{n}(\mathfrak{B}_{\mathcal{Y}}^{\mathcal{X}})^{\bot_{i}}\subseteq\left(\begin{smallmatrix}\mathcal{X}
^{\bot_{[1,n]}}\\ \mathcal{Y}^{\bot_{[1,n]}}\end{smallmatrix}\right)$.

 Let $\left(\begin{smallmatrix}C\\D\end{smallmatrix}\right)_{\psi}\in\left
(\begin{smallmatrix}\mathcal{X}^{\bot_{[1,n]}}\\ \mathcal{Y}^{\bot_{[1,n]}}\end{smallmatrix}\right)$. For any $\left(\begin{smallmatrix}C'\\D'\end{smallmatrix}\right)_{\psi^{'}}\in\mathfrak{B}_{\mathcal{Y}}
^{\mathcal{X}}$, there exists a short exact sequence
$$0\rightarrow\left(\begin{smallmatrix}T(D')\\D'\end{smallmatrix}\right)_{1}\rightarrow\left(
\begin{smallmatrix}C'\\D'\end{smallmatrix}\right)_{\psi^{'}}\rightarrow\left(\begin{smallmatrix}\mathrm{Coker}
\psi^{'}\\0\end{smallmatrix}\right)_{0}\rightarrow0$$
in $(T\downarrow\mathcal{A})$. Applying $\mathrm{Hom}_{(T\downarrow\mathcal{A})}(-,\left(
\begin{smallmatrix}C\\D\end{smallmatrix}\right)_{\psi})$ to the short exact sequence, for each $1\leqslant k\leqslant n$, we can obtain an exact sequence
$$\mathrm{Ext}_{(T\downarrow\mathcal{A})}^{k}(\left(\begin{smallmatrix}\mathrm{Coker}\psi^{'}\\0\end{smallmatrix}
\right),\left(\begin{smallmatrix}C\\D\end{smallmatrix}\right)_{\psi})\rightarrow\mathrm{Ext}_{(T\downarrow
\mathcal{A})}^{k}(\left(\begin{smallmatrix}C'\\D'\end{smallmatrix}\right)_{\psi^{'}},\left(\begin{smallmatrix}C\\D
\end{smallmatrix}\right)_{\psi})\rightarrow\mathrm{Ext}_{(T\downarrow
\mathcal{A})}^{k}(\left(\begin{smallmatrix}T(D')\\D'\end{smallmatrix}\right),\left(\begin{smallmatrix}C\\D
\end{smallmatrix}\right)_{\psi}).$$
Since $\left(\begin{smallmatrix}C'\\D'\end{smallmatrix}\right)_{\psi^{'}}
\in\mathfrak{B}_{\mathcal{Y}}^{\mathcal{X}}$, we have $D'\in\mathcal{Y}$ and $\mathrm{Coker}\psi^{'}\in
\mathcal{X}$. For each $1\leqslant k\leqslant n$, we get
$$\mathrm{Ext}_{(T\downarrow\mathcal{A})}^{k}(\left(\begin{smallmatrix}\mathrm{Coker}\psi^{'}\\0\end{smallmatrix}
\right)_{0},\left(\begin{smallmatrix}C\\D\end{smallmatrix}\right)_{\psi})\cong\mathrm{Ext}_{\mathcal{A}}
^{k}(\mathrm{Coker}\psi^{'},C)=0$$
by Lemma 3.5 (1). Since $\left(\begin{smallmatrix}T(D')\\D'\end{smallmatrix}\right)_{1}\in\mathfrak{B}_{\mathcal{Y}}^
{\mathcal{X}}$, for each $1\leqslant k\leqslant n$, we obtain
$$\mathrm{Ext}_{(T\downarrow\mathcal{A})}^{k}(\left(\begin{smallmatrix}T(D')\\D'\end{smallmatrix}\right)_{1},\left
(\begin{smallmatrix}C\\D\end{smallmatrix}\right)_{\psi})\cong\mathrm{Ext}_{\mathcal{B}}^{k}(D',D)=0$$
by Lemma 3.5 (2). Then, for each $1\leqslant k\leqslant n$, $\mathrm{Ext}_{(T\downarrow\mathcal{A})}^{k}(\left(\begin{smallmatrix}C'\\D'
\end{smallmatrix}\right)_{\psi^{'}},\left(\begin{smallmatrix}C\\D\end{smallmatrix}\right)_{\psi})=0$ and hence $\left(\begin{smallmatrix}C\\D\end{smallmatrix}
\right)_{\psi}\in\bigcap\limits_{i=1}\limits^{n}(\mathfrak{B}_{\mathcal{Y}}^{\mathcal{X}})
^{\bot_{i}}$. Therefore $\textstyle\bigcap\limits_{i=1}
\limits^{n}(\mathfrak{B}_{\mathcal{Y}}^{\mathcal{X}})^{\bot_{i}}\supseteq\left(\begin{smallmatrix}\mathcal{X}
^{\bot_{[1,n]}}\\ \mathcal{Y}^{\bot_{[1,n]}}\end{smallmatrix}\right)$.
This completes the proof.
\end{proof}

\begin{prop}\label{prop:2.4}{\it{
$\mathrm{(1)}$ Let $B\in\mathcal{B}$ and there exist an exact sequence
$$0\rightarrow B\overset{d_{0}}\rightarrow Y_{0}\overset{d_{1}}\rightarrow Y_{1}\rightarrow\cdots\rightarrow Y_{n-1}\overset{d_{n}}\rightarrow Y_{n}\rightarrow0$$
in $\mathcal{B}$ with each $Y_{i}\in\mathcal{Y}$. If $(\mathrm{L}_{1}T)Y=0$ for any $Y\in\mathcal{Y}_{n}^{\vee}$, then there exists an exact sequence
$$0\rightarrow T(B)\overset{T(d_{0})}\rightarrow T(Y_{0})\overset{T(d_{1})}\rightarrow T(Y_{1})\rightarrow\cdots\rightarrow T(Y_{n-1})\overset{T(d_{n})}\rightarrow T(Y_{n})\rightarrow0$$
in $\mathcal{A}$.

$\mathrm{(2)}$ Let $A\in\mathcal{A}$ and there exist an exact sequence
$$0\rightarrow X_{n}\overset{f_{n}}\rightarrow X_{n-1}\rightarrow\cdots\rightarrow X_{1}\overset{f_{1}}\rightarrow X_{0}\overset{f_{0}}\rightarrow A\rightarrow0$$
in $\mathcal{A}$ with each $X_{i}\in\mathcal{X}$. If $(\mathrm{R}^{1}G)X=0$ for any $X\in\mathcal{X}_{n}^{\wedge}$, then there exists an exact sequence
$$0\rightarrow G(X_{n})\overset{G(f_{n})}\rightarrow G(X_{n-1})\rightarrow\cdots\rightarrow G(X_{1})\overset{G(f_{1})}\rightarrow G(X_{0})\overset{G(f_{0})}\rightarrow G(A)\rightarrow0$$
in $\mathcal{B}$.
}}
\end{prop}
\begin{proof} The proof of (1) is given below, and the proof of (2) is similar.

The short exact sequence $0\rightarrow B\overset{d_{0}}\rightarrow Y_{0}\rightarrow\mathrm{Coker}d_{0}\rightarrow0$ induces an exact sequence
$$(\mathrm{L}_{1}T)\mathrm{Coker}d_{0}\rightarrow T(B)\rightarrow T(Y_{0})\rightarrow T(\mathrm{Coker}d_{0})\rightarrow0.$$
Clearly, $\mathrm{Coker}d_{i}\in\mathcal{Y}_{n}^{\vee}$ for each $0\leqslant i\leqslant n$. By assumption, $(\mathrm{L}_{1}T)\mathrm{Coker}d_{0}=0$. So the sequence
$$0\rightarrow T(B)\rightarrow T(Y_{0})\rightarrow T(\mathrm{Coker}d_{0})\rightarrow0$$
is exact. For $0\leqslant i\leqslant n-2$, the short exact sequence $0\rightarrow \mathrm{Coker}d_{i}\rightarrow Y_{i+1}\rightarrow\mathrm{Coker}d_{i+1}\rightarrow0$ induces the exact sequence
$$(\mathrm{L}_{1}T)\mathrm{Coker}d_{i+1}\rightarrow T(\mathrm{Coker}d_{i})\rightarrow T(Y_{i+1})\rightarrow T(\mathrm{Coker}d_{i+1})\rightarrow0.$$
By assumption, $(\mathrm{L}_{1}T)\mathrm{Coker}d_{i+1}=0$. So the sequence
$$0\rightarrow T(\mathrm{Coker}d_{i})\rightarrow T(Y_{i+1})\rightarrow T(\mathrm{Coker}d_{i+1})\rightarrow0$$
is exact. Then we have the exact sequence
$$0\rightarrow T(B)\overset{T(d_{0})}\rightarrow T(Y_{0})\overset{T(d_{1})}\rightarrow T(Y_{1})\rightarrow\cdots\rightarrow T(Y_{n-1})\overset{T(d_{n})}\rightarrow T(Y_{n})\rightarrow0$$
in $\mathcal{A}$.
\end{proof}

\begin{lem}\label{prop:2.4}{\it{The following statements hold.

$\mathrm{(1)}$ $\left(\begin{smallmatrix}\mathcal{X}\\ \mathcal{Y}\end{smallmatrix}\right)_{n}^{\vee}\subseteq\left
(\begin{smallmatrix}\mathcal{X}_{n}^{\vee}\\ \mathcal{Y}_{n}^{\vee}\end{smallmatrix}\right)$. In addition, if $\left(\begin{smallmatrix}\mathcal{X}\\ \mathcal{Y}\end{smallmatrix}\right)_{n}^{\vee}$ is closed under extensions, then $\left(\begin{smallmatrix}\mathcal{X}\\ \mathcal{Y}\end{smallmatrix}\right)_{n}^{\vee}=\left
(\begin{smallmatrix}\mathcal{X}_{n}^{\vee}\\ \mathcal{Y}_{n}^{\vee}\end{smallmatrix}\right)$.

$\mathrm{(2)}$ $\left(\begin{smallmatrix}\mathcal{M}\\ \mathcal{N}\end{smallmatrix}\right)_{n}^{\wedge}\subseteq\left
(\begin{smallmatrix}\mathcal{M}_{n}^{\wedge}\\ \mathcal{N}_{n}^{\wedge}\end{smallmatrix}\right)$. In addition, if $\left(\begin{smallmatrix}\mathcal{M}\\ \mathcal{N}\end{smallmatrix}\right)_{n}^{\wedge}$ is closed under extensions, then $\left(\begin{smallmatrix}\mathcal{M}\\ \mathcal{N}\end{smallmatrix}\right)_{n}^{\wedge}=\left
(\begin{smallmatrix}\mathcal{M}_{n}^{\wedge}\\ \mathcal{N}_{n}^{\wedge}\end{smallmatrix}\right)$.

$\mathrm{(3)}$ If $(\mathrm{L}_{1}T)Y=0$ for any $Y\in\mathcal{Y}_{n}^{\vee}$ and $(\mathfrak{B}_{\mathcal{Y}}^
{\mathcal{X}})_{n}^{\vee}$ is closed under extensions, then $(\mathfrak{B}_{\mathcal{Y}}^{\mathcal{X}})
_{n}^{\vee}=\mathfrak{B}_{\mathcal{Y}_{n}^{\vee}}^{\mathcal{X}_{n}^{\vee}}$.

$\mathrm{(4)}$ If $(\mathrm{R}^{1}G)M=0$ for any $M\in\mathcal{M}_{n}^{\wedge}$ and $(\mathfrak{D}_{\mathcal{N}}^
{\mathcal{M}})_{n}^{\wedge}$ is closed under extensions, then $(\mathfrak{D}_{\mathcal{N}}^{\mathcal{M}})
_{n}^{\wedge}=\mathfrak{D}_{\mathcal{N}_{n}^{\wedge}}^{\mathcal{M}_{n}^{\wedge}}$.
}}
\end{lem}
\begin{proof} We just prove (1) and (3) below. The proofs of (2) and (4) are similar.

$\mathrm{(1)}$ For any $\left(\begin{smallmatrix}A\\B\end{smallmatrix}\right)_{\varphi}
\in\left(\begin{smallmatrix}\mathcal{X}\\ \mathcal{Y}\end{smallmatrix}\right)_{n}^{\vee}$, there exists an exact sequence
$$0\rightarrow\left(\begin{smallmatrix}A\\B\end{smallmatrix}\right)_{\varphi}\rightarrow\left(\begin{smallmatrix}
X_{0}\\Y_{0}\end{smallmatrix}\right)_{\varphi_{0}}\rightarrow\left(\begin{smallmatrix}X_{1}\\Y_{1}\end{smallmatrix}
\right)_{\varphi_{1}}\rightarrow\cdots\rightarrow\left(\begin{smallmatrix}X_{n}\\Y_{n}\end{smallmatrix}
\right)_{\varphi_{n}}\rightarrow0$$
in $(T\downarrow\mathcal{A})$ with each $\left(\begin{smallmatrix}X_{i}\\Y_{i}\end{smallmatrix}\right)_{\varphi_{i}}
\in\left(\begin{smallmatrix}\mathcal{X}\\ \mathcal{Y}\end{smallmatrix}\right)$. Then we have the exact sequences $0\rightarrow A\rightarrow X_{0}\rightarrow X_{1}\rightarrow\cdots\rightarrow X_{n}\rightarrow0$ in $\mathcal{A}$ and $0\rightarrow B\rightarrow Y_{0}\rightarrow Y_{1}\rightarrow\cdots\rightarrow Y_{n}\rightarrow0$ in $\mathcal{B}$. So $A\in\mathcal{X}_{n}^{\vee}$ and $B\in
\mathcal{Y}_{n}^{\vee}$. Then $\left(\begin{smallmatrix}A\\B\end{smallmatrix}\right)_{\varphi}\in\left(\begin
{smallmatrix}\mathcal{X}_{n}^{\vee}\\ \mathcal{Y}_{n}^{\vee}\end{smallmatrix}\right)$ and hence $\left(
\begin{smallmatrix}\mathcal{X}\\ \mathcal{Y}\end{smallmatrix}\right)_{n}^{\vee}\subseteq\left
(\begin{smallmatrix}\mathcal{X}_{n}^{\vee}\\ \mathcal{Y}_{n}^{\vee}\end{smallmatrix}\right)$.

On the other hand, for any $\left(\begin{smallmatrix}C\\D\end{smallmatrix}\right)_{\psi}\in\left
(\begin{smallmatrix}\mathcal{X}_{n}^{\vee}\\ \mathcal{Y}_{n}^{\vee}\end{smallmatrix}\right)$, there exists an exact sequence
$$0\rightarrow\left(\begin{smallmatrix}C\\0\end{smallmatrix}\right)_{0}\rightarrow\left(\begin{smallmatrix}C\\D
\end{smallmatrix}\right)_{\psi}\rightarrow\left(\begin{smallmatrix}0\\D\end{smallmatrix}\right)_{0}\rightarrow0$$
in $(T\downarrow\mathcal{A})$. Since $C\in\mathcal{X}_{n}^{\vee}$, there exists an exact sequence $0\rightarrow C\rightarrow X_{0}\rightarrow X_{1}\rightarrow\cdots\rightarrow X_{n}\rightarrow0$ in $\mathcal{A}$ with each $X_{i}\in\mathcal{X}$. Further, we have the exact sequence
$$0\rightarrow\left(\begin{smallmatrix}C\\0\end{smallmatrix}\right)\rightarrow\left(\begin{smallmatrix}X_{0}\\0
\end{smallmatrix}\right)\rightarrow\left(\begin{smallmatrix}X_{1}\\0\end{smallmatrix}\right)\rightarrow\cdots
\rightarrow\left(\begin{smallmatrix}X_{n}\\0\end{smallmatrix}\right)\rightarrow0$$
in $(T\downarrow\mathcal{A})$ with each $\left(\begin{smallmatrix}X_{i}\\0\end{smallmatrix}\right)\in
\left(\begin{smallmatrix}\mathcal{X}\\ \mathcal{Y}\end{smallmatrix}\right)$. So $\left(\begin{smallmatrix}C\\0\end{smallmatrix}\right)\in\left
(\begin{smallmatrix}\mathcal{X}\\ \mathcal{Y}\end{smallmatrix}\right)_{n}^{\vee}$. Since $D\in\mathcal{Y}_{n}
^{\vee}$, there exists an exact sequence $0\rightarrow D\rightarrow Y_{0}\rightarrow Y_{1}\rightarrow\cdots\rightarrow Y_{n}\rightarrow0$ in $\mathcal{B}$ with each $Y_{i}\in\mathcal{Y}$. Further, we have the exact sequence
$$0\rightarrow\left(\begin{smallmatrix}0\\D\end{smallmatrix}\right)\rightarrow\left(\begin{smallmatrix}0\\Y_{0}
\end{smallmatrix}\right)\rightarrow\left(\begin{smallmatrix}0\\Y_{1}\end{smallmatrix}\right)\rightarrow\cdots
\rightarrow\left(\begin{smallmatrix}0\\Y_{n}\end{smallmatrix}\right)\rightarrow0$$
in $(T\downarrow\mathcal{A})$ with each $\left(\begin{smallmatrix}0\\Y_{i}\end{smallmatrix}\right)\in
\left(\begin{smallmatrix}\mathcal{X}\\ \mathcal{Y}\end{smallmatrix}\right)$. So $\left(\begin{smallmatrix}0\\D\end{smallmatrix}\right)
\in\left(\begin{smallmatrix}\mathcal{X}\\ \mathcal{Y}\end{smallmatrix}\right)_{n}^{\vee}$. Since $\left(\begin
{smallmatrix}\mathcal{X}\\ \mathcal{Y}\end{smallmatrix}\right)_{n}^{\vee}$ is closed under extensions, $\left
(\begin{smallmatrix}C\\D\end{smallmatrix}\right)_{\psi}\in\left(\begin{smallmatrix}\mathcal{X}\\ \mathcal{Y}\end{smallmatrix}\right)_{n}^{\vee}$. Hence $\left(\begin{smallmatrix}\mathcal{X}_{n}^{\vee}\\ \mathcal{Y}_{n}^{\vee}\end{smallmatrix}\right)\subseteq\left(\begin{smallmatrix}\mathcal{X}\\ \mathcal{Y}\end{smallmatrix}\right)_{n}^{\vee}$.

$\mathrm{(3)}$ For any $\left(\begin{smallmatrix}A\\B\end{smallmatrix}\right)_{\varphi}\in(\mathfrak{B}
_{\mathcal{Y}}^{\mathcal{X}})_{n}^{\vee}$, there exists an exact sequence
$$0\rightarrow\left(\begin{smallmatrix}A\\B\end{smallmatrix}\right)_{\varphi}\rightarrow\left(\begin{smallmatrix}
A_{0}\\B_{0}\end{smallmatrix}\right)_{\alpha_{0}}\rightarrow\left(\begin{smallmatrix}A_{1}\\B_{1}\end{smallmatrix}
\right)_{\alpha_{1}}\rightarrow\cdots\rightarrow\left(\begin{smallmatrix}A_{n}\\B_{n}\end{smallmatrix}
\right)_{\alpha_{n}}\rightarrow0\eqno(\ast)$$
in $(T\downarrow\mathcal{A})$ with each $\left(\begin{smallmatrix}A_{i}\\B_{i}\end{smallmatrix}\right)
_{\alpha_{i}}\in\mathfrak{B}_{\mathcal{Y}}^{\mathcal{X}}$. Then we have the exact sequence
$$0\rightarrow B\rightarrow B_{0}\rightarrow B_{1}\rightarrow\cdots\rightarrow B_{n}\rightarrow0$$
in $\mathcal{B}$ with each $B_{i}\in\mathcal{Y}$. So $B\in\mathcal{Y}_{n}^{\vee}$. Since $\left(\begin{smallmatrix}A_{i}\\B_{i}\end{smallmatrix}\right)_{\alpha_{i}}\in\mathfrak{B}_
{\mathcal{Y}}^{\mathcal{X}}$, we have $\alpha_{i}$ is a monomorphism and $\mathrm{Coker}\alpha_{i}\in\mathcal{X}$ for each $0\leqslant i\leqslant n$. Consider the following diagram:
$$\small\xymatrix@R=0.8cm@C=0.5cm{
                 &0 \ar[d]_{}  &0 \ar[d]_{}   &0 \ar[d]_{}   &  &0 \ar[d]_{}     &0 \ar[d]_{}  \\
   0 \ar[r]^{}   &T(B) \ar[d]_{\varphi} \ar[r]^{} &T(B_{0}) \ar[d]_{\alpha_{0}} \ar[r]^{}&T(B_{1}) \ar[d]_{\alpha_{1}} \ar[r]^{}  &\cdots  \ar[r]^{}  &T(B_{n-1})  \ar[d]_{\alpha_{n-1}} \ar[r]^{}  &T(B_{n})  \ar[d]_{\alpha_{n}} \ar[r]^{}   &0  \\
   0 \ar[r]^{}   &A \ar[d]_{} \ar[r]^{} &A_{0}  \ar[d]_{} \ar[r]^{}&A_{1}  \ar[d]_{} \ar[r]^{}&\cdots \ar[r]^{}&A_{n-1} \ar[d]_{} \ar[r]^{} &A_{n}  \ar[d]_{} \ar[r]^{} &0   \\
   0 \ar[r]^{}   &\mathrm{Coker}\varphi  \ar[d]_{} \ar[r]^{} &\mathrm{Coker}\alpha_{0} \ar[d]_{} \ar[r]^{}   &\mathrm{Coker}\alpha_{1}  \ar[d]_{} \ar[r]^{}&\cdots \ar[r]^{}  &\mathrm{Coker}\alpha_{n-1}  \ar[d]_{} \ar[r]^{} &\mathrm{Coker}\alpha_{n}  \ar[d]_{} \ar[r]^{}    &0   \\
                 &0   &0       &0    &   &0         &0
    &      \\   }\vspace*{2mm}$$
Due to the exact sequence $(\ast)$, the second row is exact and the upper squares are commutative. By assumption and Proposition 3.8 (1), the first row is exact. In addition, since $\alpha_{0}$ is a monomorphism, so is $\varphi$. Since the first row and the second row in the above diagram are exact, so is the third row. It follows that $\mathrm{Coker}\varphi\in\mathcal{X}_{n}^{\vee}$. Then $\left(\begin{smallmatrix}A\\B\end{smallmatrix}
\right)_{\varphi}\in\mathfrak{B}_{\mathcal{Y}_{n}^{\vee}}^{\mathcal{X}_{n}^{\vee}}$ and hence $(\mathfrak{B}_{
\mathcal{Y}}^{\mathcal{X}})_{n}^{\vee}\subseteq\mathfrak{B}_{\mathcal{Y}_{n}^{\vee}}^{\mathcal{X}_{n}^{\vee}}$.

For any $\left(\begin{smallmatrix}C\\D\end{smallmatrix}\right)_{\psi}\in\mathfrak{B}_{\mathcal{Y}_{n}^{\vee}}
^{\mathcal{X}_{n}^{\vee}}$, we have $D\in\mathcal{Y}_{n}^{\vee}$, $\mathrm{Coker}\psi\in\mathcal{X}_{n}
^{\vee}$ and $\psi$ is a monomorphism. Consider the following exact sequence
$$0\rightarrow\left(\begin{smallmatrix}T(D)\\D\end{smallmatrix}\right)_{1}\rightarrow\left(
\begin{smallmatrix}C\\D\end{smallmatrix}\right)_{\psi}\rightarrow\left(
\begin{smallmatrix}\mathrm{Coker}\psi\\0\end{smallmatrix}\right)_{0}\rightarrow0$$
in $(T\downarrow\mathcal{A})$. Since $\mathrm{Coker}\psi\in\mathcal{X}_{n}^{\vee}$, there exists an exact sequence
$$0\rightarrow\mathrm{Coker}\psi\rightarrow X_{0}\rightarrow X_{1}\rightarrow\cdots\rightarrow X_{n}\rightarrow0$$
in $\mathcal{A}$ with each $X_{i}\in\mathcal{X}$. Further, there is an exact sequence
$$0\rightarrow\left(\begin{smallmatrix}\mathrm{Coker}\psi\\0\end{smallmatrix}\right)\rightarrow\left(
\begin{smallmatrix}X_{0}\\0\end{smallmatrix}\right)\rightarrow\left(
\begin{smallmatrix}X_{1}\\0\end{smallmatrix}\right)\rightarrow\cdots\rightarrow\left(
\begin{smallmatrix}X_{n}\\0\end{smallmatrix}\right)\rightarrow0$$
in $(T\downarrow\mathcal{A})$. Clearly, $\left(\begin{smallmatrix}X_{i}\\0\end{smallmatrix}\right)_{0}\in
\mathfrak{B}_{\mathcal{Y}}^{\mathcal{X}}$ for each $0\leqslant i\leqslant n$. Then $\left(\begin{smallmatrix}
\mathrm{Coker}\psi\\0\end{smallmatrix}\right)_{0}\in
(\mathfrak{B}_{\mathcal{Y}}^{\mathcal{X}})_{n}^{\vee}$.
Since $D\in\mathcal{Y}_{n}^{\vee}$, there exists an exact sequence
$$0\rightarrow D\rightarrow Y_{0}\rightarrow Y_{1}\rightarrow\cdots\rightarrow Y_{n}\rightarrow0$$
in $\mathcal{B}$ with each $Y_{i}\in\mathcal{Y}$. By assumption and Proposition 3.8 (1), there exists an exact sequence
$$0\rightarrow T(D)\rightarrow T(Y_{0})\rightarrow T(Y_{1})\rightarrow\cdots\rightarrow T(Y_{n})\rightarrow0$$
in $\mathcal{A}$. Further, there exists the following exact sequence
$$0\rightarrow\left(\begin{smallmatrix}T(D)\\D\end{smallmatrix}\right)\rightarrow\left(\begin{smallmatrix}T(Y_{0})
\\Y_{0}\end{smallmatrix}\right)\rightarrow\left(\begin{smallmatrix}T(Y_{1})\\Y_{1}\end{smallmatrix}\right)
\rightarrow\cdots\rightarrow\left(\begin{smallmatrix}T(Y_{n})\\Y_{n}\end{smallmatrix}\right)\rightarrow0$$
in $(T\downarrow\mathcal{A})$. Clearly, $\left(\begin{smallmatrix}T(Y_{i})\\Y_{i}\end{smallmatrix}\right)_{1}\in
\mathfrak{B}_{\mathcal{Y}}^{\mathcal{X}}$ for each $0\leqslant i\leqslant n$. So $\left(\begin{smallmatrix}T(D)
\\D\end{smallmatrix}\right)_{1}\in(\mathfrak{B}_{\mathcal{Y}}^{\mathcal{X}})_{n}^{\vee}$. Since $(\mathfrak{B}
_{\mathcal{Y}}^{\mathcal{X}})_{n}^{\vee}$ is closed under extensions, $\left(\begin{smallmatrix}C\\D
\end{smallmatrix}\right)_{\psi}\in(\mathfrak{B}_{\mathcal{Y}}^{\mathcal{X}})_{n}^{\vee}$. Then $\mathfrak{B}
_{\mathcal{Y}_{n}^{\vee}}^{\mathcal{X}_{n}^{\vee}}\subseteq
(\mathfrak{B}_{\mathcal{Y}}^{\mathcal{X}})_{n}^{\vee}$.
\end{proof}

\begin{rem}\label{prop:2.4}{\rm{If $(\mathrm{R}^{j}G)M=0$ for each $1\leqslant j\leqslant n$ and $M\in\mathcal{M}$, then $(\mathrm{R}^{1}G)M'=0$ for any $M'\in\mathcal{M}_{n-1}^{\wedge}$. Indeed, for any $M'\in\mathcal{M}_{n-1}^{\wedge}$, there exists an exact sequence
$$0\rightarrow M_{n-1}\overset{d_{n-1}}\rightarrow\cdots\rightarrow M_{1}\overset{d_{1}}\rightarrow M_{0}\overset{d_{0}}\rightarrow M'\rightarrow0$$
in $\mathcal{A}$ with each $M_{i}\in\mathcal{M}$ . Then it follows from the assumption that
$$(\mathrm{R}^{1}G)M'\cong(\mathrm{R}^{2}G)\mathrm{Ker}d_{0}\cong\cdots\cong(\mathrm{R}^{n-1}G)\mathrm{Ker}d_{n-3}
\cong(\mathrm{R}^{n}G)M_{n-1}=0.$$
Similarly, if $(\mathrm{L}_{j}T)Y=0$ for each $1\leqslant j
\leqslant n$ and $Y\in\mathcal{Y}$, then $(\mathrm{L}_{1}T)Y'=0$ for any $Y'\in\mathcal{Y}_{n-1}^{\vee}$.
}}
\end{rem}

In the following, we construct the left $n$-cotorsion pairs in the comma category $(\mathcal{B}\downarrow G)$ using the left $n$-cotorsion pairs in $\mathcal{A}$ and $\mathcal{B}$, and construct the right $n$-cotorsion pairs in the comma category $(T\downarrow\mathcal{A})$ using the right $n$-cotorsion pairs in $\mathcal{A}$ and $\mathcal{B}$, together with Proposition 3.4, Lemma 3.7, Lemma 3.9 and other known results. The following theorem is the main result of this paper.

\begin{thm}\label{prop:2.4}{\it{$\mathrm{(1)}$ Suppose that the comma category $(\mathcal{B}\downarrow G)$ has enough projective objects, $(\mathcal{X}, \mathcal{M})$ is a left $n$-cotorsion pair in $\mathcal{A}$, and $(\mathcal{Y},\mathcal{N})$ is a left $n$-cotorsion pair in $\mathcal{B}$. If $(\mathrm{R}^j G)M=0$ for each $1\leqslant j \leqslant n$ and any $M\in\mathcal{M}$, the class $\left(\begin{smallmatrix}\mathcal{X}\\ \mathcal{Y}\end{smallmatrix}\right)$ is resolving, and $(\mathfrak{D}_{\mathcal{N}}
^{\mathcal{M}})_{n-1}^{\wedge}$ is closed under extensions, then $(\left(\begin{smallmatrix}\mathcal{X}\\ \mathcal{Y}\end{smallmatrix}\right),\mathfrak{D}_{\mathcal{N}}^{\mathcal{M}})$ is a left $n$-cotorsion pair in the comma category $(\mathcal{B} \downarrow G)$.

$\mathrm{(2)}$ Suppose that the comma category $(T\downarrow\mathcal{A})$ has enough injective objects, $(\mathcal{X},\mathcal{M})$ is a right $n$-cotorsion pair in $\mathcal{A}$, and $(\mathcal{Y},\mathcal{N})$ is a right $n$-cotorsion pair in $\mathcal{B}$. If $(\mathrm{L}
_{j}T)Y=0$ for each $1\leqslant j\leqslant n$ and any $Y\in\mathcal{Y}$, the class $\left(\begin{smallmatrix}\mathcal{M}\\ \mathcal{N}\end{smallmatrix}\right)$ is coresolving, and $(\mathfrak{B}_{\mathcal{Y}}
^{\mathcal{X}})_{n-1}^{\vee}$ is closed under extensions, then $(\mathfrak{B}_{\mathcal{Y}}^{\mathcal{X}}
,\left(\begin{smallmatrix}\mathcal{M}\\ \mathcal{N}\end{smallmatrix}\right))$ is a right $n$-cotorsion pair in the comma category $(T\downarrow\mathcal{A})$.
}}
\end{thm}
\begin{proof} The proof of (1) is given below, and the proof of (2) is similar.

Since $(\mathcal{X}, \mathcal{M})$ is a left $n$-cotorsion pair in  $\mathcal{A}$ and $(\mathcal{Y},\mathcal{N})$ is a left $n$-cotorsion pair in $\mathcal{B}$, we have $\mathcal{X}=\bigcap
\limits_{i=1}\limits^{n}{^{\bot_{i}}\mathcal{M}}$ and $\mathcal{Y}=\bigcap\limits_{i=1}\limits^{n}{^{\bot_{i}}
\mathcal{N}}$ by \cite[\text{Theorem}~2.7]{MOM2021}. Then $\bigcap\limits_{i=1}\limits^{n}{^{\bot_{i}}
(\mathfrak{D}_{\mathcal{N}}^{\mathcal{M}})}=\left(\begin{smallmatrix}{^{\bot_{[1,n]}}\mathcal{M}}\\ {^{\bot_{[1,n]}}\mathcal{N}}\end{smallmatrix}\right)=\left(\begin{smallmatrix}\mathcal{X}\\ \mathcal{Y}\end{smallmatrix}\right)$ by Lemma 3.7 (2). Note that $\left(\begin{smallmatrix}\mathcal{X}\\ \mathcal{Y}\end{smallmatrix}\right)$ contains all projective objects, that is, $\left(\begin{smallmatrix}
\mathcal{X}\\ \mathcal{Y}\end{smallmatrix}\right)$ is a projectively resolving class. By
\cite[\text{Proposition}~2.5]{MOM2021}, we have $\mathrm{Ext}
_{(\mathcal{B}\downarrow G)}^{1}(\left(\begin{smallmatrix}\mathcal{X}\\ \mathcal{Y}\end{smallmatrix}\right),(\mathfrak{D}_{\mathcal{N}}^{\mathcal{M}})_{n-1}^{\wedge})
=0$ and $(\mathfrak{D}_{\mathcal{N}}^{\mathcal{M}})_{n-1}^{\wedge}\subseteq{\left(\begin{smallmatrix}\mathcal{X}\\ \mathcal{Y}\end{smallmatrix}\right)}^{\bot}$. For any $\left(\begin{smallmatrix}C\\D\end{smallmatrix}
\right)_{\psi}\in(\mathcal{B}\downarrow G)$, there exists a short exact sequence
$$0\rightarrow\left(\begin{smallmatrix}C\\0\end{smallmatrix}\right)_{0}\rightarrow
\left(\begin{smallmatrix}C\\D\end{smallmatrix}\right)_{\psi}\rightarrow\left(\begin{smallmatrix}0
\\D\end{smallmatrix}\right)_{0}\rightarrow0$$
in $(\mathcal{B}\downarrow G)$. Since $(\mathcal{X},\mathcal{M})$ is a left $n$-cotorsion pair in $\mathcal{A}$, there exists a short exact sequence $0\rightarrow M\overset{\alpha}\rightarrow X\overset{\pi}\rightarrow C\rightarrow0$ in $\mathcal{A}$ with $X\in\mathcal{X}$ and $M\in\mathcal{M}
_{n-1}^{\wedge}$. Since $(\mathcal{Y},\mathcal{N})$ is a left $n$-cotorsion pair in $\mathcal{B}$, there exists a short exact sequence $0\rightarrow N\rightarrow Y\overset{\beta}\rightarrow G(M)\rightarrow0$ in $\mathcal{B}$, where $Y\in\mathcal{Y}$ and $N\in\mathcal{N}_{n-1}^{\wedge}$. Then $\left(\begin{smallmatrix}
X\\Y\end{smallmatrix}\right)_{G(\alpha)\beta}$ is an object in $(\mathcal{B}\downarrow G)$ and $\left(\begin
{smallmatrix}X\\Y\end{smallmatrix}\right)_{G(\alpha)\beta}\in\left(\begin{smallmatrix}
\mathcal{X}\\ \mathcal{Y}\end{smallmatrix}\right)$. There exists a short exact sequence
$$0\rightarrow\left(\begin{smallmatrix}M\\Y\end{smallmatrix}\right)_{\beta}\overset{\left(\begin{smallmatrix}
\alpha\\1\end{smallmatrix}\right)}\rightarrow\left(\begin{smallmatrix}X\\Y
\end{smallmatrix}\right)_{G(\alpha)\beta}\overset{\left(\begin{smallmatrix}
\pi\\0\end{smallmatrix}\right)}\rightarrow\left(\begin{smallmatrix}C\\0\end{smallmatrix}\right)_{0}
\rightarrow0$$
in $(\mathcal{B}\downarrow G)$. In addition, by assumption, Lemma 3.9 (4) and Remark 3.10, $$\left(\begin{smallmatrix}M\\Y\end{smallmatrix}\right)_{\beta}\in
\mathfrak{D}_{\mathcal{N}_{n-1}^{\wedge}}^{\mathcal{M}_{n-1}^{\wedge}}=(\mathfrak{D}_{\mathcal{N}}
^{\mathcal{M}})_{n-1}^{\wedge}\subseteq{\left(\begin{smallmatrix}\mathcal{X}\\ \mathcal{Y}\end{smallmatrix}\right)}^{\bot}.$$ So $\left(\begin{smallmatrix}
C\\0\end{smallmatrix}\right)_{0}$ has a special $\left(\begin{smallmatrix}\mathcal{X}\\ \mathcal{Y}\end{smallmatrix}\right)$-precover. For $D\in\mathcal{B}$, there exists a short exact sequence $0\rightarrow N_{1}\rightarrow Y_{1}\rightarrow D\rightarrow0$ in $\mathcal{B}$ with $Y_{1}\in
\mathcal{Y}$ and $N_{1}\in\mathcal{N}_{n-1}^{\wedge}$. Further, there exists an exact sequence $$0\rightarrow\left(\begin{smallmatrix}0\\N_{1}\end{smallmatrix}\right)_{0}\rightarrow
\left(\begin{smallmatrix}0\\Y_{1}\end{smallmatrix}\right)_{0}\rightarrow\left(\begin{smallmatrix}0\\D
\end{smallmatrix}\right)_{0}\rightarrow0$$
in $(\mathcal{B}\downarrow G)$ with $\left(\begin{smallmatrix}0\\Y_{1}\end{smallmatrix}\right)_{0}\in\left(
\begin{smallmatrix}\mathcal{X}\\ \mathcal{Y}\end{smallmatrix}\right)$. By assumption, Lemma 3.9 (4) and Remark 3.10 again, $$\left(\begin{smallmatrix}0
\\N_{1}\end{smallmatrix}\right)_{0}\in\mathfrak{D}_{\mathcal{N}_{n-1}^{\wedge}}^{\mathcal{M}_{n-1}^{\wedge}}
=(\mathfrak{D}_{\mathcal{N}}^{\mathcal{M}})_{n-1}^{\wedge}\subseteq{\left(\begin{smallmatrix}\mathcal{X}\\ \mathcal{Y}\end{smallmatrix}\right)}^{\bot}.$$
So $\left(\begin{smallmatrix}0
\\D\end{smallmatrix}\right)_{0}$ also has a special $\left(\begin{smallmatrix}\mathcal{X}\\ \mathcal{Y}\end{smallmatrix}\right)$-precover. By Proposition 3.4 (2), there exists the following commutative diagram with exact rows and columns:
$$\small\xymatrix@R=0.8cm@C=0.8cm{
   &0 \ar[d]_{}                       &0 \ar[d]_{}                         &0 \ar[d]_{} &   \\
  0  \ar[r]^{} &{\left(\begin{smallmatrix}M\\Y\end{smallmatrix}\right)}  \ar[d]_{} \ar[r]^{}  &{\left(\begin{smallmatrix}M\\H\end{smallmatrix}\right)} \ar[d]_{} \ar[r]^{}  &{\left(\begin{smallmatrix}0\\N_{1}\end{smallmatrix}\right)} \ar[d]_{} \ar[r]^{}        &0 \\
 0 \ar[r]^{}      &{\left(\begin{smallmatrix}X\\Y\end{smallmatrix}\right)} \ar[d]_{} \ar[r]^{}  &{\left(\begin{smallmatrix}X\\Y_{2}\end{smallmatrix}\right)} \ar[d]_{} \ar[r]^{}  &{\left(\begin{smallmatrix}0\\Y_{1}\end{smallmatrix}\right)} \ar[d]_{} \ar[r]^{}        &0  \\
  0  \ar[r]^{}     &{\left(\begin{smallmatrix}C\\0\end{smallmatrix}\right)} \ar[d]_{} \ar[r]^{}   &{\left(\begin{smallmatrix}C\\D\end{smallmatrix}\right)} \ar[d]_{} \ar[r]^{}   &{\left(\begin{smallmatrix}0\\D\end{smallmatrix}\right)} \ar[d]_{} \ar[r]^{}        &0  \\
  &0 &0 &0 &\\
   }\vspace*{2mm}$$
with $\left(\begin{smallmatrix}X\\Y_{2}\end{smallmatrix}\right)\in\left(
\begin{smallmatrix}\mathcal{X}\\ \mathcal{Y}\end{smallmatrix}\right)$. Since $(\mathfrak{D}_{\mathcal{N}}^{\mathcal{M}})_{n-1}^{\wedge}$ is closed under extensions, we have $\left(\begin{smallmatrix}M\\H\end{smallmatrix}\right)\in(\mathfrak{D}_{\mathcal{N}}
^{\mathcal{M}})_{n-1}^{\wedge}$. Hence $(\left(\begin{smallmatrix}\mathcal{X}
\\ \mathcal{Y}\end{smallmatrix}\right),\mathfrak{D}_{\mathcal{N}}^{\mathcal{M}})$ is a left $n$-cotorsion pair in $(\mathcal{B} \downarrow G)$ by \cite[\text{Theorem}~2.7]{MOM2021}.
\end{proof}

Next, we construct the left $n$-cotorsion pairs in $\mathcal{A}$ and $\mathcal{B}$ through the left $n$-cotorsion pairs in the comma category $(\mathcal{B}\downarrow G)$, and construct the right $n$-cotorsion pairs in $\mathcal{A}$ and $\mathcal{B}$ through the right $n$-cotorsion pairs in the comma category $(T\downarrow\mathcal{A})$.

\begin{thm}\label{prop:2.4}{\it{
$\mathrm{(1)}$ Assume that $(\left(\begin{smallmatrix}\mathcal{X}\\ \mathcal{Y}\end{smallmatrix}
\right),\mathfrak{D}_{\mathcal{N}}^{\mathcal{M}})$ is a left $n$-cotorsion pair in the comma category $(\mathcal{B}\downarrow G)$.

$(a)$ If $(\mathrm{R}^{j}G)M=0$ for each $1\leqslant j\leqslant n$ and any $M\in\mathcal{M}$, then $(\mathcal{X},
\mathcal{M})$ is a left $n$-cotorsion pair in $\mathcal{A}$.

$(b)$ If $(\mathrm{R}^{j}G)M=0$ for each $1\leqslant j\leqslant n$ and any $M\in\mathcal{M}$, $\mathcal{N}_{n-1}
^{\wedge}$ is closed under extensions and $G(X)\in\mathcal{N}_{n-1}^{\wedge}$ for any $X\in\mathcal{X}$, then $(\mathcal{Y},\mathcal{N})$ is a left $n$-cotorsion pair in $\mathcal{B}$.

$\mathrm{(2)}$ Assume that $(\mathfrak{B}_{\mathcal{Y}}^{\mathcal{X}},\left(\begin{smallmatrix}\mathcal{M}\\ \mathcal{N}\end{smallmatrix}\right))$ is a right $n$-cotorsion pair in the comma category $(T\downarrow
\mathcal{A})$.

$(a)$ If $(\mathrm{L}_{j}T)Y=0$ for each $1\leqslant j\leqslant n$ and any $Y\in\mathcal{Y}$, then $(\mathcal{Y},\mathcal{N})$ is a right $n$-cotorsion pair in $\mathcal{B}$.

$(b)$ If $(\mathrm{L}_{j}T)Y=0$ for each $1\leqslant j\leqslant n$ and any $Y\in\mathcal{Y}$, $\mathcal{X}
_{n-1}^{\vee}$ is closed under extensions and $T(N)\in\mathcal{X}_{n-1}^{\vee}$ for any $N\in\mathcal{N}$, then $(\mathcal{X},\mathcal{M})$ is a right $n$-cotorsion pair in $\mathcal{A}$.
}}
\end{thm}
\begin{proof} The proof of (1) is given below, and the proof of (2) is similar.

$(a)$ For any $X\in\mathcal{X}$, $M\in\mathcal{M}$ and $1\leqslant i\leqslant n$, since $\left(\begin{smallmatrix}X\\0\end{smallmatrix}\right)_{0}\in
\left(\begin{smallmatrix}\mathcal{X}\\ \mathcal{Y}\end{smallmatrix}\right)$ and $\left(\begin{smallmatrix}M\\G(M)
\end{smallmatrix}\right)_{1}\in\mathfrak{D}_{\mathcal{N}}^{\mathcal{M}}$, we have
$$\mathrm{Ext}_{\mathcal{A}}^{i}(X,M)\cong\mathrm{Ext}_{(\mathcal{B}\downarrow G)}^{i}(\left(\begin{smallmatrix}X\\0\end{smallmatrix}\right)_{0},
\left(\begin{smallmatrix}M\\G(M)
\end{smallmatrix}\right)_{1})=0$$
by assumption and Lemma 3.6 (2).
So $\mathcal{X}\subseteq\bigcap\limits_{i=1}\limits^{n}{^{\bot_{i}}
\mathcal{M}}$. For any $C\in\bigcap\limits_{i=1}\limits^{n}{^{\bot_{i}}\mathcal{M}}$, we have
$$\left(\begin{smallmatrix}C\\0\end{smallmatrix}\right)_{0}\in\left
(\begin{smallmatrix}{^{\bot_{[1,n]}}\mathcal{M}}\\{^{\bot_{[1,n]}}\mathcal{N}}\end{smallmatrix}\right)=
\textstyle\bigcap\limits_{i=1}\limits^{n}{^{\bot_{i}}(\mathfrak{D}_{\mathcal{N}}^{\mathcal{M}})}=\left
(\begin{smallmatrix}\mathcal{X}\\ \mathcal{Y}\end{smallmatrix}\right)$$
by Lemma 3.7 (2) and \cite[\text{Theorem}~2.7]{MOM2021}. We obtain $C\in\mathcal{X}$. Namely, $\bigcap\limits_{i=1}
\limits^{n}{^{\bot_{i}}\mathcal{M}}\subseteq\mathcal{X}$. Hence $\mathcal{X}=\bigcap\limits_{i=1}\limits^{n}
{^{\bot_{i}}\mathcal{M}}$. For any $A\in\mathcal{A}$, since $(\left(\begin{smallmatrix}\mathcal{X}\\ \mathcal{Y}\end{smallmatrix}\right),\mathfrak{D}_{\mathcal{N}}^{\mathcal{M}})$ is a left $n$-cotorsion pair in $(\mathcal{B}\downarrow G)$, there exists a short exact sequence
$$0\rightarrow\left(\begin{smallmatrix}X_{2}\\Y_{1}\end{smallmatrix}\right)_{\psi}\rightarrow\left(
\begin{smallmatrix}X_{1}\\Y\end{smallmatrix}\right)_{\varphi}\rightarrow\left(\begin{smallmatrix}A\\0
\end{smallmatrix}\right)_{0}\rightarrow0$$
in $(\mathcal{B}\downarrow G)$, where $\left(\begin{smallmatrix}X_{1}\\Y\end{smallmatrix}\right)_{\varphi}\in
\left(\begin{smallmatrix}\mathcal{X}\\ \mathcal{Y}\end{smallmatrix}\right)$ and $\left(\begin{smallmatrix}X_{2}
\\Y_{1}\end{smallmatrix}\right)_{\psi}\in(\mathfrak{D}_{\mathcal{N}}^{\mathcal{M}})_{n-1}^{\wedge}$. In addition, by assumption, Lemma 3.9 (4) and Remark 3.10, $\left(\begin{smallmatrix}X_{2}\\Y_{1}\end{smallmatrix}\right)_{\psi}
\in(\mathfrak{D}_{\mathcal{N}}^{\mathcal{M}})_{n-1}^{\wedge}\subseteq\mathfrak{D}_{\mathcal{N}_{n-1}^{\wedge}}
^{\mathcal{M}_{n-1}^{\wedge}}$. Then we have the short exact sequence $0\rightarrow X_{2}\rightarrow X_{1}\rightarrow A\rightarrow0$ in $\mathcal{A}$ with $X_{1}\in\mathcal{X}$ and $X_{2}\in
\mathcal{M}_{n-1}^{\wedge}$. Hence $(\mathcal{X},\mathcal{M})$ is a left $n$-cotorsion pair in $\mathcal{A}$ by \cite[\text{Theorem}~2.7]{MOM2021}.

$(b)$ For any $Y\in\mathcal{Y}$, $N\in\mathcal{N}$ and $1\leqslant i\leqslant n$, since $\left(\begin{smallmatrix}0\\Y\end{smallmatrix}\right)_{0}\in
\left(\begin{smallmatrix}\mathcal{X}\\ \mathcal{Y}\end{smallmatrix}\right)$ and $\left(\begin{smallmatrix}0\\N
\end{smallmatrix}\right)_{0}\in\mathfrak{D}_{\mathcal{N}}^{\mathcal{M}}$, we have
$$\mathrm{Ext}_{\mathcal{B}}^{i}(Y,N)\cong\mathrm{Ext}_{(\mathcal{B}\downarrow G)}^{i}(\left(\begin{smallmatrix}0\\Y\end{smallmatrix}\right)_{0},\left(\begin{smallmatrix}0\\N
\end{smallmatrix}\right)_{0})=0$$
by assumption.
So $\mathcal{Y}\subseteq\bigcap\limits_{i=1}\limits^{n}{^{\bot_{i}}
\mathcal{N}}$. For any $D\in\bigcap\limits_{i=1}\limits^{n}{^{\bot_{i}}\mathcal{N}}$, we have
$$\left(\begin{smallmatrix}0\\D\end{smallmatrix}\right)_{0}\in\left
(\begin{smallmatrix}{^{\bot_{[1,n]}}\mathcal{M}}\\{^{\bot_{[1,n]}}\mathcal{N}}\end{smallmatrix}\right)=
\textstyle\bigcap\limits_{i=1}\limits^{n}{^{\bot_{i}}(\mathfrak{D}_{\mathcal{N}}^{\mathcal{M}})}=\left
(\begin{smallmatrix}\mathcal{X}\\ \mathcal{Y}\end{smallmatrix}\right)$$
by Lemma 3.7 (2) and \cite[\text{Theorem}~2.7]{MOM2021}. So $D\in\mathcal{Y}$ and $\bigcap\limits_{i=1}
\limits^{n}{^{\bot_{i}}\mathcal{N}}\subseteq\mathcal{Y}$. Hence $\mathcal{Y}=\bigcap\limits_{i=1}\limits^{n}
{^{\bot_{i}}\mathcal{N}}$. For any $B\in\mathcal{B}$, there exists a short exact sequence
$$0\rightarrow\left(\begin{smallmatrix}K_{2}\\L_{1}\end{smallmatrix}\right)_{\psi^{'}}\rightarrow
\left(\begin{smallmatrix}K_{1}\\L\end{smallmatrix}\right)_{\varphi^{'}}\rightarrow\left(
\begin{smallmatrix}0\\B\end{smallmatrix}\right)_{0}\rightarrow0$$
in $(\mathcal{B}\downarrow G)$ with $\left(\begin{smallmatrix}K_{1}\\L\end{smallmatrix}\right)
_{\varphi^{'}}\in\left(\begin{smallmatrix}\mathcal{X}\\ \mathcal{Y}\end{smallmatrix}
\right)$ and $\left(\begin{smallmatrix}K_{2}\\L_{1}\end{smallmatrix}\right)_{\psi^{'}}\in(\mathfrak{D}
_{\mathcal{N}}^{\mathcal{M}})_{n-1}^{\wedge}$. Obviously, $K_{1}\in\mathcal{X}$, $L\in\mathcal{Y}$ and $K_{1}\cong K_{2}$. By assumption, Lemma 3.9 (4) and Remark 3.10 again, $\left(\begin{smallmatrix}K_{2}\\L_{1}
\end{smallmatrix}\right)_{\psi^{'}}\in(\mathfrak{D}_{\mathcal{N}}^{\mathcal{M}})_{n-1}^{\wedge}\subseteq
\mathfrak{D}_{\mathcal{N}_{n-1}^{\wedge}}^{\mathcal{M}_{n-1}^{\wedge}}$ . Then we get a short exact sequence
$$0\rightarrow\mathrm{Ker}\psi^{'}\rightarrow L_{1}\rightarrow G(K_{2})\rightarrow0$$
in $\mathcal{B}$ with $\mathrm{Ker}\psi^{'}\in\mathcal{N}_{n-1}^{\wedge}$. By the assumption that $G(X)\in
\mathcal{N}_{n-1}^{\wedge}$ for any $X\in\mathcal{X}$, we obtain $G(K_{2})\cong G(K_{1})\in\mathcal{N}_{n-1}^{\wedge}$. Since $\mathcal{N}_{n-1}^{\wedge}$ is closed under extensions, we have $L_{1}\in\mathcal{N}_{n-1}^{\wedge}$. So there exists a short exact sequence $0\rightarrow L_{1}\rightarrow L\rightarrow B\rightarrow0$ in $\mathcal{B}$ with $L\in\mathcal{Y}$ and $L_{1}\in\mathcal{N}_{n-1}^{\wedge}$. By \cite[\text{Theorem}~2.7]{MOM2021}, $(\mathcal{Y},\mathcal{N})$ is a left $n$-cotorsion pair in $\mathcal{B}$.
\end{proof}

\begin{rem}\label{prop:2.4}{\rm{In the case that keep the conditions as Theorem 3.11 or Theorem 3.12, according to \cite[\text{Proposition}~3.1]{MOM2021} and its dual, one gets that $\mathcal{X}$, $\mathcal{Y}$ and $\left(\begin{smallmatrix}\mathcal{X}\\ \mathcal{Y}\end{smallmatrix}\right)$ are special precovering classes, and $\mathcal{M}$, $\mathcal{N}$ and $\left(\begin{smallmatrix}\mathcal{M}\\ \mathcal{N}
\end{smallmatrix}\right)$ are special preenveloping classes.
}}
\end{rem}

Of course, we know that in general, Theorem 3.11 seems to be more important than Theorem 3.12. However, Theorem 3.12 can be regarded as an incidental result in the study of how to construct left or right $n$-cotorsion pairs in comma categories. By Remark 3.13, we realize that the study of left or right $n$-cotorsion pairs can enrich and develop the approximation theory. To conclude this part, we discuss the heredity of left or right $n$-cotorsion pairs in comma categories and the corresponding abelian categories.

The following lemma is a supplement to \cite[\text{Proposition}~4.9]{MOM2021}.

\begin{lem}\label{prop:2.4}{\it{$\mathrm{(1)}$ Let $(\mathcal{K},\mathcal{L})$ be a left $n$-cotorsion pair in the abelian category $\mathcal{C}$. Consider the following conditions.

$(a)$ For any $K\in\mathcal{K}$, $L\in\mathcal{L}$ and $i\geqslant1$, we have $\mathrm{Ext}_{\mathcal{C}}
^{i}(K,L)=0$.

$(b)$ $(\mathcal{K},\mathcal{L})$ is hereditary, i.e., $\mathrm{Ext}_{\mathcal{C}}^{n+1}(K,L)=0$ for any $K\in\mathcal{K}$ and $L\in\mathcal{L}$.

$(c)$ $\mathcal{K}$ is a resolving class in $\mathcal{C}$.

$(d)$ $\mathcal{K}$ is closed under kernels of epimorphisms.\\
Then $(a)\Rightarrow(b)\Rightarrow(c)\Rightarrow(d)$ holds. In addition, if $\mathcal{C}$ has enough projective objects, then $(d)\Rightarrow(a)$ holds.

$\mathrm{(2)}$ Let $(\mathcal{K},\mathcal{L})$ be a right $n$-cotorsion pair in the abelian category $\mathcal{C}$. Consider the following conditions.

$(a)$ For any $K\in\mathcal{K}$, $L\in\mathcal{L}$ and $i\geqslant1$, we have $\mathrm{Ext}_{\mathcal{C}}
^{i}(K,L)=0$.

$(b)$ $(\mathcal{K},\mathcal{L})$ is hereditary, i.e., $\mathrm{Ext}_{\mathcal{C}}^{n+1}(K,L)=0$ for any $K\in\mathcal{K}$ and $L\in\mathcal{L}$.

$(c)$ $\mathcal{L}$ is a coresolving class in $\mathcal{C}$.

$(d)$ $\mathcal{L}$ is closed under cokernels of monomorphisms.\\
Then $(a)\Rightarrow(b)\Rightarrow(c)\Rightarrow(d)$ holds. In addition, if $\mathcal{C}$ has enough injective objects, then $(d)\Rightarrow(a)$ holds.
}}
\end{lem}
\begin{proof} The proof of (1) is given below, and the proof of (2) is similar.

Clearly, $(a)\Rightarrow(b)$ and $(c)\Rightarrow(d)$ hold. By \cite[\text{Proposition}~4.9]{MOM2021}, we have $(b)\Rightarrow(c)$. It remains to prove that $(d)\Rightarrow(a)$.

Let $K\in\mathcal{K}$ and $L\in\mathcal{L}$. Since $(\mathcal{K},\mathcal{L})$ is a left $n$-cotorsion pair, we have $\mathrm{Ext}_{\mathcal{C}}^{i}(K,L)=0$ for each $1\leqslant i\leqslant n$. For any $j\geqslant1$, take a partial projective resolution of $K$
$$0\rightarrow K_{j-1}\rightarrow P_{j-1}\overset{d_{j-1}}\rightarrow P_{j-2}\overset{d_{j-2}}\rightarrow\cdots\rightarrow P_{1}\overset{d_{1}}\rightarrow P_{0}
\overset{d_{0}}\rightarrow K\rightarrow0,$$
where each $P_{l}$ is a projective object in $\mathcal{C}$, and $K_{j-1}$ is the $(j-1)$-th syzygy of $K$. By dimension shifting, we have $\mathrm{Ext}_{\mathcal{C}}^{n}(K_{j-1},L)\cong\mathrm{Ext}_{\mathcal{C}}
^{n+j}(K,L)$. Since $(\mathcal{K},\mathcal{L})$ is a left $n$-cotorsion pair and $\mathcal{C}$ has enough projective objects, $\mathcal{K}$ contains all projective objects. Moreover, since $\mathcal{K}$ is closed under kernels of epimorphisms, we have $K_{j-1}\in\mathcal{K}$. Therefore, $\mathrm{Ext}_{\mathcal{C}}^{n}(K_{j-1},L)=0$ for any $j\geqslant1$. Then $\mathrm{Ext}_{\mathcal{C}}^{i}(K,L)=0$ for any $i\geqslant1$.
\end{proof}

\begin{rem}\label{prop:2.4}{\rm{(1) In Theorem 3.11 (1), since the comma category $(\mathcal{B}\downarrow G)$ has enough projective objects and  $\left(\begin{smallmatrix}\mathcal{X}\\ \mathcal{Y}\end{smallmatrix}\right)$ is a resolving class, we have that the left $n$-cotorsion pair $(\left(\begin{smallmatrix}
\mathcal{X}\\ \mathcal{Y}\end{smallmatrix}\right),\mathfrak{D}_{\mathcal{N}}
^{\mathcal{M}})$ in $(\mathcal{B}\downarrow G)$ is hereditary by Lemma 3.14 (1).

(2) In Theorem 3.11 (2), since the comma category $(T\downarrow\mathcal{A})$ has enough injective objects and $\left(\begin{smallmatrix}\mathcal{M}\\ \mathcal{N}
\end{smallmatrix}\right)$ is a coresolving class, we have that the right $n$-cotorsion pair $(\mathfrak{B}_{\mathcal{Y}}
^{\mathcal{X}},\left(\begin{smallmatrix}\mathcal{M}\\ \mathcal{N}\end{smallmatrix}\right))$ in $(T\downarrow
\mathcal{A})$ is hereditary by Lemma 3.14 (2).
}}
\end{rem}

Now, we characterize hereditary left (resp. right) $n$-cotorsion pairs in comma categories by hereditary left (resp. right) $n$-cotorsion pairs in corresponding abelian categories.

\begin{prop}\label{prop:2.4}{\it{$\mathrm{(1)}$ Suppose that the comma category $(\mathcal{B}\downarrow G)$ has enough projective objects, $(\mathcal{X}, \mathcal{M})$ is a hereditary left $n$-cotorsion pair in $\mathcal{A}$, and $(\mathcal{Y},\mathcal{N})$ is a hereditary left $n$-cotorsion pair in $\mathcal{B}$. If $(\mathrm{R}^j G)M=0$ for each $1\leqslant j \leqslant n$ and any $M\in\mathcal{M}$, and $(\mathfrak{D}_{\mathcal{N}}
^{\mathcal{M}})_{n-1}^{\wedge}$ is closed under extensions, then $(\left(\begin{smallmatrix}\mathcal{X}\\ \mathcal{Y}\end{smallmatrix}\right),\mathfrak{D}_{\mathcal{N}}^{\mathcal{M}})$ is a hereditary left $n$-cotorsion pair in the comma category $(\mathcal{B}\downarrow G)$.

$\mathrm{(2)}$ Suppose that the comma category $(T\downarrow\mathcal{A})$ has enough injective objects, $(\mathcal{X},\mathcal{M})$ is a hereditary right $n$-cotorsion pair in $\mathcal{A}$, and $(\mathcal{Y},\mathcal{N})$ is a hereditary right $n$-cotorsion pair in $\mathcal{B}$. If $(\mathrm{L}_{j}T)Y=0$ for each $1\leqslant j\leqslant n$ and any $Y\in\mathcal{Y}$, and $(\mathfrak{B}_{\mathcal{Y}}^{\mathcal{X}})_{n-1}^{\vee}$ is closed under extensions, then $(\mathfrak{B}_{\mathcal{Y}}^{\mathcal{X}},\left(\begin{smallmatrix}\mathcal{M}\\ \mathcal{N}\end{smallmatrix}\right))$ is a hereditary right $n$-cotorsion pair in the comma category $(T\downarrow
\mathcal{A})$.
}}
\end{prop}
\begin{proof} $\mathrm{(1)}$ It follows from \cite[\text{Lemma}~4.1~(1)]{YJD2024}, Theorem 3.11 (1) and Remark 3.15 (1) that $(\left(\begin{smallmatrix}\mathcal{X}\\\mathcal{Y}\end{smallmatrix}\right),\mathfrak{D}_{\mathcal{N}}
^{\mathcal{M}})$ is a hereditary left $n$-cotorsion pair in the comma category $(\mathcal{B}\downarrow G)$.

$\mathrm{(2)}$ It follows from \cite[\text{Lemma}~3.2~(1)]{JH2022}, Theorem 3.11 (2) and Remark 3.15 (2) that $(\mathfrak{B}_{\mathcal{Y}}^{\mathcal{X}},\left(\begin{smallmatrix}\mathcal{M}\\ \mathcal{N}\end{smallmatrix}\right))$ is a hereditary right $n$-cotorsion pair in the comma category $(T\downarrow
\mathcal{A})$.
\end{proof}

At last, we characterize hereditary left (resp. right) $n$-cotorsion pairs in abelian categories by hereditary left (resp. right) $n$-cotorsion pairs in corresponding comma categories.

\begin{prop}\label{prop:2.4}{\it{$\mathrm{(1)}$ Assume that $(\left(\begin{smallmatrix}\mathcal{X}\\ \mathcal{Y}\end{smallmatrix}
\right),\mathfrak{D}_{\mathcal{N}}^{\mathcal{M}})$ is a hereditary left $n$-cotorsion pair in the comma category $(\mathcal{B}\downarrow G)$.

$(a)$ If $(\mathrm{R}^{j}G)M=0$ for each $1\leqslant j\leqslant n$ and any $M\in\mathcal{M}$, then $(\mathcal{X},
\mathcal{M})$ is a hereditary left $n$-cotorsion pair in $\mathcal{A}$.

$(b)$ If $(\mathrm{R}^{j}G)M=0$ for each $1\leqslant j\leqslant n$ and any $M\in\mathcal{M}$, $\mathcal{N}_{n-1}
^{\wedge}$ is closed under extensions and $G(X)\in\mathcal{N}_{n-1}^{\wedge}$ for any $X\in\mathcal{X}$, then $(\mathcal{Y},\mathcal{N})$ is a hereditary left $n$-cotorsion pair in $\mathcal{B}$.

$\mathrm{(2)}$ Assume that $(\mathfrak{B}_{\mathcal{Y}}^{\mathcal{X}},\left(\begin{smallmatrix}\mathcal{M}\\ \mathcal{N}\end{smallmatrix}\right))$ is a hereditary right $n$-cotorsion pair in the comma category $(T\downarrow
\mathcal{A})$.

$(a)$ If $(\mathrm{L}_{j}T)Y=0$ for each $1\leqslant j\leqslant n$ and any $Y\in\mathcal{Y}$, then $(\mathcal{Y},
\mathcal{N})$ is a hereditary right $n$-cotorsion pair in $\mathcal{B}$.

$(b)$ If $(\mathrm{L}_{j}T)Y=0$ for each $1\leqslant j\leqslant n$ and any $Y\in\mathcal{Y}$, $\mathcal{X}
_{n-1}^{\vee}$ is closed under extensions and $T(N)\in\mathcal{X}_{n-1}^{\vee}$ for any $N\in\mathcal{N}$, then $(\mathcal{X},\mathcal{M})$ is a hereditary right $n$-cotorsion pair in $\mathcal{A}$.
}}
\end{prop}
\begin{proof} $\mathrm{(1)}$ It follows from Theorem 3.12 (1), \cite[\text{Lemma}~4.1~(1)]{YJD2024} and Lemma 3.14 (1).

$\mathrm{(2)}$ It follows from Theorem 3.12 (2), \cite[\text{Lemma}~3.2~(1)]{JH2022} and Lemma 3.14 (2).
\end{proof}

\section{$n$-Cotorsion pairs over trivial ring extensions}

The goal of this section is to construct left $n$-cotorsion pairs and right $n$-cotorsion pairs over trivial ring extensions by investigating some special classes of modules over such rings. As an application, we immediately apply main results obtained over trivial ring extensions to Morita rings with zero bimodule homomorphisms. The research framework employed in this section is consistent with that adopted in the third section.

Throughout this section, let $\mathcal{C}$ be a class of left $R$-modules. Now, we define several special classes in $R\ltimes M\text{-}\mathrm{Mod}$ as follows.

(1) $\mathbf{T}(\mathcal{C}):=\{\mathbf{T}(C)\in R\ltimes M\text{-}\mathrm{Mod}\mid \forall~C\in\mathcal{C}\}$.

(2) $\mathbf{H}(\mathcal{C}):=\{\mathbf{H}(C)\in R\ltimes M\text{-}\mathrm{Mod}\mid \forall~C\in\mathcal{C}\}$.

(3) $\widetilde{\mathbf{T}}(\mathcal{C}):=\{(X,\alpha)\in R\ltimes M\text{-}\mathrm{Mod}\mid\text{there exists an exact sequence}~0\rightarrow\mathbf{T}(C_{1})\rightarrow(X,\alpha)\rightarrow\mathbf{T}(C_{2})
\rightarrow0,~\text{where}~C_{1}\in\mathcal{C}~\text{and}~C_{2}\in\mathcal{C}\}$.

(4) $\widetilde{\mathbf{H}}(\mathcal{C}):=\{[Y,\beta]\in R\ltimes M\text{-}\mathrm{Mod}\mid\text{there exists an exact sequence}~0\rightarrow\mathbf{H}(C_{1})\rightarrow[Y,\beta]\rightarrow\mathbf{H}(C_{2})
\rightarrow0,~\text{where}~C_{1}\in\mathcal{C}~\text{and}~C_{2}\in\mathcal{C}\}$.

(5) $\mathfrak{U}^{\mathcal{C}}:=\{(X,\alpha)\in R\ltimes M\text{-}\mathrm{Mod}\mid \forall~X\in\mathcal{C}\}=\{[Y,\beta]\in R\ltimes M\text{-}\mathrm{Mod}\mid \forall~Y\in\mathcal{C}\}$.

(6) $\mathfrak{B}^{\mathcal{C}}:=\{(X,\alpha)\in R\ltimes M\text{-}\mathrm{Mod}\mid\mathrm{Coker}\alpha\in
\mathcal{C},~\text{and the sequence}~M\otimes_{R}M\otimes_{R}X\overset{M\otimes_{R}\alpha}\longrightarrow M\otimes_{R}X\overset{\alpha}\longrightarrow X~\text{is exact}\}$.

(7) $\mathfrak{J}^{\mathcal{C}}:=\{[Y,\beta]\in R\ltimes M\text{-}\mathrm{Mod}\mid\mathrm{Ker}\beta\in
\mathcal{C},~\text{and the sequence}~Y\overset{\beta}\longrightarrow\mathrm{Hom}_{R}(M,Y)\overset{\mathrm{Hom}_{R}
(M,\beta)}\longrightarrow\mathrm{Hom}_{R}(M,\mathrm{Hom}_{R}(M,Y))~\text{is exact}\}$.

Clearly, we have $\mathbf{T}(\mathcal{C})\subseteq\widetilde{\mathbf{T}}(\mathcal{C})\cap\mathfrak{B}
^{\mathcal{C}}$ and $\mathbf{H}(\mathcal{C})\subseteq\widetilde{\mathbf{H}}(\mathcal{C})\cap\mathfrak{J}
^{\mathcal{C}}$.

{\bf 4.1 Construct left (right) $n$-cotorsion pairs over trivial ring extensions}

In this subsection, based on known homological formulas over trivial ring extensions,
we first investigate some classes of modules satisfying certain conditions over such rings, and then through these classes, we establish left $n$-cotorsion pairs and right $n$-cotorsion pairs over trivial ring extensions.

By observing \cite[\text{Lemma}~3.4]{L2023} and its proof, we directly give the following improved lemma, whose proof is omitted. By the way, its proof is similar to the proofs of Lemma 3.5 (2) and Lemma 3.6 (2), and can make use of adjoint pairs $(\mathbf{T},\mathbf{U})$ and $(\mathbf{U}', \mathbf{H})$ in Remark 2.7.
It is worth emphasizing that this improvement will also play a significant role in investigating the relationships among some special classes of modules over trivial ring extensions.

\begin{lem}\label{prop:2.4}{\it{Let $X$ be a left $R$-module, $(Y,\beta)$ a left $R\ltimes M$-module, and $n\geqslant1$ an integer.

$(1)$ If $\mathrm{Tor}_{i}^{R}(M,X)=0$ for each $1\leqslant i\leqslant n$, then $\mathrm{Ext}_{R\ltimes M}
^{i}(\mathbf{T}(X),(Y,\beta))\cong\mathrm{Ext}_{R}^{i}(X,Y)$.

$(2)$ If $\mathrm{Ext}_{R}^{i}(M,X)=0$ for each $1\leqslant i\leqslant n$, then $\mathrm{Ext}_{R\ltimes M}
^{i}((Y,\beta),\mathbf{H}(X))\cong\mathrm{Ext}_{R}^{i}(Y,X)$.
}}
\end{lem}

\begin{lem}\label{prop:2.4}{\it{For any integer $n\geqslant1$, the following statements hold.

$\mathrm{(1)}$ If $\mathrm{Ext}_{R}^{j}(M,C)=0$ for each $1\leqslant j\leqslant n$ and any $C\in\mathcal{C}$, then
$$\bigcap\limits_{i=1}\limits^{n}{^{\bot_{i}}(\mathbf{H}(\mathcal{C}))}=\mathfrak{U}^{^{\bot_{[1,n]}}\mathcal{C}}
=\bigcap\limits_{i=1}\limits^{n}{^{\bot_{i}}(\widetilde{\mathbf{H}}(\mathcal{C}))}.$$

$\mathrm{(2)}$ If $\mathrm{Tor}_{j}^{R}(M,C)=0$ for each $1\leqslant j\leqslant n$ and any $C\in\mathcal{C}$, then
$$\bigcap\limits_{i=1}\limits^{n}(\mathbf{T}(\mathcal{C}))^{\bot_{i}}=\mathfrak{U}^{\mathcal{C}
^{\bot_{[1,n]}}}=\bigcap\limits_{i=1}\limits^{n}(\widetilde{\mathbf{T}}(\mathcal{C}))^{\bot_{i}}.$$
}}
\end{lem}
\begin{proof} The proof of (2) is given below, and the proof of (1) is similar.

Let $(X,\alpha)\in\bigcap\limits_{i=1}\limits^{n}(\mathbf{T}(\mathcal{C}))^{\bot_{i}}$ and $C\in
\mathcal{C}$. By Lemma 4.1 (1), we have $\mathrm{Ext}_{R}^{i}(C,X)\cong\mathrm{Ext}_{R\ltimes M}^{i}(\mathbf{T}(C),(X,\alpha))=0$ for each $1\leqslant i\leqslant n$. Thus $X\in\mathcal{C}^{\bot_{[1,n]}}$. Then $(X,\alpha)\in\mathfrak{U}^{\mathcal{C}^{\bot_{[1,n]}}}$ and hence $\bigcap\limits_{i=1}\limits^{n}(\mathbf{T}
(\mathcal{C}))^{\bot_{i}}\subseteq\mathfrak{U}^{\mathcal{C}^{\bot_{[1,n]}}}$.

Next, let $(Y,\beta)\in\mathfrak{U}^{\mathcal{C}^{\bot_{[1,n]}}}$ and $(G,f)\in\widetilde{\mathbf{T}}
(\mathcal{C})$. Then there exists a short exact sequence
$$0\rightarrow\mathbf{T}(C_{1})\rightarrow(G,f)\rightarrow\mathbf{T}(C_{2})\rightarrow0$$
in $R\ltimes M$-Mod with $C_{1}$, $C_{2}\in\mathcal{C}$. Applying the functor $\mathrm{Hom}_{R\ltimes M}(-,(Y,\beta))$ to the above short exact sequence, we have an exact sequence
$$\mathrm{Ext}_{R\ltimes M}^{i}(\mathbf{T}(C_{2}),(Y,\beta))\rightarrow\mathrm{Ext}_{R\ltimes M}^{i}((G,f),(Y,\beta))\rightarrow\mathrm{Ext}_{R\ltimes M}^{i}(\mathbf{T}(C_{1}),(Y,\beta))$$
for each $1\leqslant i\leqslant n$. By Lemma 4.1  (1), we have $\mathrm{Ext}_{R\ltimes M}^{i}(\mathbf{T}
(C_{2}),(Y,\beta))\cong\mathrm{Ext}_{R}^{i}(C_{2},Y)=0$ and $\mathrm{Ext}_{R\ltimes M}^{i}(\mathbf{T}(C_{1}),
(Y,\beta))\cong\mathrm{Ext}_{R}^{i}(C_{1},Y)=0$. Then $\mathrm{Ext}_{R\ltimes M}^{i}((G,f),(Y,\beta))=0$ and hence $(Y,\beta)\in\bigcap\limits_{i=1}\limits^{n}(\widetilde{\mathbf{T}}(\mathcal{C}))^{\bot_{i}}$. Therefore, $\mathfrak{U}^{\mathcal{C}^{\bot_{[1,n]}}}\subseteq\bigcap\limits_{i=1}\limits^{n}(\widetilde{\mathbf{T}}
(\mathcal{C}))^{\bot_{i}}$.

Finally, let $(N,\tau)\in\bigcap\limits_{i=1}\limits^{n}(\widetilde{\mathbf{T}}(\mathcal{C}))^{\bot_{i}}$ and $(Z,\gamma)\in\mathbf{T}(\mathcal{C})$. Since $\mathbf{T}(\mathcal{C})\subseteq\widetilde{\mathbf{T}}
(\mathcal{C})$, we have $(Z,\gamma)\in\widetilde{\mathbf{T}}(\mathcal{C})$. Then we have $\mathrm{Ext}_{R\ltimes M}^{i}((Z,\gamma),(N,\tau))=0$ for each $1\leqslant i\leqslant n$. So $(N,\tau)\in\bigcap\limits_{i=1}\limits^{n}
(\mathbf{T}(\mathcal{C}))^{\bot_{i}}$. Thus $\bigcap\limits_{i=1}\limits^{n}(\widetilde
{\mathbf{T}}(\mathcal{C}))^{\bot_{i}}\subseteq\bigcap\limits_{i=1}\limits^{n}(\mathbf{T}(\mathcal{C}))^{\bot_{i}}$.
This completes the proof.
\end{proof}

\begin{lem}\label{prop:2.4}{\it{The following statements hold.

$\mathrm{(1)}$ Suppose that $\mathrm{Ext}^{i}_{R}(M,C)=0$ and $\mathrm{Ext}_{R}^{1}(\mathrm{Hom}_{R}(M,F),F)=0$ for each $1
\leqslant i\leqslant n$, any $C\in\mathcal{C}$ and any $F\in\mathcal{C}_{n-1}^{\wedge}$. Then $(\mathfrak{J}^{\mathcal{C}})_{n-1}^{\wedge}=\mathfrak{J}^{\mathcal{C}_{n-1}^{\wedge}}$.

$\mathrm{(2)}$ Suppose that $\mathrm{Tor}_{i}^{R}(M,C)=0$ and $\mathrm{Ext}_{R}^{1}(E,M\otimes_{R}E)=0$ for each $1\leqslant i\leqslant n$, any $C\in\mathcal{C}$ and any $E\in\mathcal{C}_{n-1}^{\vee}$. Then $(\mathfrak{B}^{\mathcal{C}})_{n-1}
^{\vee}=\mathfrak{B}^{\mathcal{C}_{n-1}^{\vee}}$.
}}
\end{lem}
\begin{proof} The proof of (2) is given below, and the proof of (1) is similar.

Let $(X,\alpha)\in(\mathfrak{B}^{\mathcal{C}})_{n-1}^{\vee}$. Then there exists an exact sequence
$$0\rightarrow(X,\alpha)\rightarrow(X_{0},\alpha_{0})\rightarrow(X_{1},\alpha_{1})\rightarrow\cdots\rightarrow
(X_{n-1},\alpha_{n-1})\rightarrow0$$
in $R\ltimes M$-Mod with each $(X_{i},\alpha_{i})\in\mathfrak{B}^{\mathcal{C}}$. By \cite[\text{Theorem}~3.4~(1)]{L2024}, we have
$$\mathrm{Tor}_{i}^{R\ltimes M}(\mathbf{Z}(R),(Y,\beta))
\cong\mathrm{Tor}_{i}^{R}(R,\mathrm{Coker}\beta)=0$$
for each $1\leqslant i\leqslant n$ and any $(Y,\beta)\in\mathfrak{B}^{\mathcal{C}}$. By \cite[\text{Lemma}~3.2~(2)]{TX2025}, for any $(Z,\gamma)\in(\mathfrak{B}^{\mathcal{C}})_{n-1}
^{\vee}$, we have $\mathrm{Tor}_{1}^{R\ltimes M}(\mathbf{Z}(R),(Z,\gamma))=0$. Therefore, by Proposition~3.8~(1), there exists an exact sequence
$$0\rightarrow\mathbf{Z}(R)\otimes_{R\ltimes M}(X,\alpha)\rightarrow\mathbf{Z}(R)\otimes_{R\ltimes M}(X_{0},\alpha_{0})\rightarrow\mathbf{Z}(R)\otimes_{R\ltimes M}(X_{1},\alpha_{1})\rightarrow$$
$$\cdots\rightarrow
\mathbf{Z}(R)\otimes_{R\ltimes M}(X_{n-1},\alpha_{n-1})\rightarrow0$$
in $R\ltimes M$-Mod. By \cite[\text{Lemma}~3.2~(5)]{L2024}, for each $0\leqslant j\leqslant n-1$, we have $$\mathbf{Z}(R)\otimes_{R\ltimes M}(X_{j},\alpha_{j})\cong R\otimes_{R}\mathrm{Coker}\alpha_{j}
\cong\mathrm{Coker}\alpha_{j},$$
$$\mathbf{Z}(R)\otimes_{R\ltimes M}(X,\alpha)\cong R\otimes_{R}\mathrm{Coker}\alpha\cong\mathrm{Coker}\alpha.$$
Thus there is an exact sequence
$$0\rightarrow\mathrm{Coker}\alpha\rightarrow\mathrm{Coker}\alpha_{0}
\rightarrow\mathrm{Coker}\alpha_{1}
\rightarrow\cdots\rightarrow\mathrm{Coker}\alpha_{n-1}\rightarrow0$$
in $R$-Mod with each $\mathrm{Coker}\alpha_{j}\in\mathcal{C}$. Then $\mathrm{Coker}
\alpha\in\mathcal{C}_{n}^{\vee}$. By assumption, Proposition~3.8~(1) and Remark 3.10, we obtain the following exact sequence
$$0\rightarrow M\otimes_{R}\mathrm{Coker}\alpha\rightarrow M\otimes_{R}\mathrm{Coker}\alpha_{0}
\rightarrow M\otimes_{R}\mathrm{Coker}\alpha_{1}
\rightarrow\cdots\rightarrow M\otimes_{R}\mathrm{Coker}\alpha_{n-1}\rightarrow0$$
in $R$-Mod.
By \cite[\text{Lemma}~3.2~(3)]{L2024}, we have the following commutative diagram with exact rows and exact columns:
$$\small\xymatrix@C=0.8cm@R=0.8cm{
&&0 \ar[d]_{}&0 \ar[d]_{}& \\
&& M\otimes_{R}\mathrm{Coker}\alpha \ar[d]_{} \ar[r]^{\quad\quad\quad\kappa} & X \ar[d]_{} \ar[r]^{} &\mathrm{Coker}\alpha \ar[d]_{}\ar[r]^{} & 0 \\
&0 \ar[r]^{} & M\otimes_{R}\mathrm{Coker}\alpha_{0} \ar[r]^{\quad\quad\quad\kappa_{0}} &  X_{0} \ar[r]^{} & \mathrm{Coker}\alpha_{0} \ar[r]^{} & 0  \\   }\vspace*{2mm}$$
Since the first square in the above diagram is commutative, $\kappa$ is a monomorphism.
By \cite[\text{Lemma}~3.2~(3)]{L2024}, the sequence
$$M\otimes_{R}M\otimes_{R}X\overset{M\otimes_{R}\alpha}\longrightarrow M\otimes_{R}X\overset{\alpha}\longrightarrow X$$
is exact. Then $(X,\alpha)\in\mathfrak{B}^{\mathcal{C}_{n-1}^{\vee}}$ and hence $(\mathfrak{B}^{\mathcal{C}})_{n-1}
^{\vee}\subseteq\mathfrak{B}^{\mathcal{C}_{n-1}^{\vee}}$.

Let $(G,\tau)\in\mathfrak{B}^{\mathcal{C}_{n-1}^{\vee}}$. Then $\mathrm{Coker}\tau\in\mathcal{C}_{n-1}^{\vee}$ and there exists a short exact sequence
$$0\rightarrow M\otimes_{R}\mathrm{Coker}\tau\overset{\kappa}\rightarrow G\overset{\rho}\rightarrow\mathrm{Coker}\tau\rightarrow0 \eqno(\ast)$$
in $R$-Mod such that $\tau=\kappa(M\otimes_{R}\rho)$ by \cite[\text{Lemma}~3.2~(3)]{L2024}. Then there exists an exact sequence
$$0\rightarrow\mathrm{Coker}\tau\rightarrow C_{0}\rightarrow C_{1}\rightarrow C_{2}\rightarrow\cdots
\rightarrow C_{n-1}\rightarrow0$$
in $R$-Mod with each $C_{j}\in\mathcal{C}$. By assumption, Proposition~3.8~(1) and Remark 3.10, we obtain an exact sequence
$$0\rightarrow M\otimes_{R}\mathrm{Coker}\tau\rightarrow M\otimes_{R}C_{0}\rightarrow M\otimes_{R}C_{1}\rightarrow
\cdots\rightarrow M\otimes_{R}C_{n-1}\rightarrow0$$
in $R$-Mod. Thus there exists an exact sequence
$$0\rightarrow(\mathrm{Coker}\tau\oplus(M\otimes_{R}\mathrm{Coker}\tau),\left(\begin{smallmatrix}0&0\\1&0
\end{smallmatrix}\right))\rightarrow(C_{0}\oplus(M\otimes_{R}C_{0}),\left(\begin{smallmatrix}0&0\\1&0
\end{smallmatrix}\right))\rightarrow$$
$$(C_{1}\oplus(M\otimes_{R}C_{1}),\left(\begin{smallmatrix}0&0\\1&0
\end{smallmatrix}\right))\rightarrow\cdots\rightarrow(C_{n-1}\oplus(M\otimes_{R}C_{n-1}),\left(\begin{smallmatrix}
0&0\\1&0\end{smallmatrix}\right))\rightarrow0$$
in $R\ltimes M$-Mod with each $(C_{j}\oplus(M\otimes_{R}C_{j}),\left(\begin{smallmatrix}0&0\\1&0\end{smallmatrix}
\right))=\mathbf{T}(C_{j})\in\mathfrak{B}^{\mathcal{C}}$. Therefore, $$(\mathrm{Coker}\tau\oplus(M\otimes_{R}\mathrm{Coker}\tau),\left(\begin{smallmatrix}0&0\\1&0
\end{smallmatrix}\right))=\mathbf{T}(\mathrm{Coker}\tau)
\in(\mathfrak{B}^{\mathcal{C}})_{n-1}^{\vee}.$$
Since the exact sequence $(\ast)$ is split by assumption, there exist $\theta:G\rightarrow M\otimes_{R}\mathrm{Coker}\tau$ and $\eta:\mathrm{Coker}\tau\rightarrow G$ such that $\theta\kappa=1_{M\otimes_{R}\mathrm{Coker}\tau}$, $\rho\eta=1_{\mathrm{Coker}\tau}$, $\kappa\theta+\eta\rho=1_{G}$ and $\theta\eta=0$. Then we get the isomorphism $\xi:G\rightarrow \mathrm{Coker}\tau\oplus(M\otimes_{R}\mathrm{Coker}\tau)$ defined by $\xi(x)=(\rho(x),\theta(x))$ for any $x\in G$ in $R$-Mod. Consider the following diagram:
$$\small\xymatrix@C=2cm@R=0.9cm{
      & M\otimes_{R}G \ar[d]_{\tau} \ar[r]^{M\otimes_{R}\xi\quad\quad\quad\quad\quad} & M\otimes_{R}(\mathrm{Coker}\tau\oplus(M\otimes_{R}\mathrm{Coker}\tau)) \ar[d]^{\left(\begin{smallmatrix}0&0\\1&0\end{smallmatrix}\right)}  & \\
      & G  \ar[r]^{\xi\quad\quad\quad\quad\quad} &  \mathrm{Coker}\tau\oplus(M\otimes_{R}\mathrm{Coker}\tau)  &  \\   }\vspace*{2mm}$$
Then $\xi\tau=\left(\begin{smallmatrix}\rho\\ \theta\end{smallmatrix}\right)\tau=\left(\begin{smallmatrix}0\\ \theta\tau\end{smallmatrix}\right)=\left(\begin{smallmatrix}0\\ \theta\kappa(M\otimes_{R}\rho)\end{smallmatrix}\right)=\left(\begin{smallmatrix}0\\ M\otimes_{R}\rho\end{smallmatrix}\right)=
\left(\begin{smallmatrix}0&0\\1&0\end{smallmatrix}\right)\left(\begin{smallmatrix}M\otimes_{R}\rho\\ M\otimes_{R}\theta\end{smallmatrix}\right)
=\left(\begin{smallmatrix}0&0\\1&0\end{smallmatrix}\right)M\otimes_{R}\xi$ and hence the above diagram is commutative.
Note that
$$(\eta,\kappa)\left(\begin{smallmatrix}\rho\\ \theta\end{smallmatrix}\right)=\eta\rho+\kappa\theta=1_{G}, \left(\begin{smallmatrix}\rho\\ \theta\end{smallmatrix}\right)(\eta,\kappa)=\left(\begin{smallmatrix}1&0\\0&1\end{smallmatrix}
\right).$$
Then $\xi$ is an isomorphism in $R\ltimes M$-Mod.
So $(G,\tau)\cong\mathbf{T}(\mathrm{Coker}\tau)$. Then $(G,\tau)\in(\mathfrak{B}^{\mathcal{C}})_{n-1}^{\vee}$ and hence
$\mathfrak{B}^{\mathcal{C}_{n-1}^{\vee}}\subseteq(\mathfrak{B}^{\mathcal{C}})_{n-1}^{\vee}$.
This completes the proof.
\end{proof}

Now, based on Lemma 4.2, Lemma 4.3 and other known results, we construct left (resp. right) $n$-cotorsion pairs in $R\ltimes M$-Mod in terms of left (resp. right) $n$-cotorsion pairs in $R$-Mod. The following theorem is another main result of this paper.

\begin{thm}\label{prop:2.4}{\it{$\mathrm{(1)}$ Let $(\mathcal{C},\mathcal{D})$ be a left $n$-cotorsion pair in $R$-$\mathrm{Mod}$. Assume that $\mathrm{Ext}^{i}_{R}(M,D)=0$ and $\mathrm{Ext}_{R}^{1}(\mathrm{Hom}_{R}(M,F),F)=0$ for each $1\leqslant i\leqslant n$, any $D\in\mathcal{D}$ and any $F\in\mathcal{D}_{n-1}^{\wedge}$, $\mathfrak{U}^{\mathcal{C}}$ is a resolving class, and $(\mathfrak{J}^{\mathcal{D}})_{n-1}^{\wedge}$ is closed under extensions. Then $(\mathfrak{U}^{\mathcal{C}},
\mathfrak{J}^{\mathcal{D}})$ is a left $n$-cotorsion pair in $R\ltimes M$-$\mathrm{Mod}$.

$\mathrm{(2)}$ Let $(\mathcal{C},\mathcal{D})$ be a right $n$-cotorsion pair in $R$-$\mathrm{Mod}$. Assume that $\mathrm{Tor}_{i}^{R}(M,C)=0$ and $\mathrm{Ext}_{R}^{1}(E,M\otimes_{R}E)=0$ for each $1\leqslant i\leqslant n$, any $C\in\mathcal{C}$ and any $E\in\mathcal{C}_{n-1}^{\vee}$, $\mathfrak{U}^{\mathcal{D}}$ is a coresolving class, and $(\mathfrak{B}^{\mathcal{C}})_{n-1}^{\vee}$ is closed under extensions. Then $(\mathfrak{B}^{\mathcal{C}},
\mathfrak{U}^{\mathcal{D}})$ is a right $n$-cotorsion pair in $R\ltimes M$-$\mathrm{Mod}$.
}}
\end{thm}
\begin{proof} The proof of (1) is given below, and the proof of (2) is similar.

Since $(\mathcal{C},\mathcal{D})$ is a left $n$-cotorsion pair in $R$-Mod, it follows from
\cite[\text{Theorem}~2.7]{MOM2021} that $\mathcal{C}=\bigcap\limits_{i=1}\limits^{n}{^{\bot_{i}}
\mathcal{D}}$. Because $\mathcal{D}\subseteq\mathcal{D}_{n-1}^{\wedge}$, we get $\mathfrak{J}^{\mathcal{D}}=\mathbf{H}(\mathcal{D})$ by assumption and \cite[\text{Lemma}~4.6~(2)]{L2024}. By Lemma 4.2 (1), we have
$$\textstyle\bigcap\limits_{i=1}\limits^{n}{^{\bot_{i}}(\mathfrak{J}^{\mathcal{D}})}
=\textstyle\bigcap\limits_{i=1}\limits^{n}{^{\bot_{i}}(\mathbf{H}(\mathcal{D}))}=\mathfrak{U}
^{^{\bot_{[1,n]}}\mathcal{D}}=\mathfrak{U}^{\mathcal{C}}.$$
Note that every projective left $R\ltimes M$-module is contained in $\mathfrak{U}^{\mathcal{C}}$. So $\mathfrak{U}
^{\mathcal{C}}$ is projectively resolving class. By \cite[\text{Proposition}~2.5]{MOM2021}, we have $\mathrm{Ext}
_{R\ltimes M}^{1}(\mathfrak{U}^{\mathcal{C}},(\mathfrak{J}^{\mathcal{D}})_{n-1}^{\wedge})=0$. That is, $(\mathfrak{J}
^{\mathcal{D}})_{n-1}^{\wedge}\subseteq(\mathfrak{U}^{\mathcal{C}})^{\bot}$. Let $[X,\alpha]\in R\ltimes M$-Mod. By \cite[\text{Lemma}~3.2~(2)]{L2024}, there exists a short exact sequence
$$0\rightarrow[\mathrm{Ker}\alpha,0]\rightarrow[X,\alpha]\rightarrow[\mathrm{Im}\alpha,0]
\rightarrow0$$
in $R\ltimes M$-Mod. Since $(\mathcal{C},\mathcal{D})$ is a left $n$-cotorsion pair in $R$-Mod, there exists a short exact sequence $0\rightarrow D\overset{g}\rightarrow C\overset{f}\rightarrow\mathrm{Im}\alpha
\rightarrow0$ in $R$-Mod with $C\in\mathcal{C}$ and $D\in\mathcal{D}_{n-1}^{\wedge}$. This yields the short exact sequence
$$0\longrightarrow\mathrm{Hom}_{R}(M,D)\oplus D\overset{\left(\begin{smallmatrix}1&0\\0&g
\end{smallmatrix}\right)}\longrightarrow\mathrm{Hom}_{R}(M,D)\oplus C\overset{(0,f)}\longrightarrow\mathrm{Im}\alpha\longrightarrow0.$$
Define
$$\beta:\mathrm{Hom}_{R}(M,D)\oplus C\rightarrow\mathrm{Hom}_{R}(M,\mathrm{Hom}_{R}(M,D)\oplus C)$$
by $\beta(a,c)=(0,g_{\ast}(a))$ for any $a\in\mathrm{Hom}_{R}(M,D)$ and any $c\in C$, where $g_{\ast}=\mathrm{Hom}_{R}(M,g)$. Then we have the following commutative diagram with exact rows:
$$\small\xymatrix@C=0.7cm@R=0.7cm{
& \mathrm{Hom}_{R}(M,D)\oplus D \ar[d]_{\left(\begin{smallmatrix}0&0\\1&0
\end{smallmatrix}\right)} \ar[r]^{\left(\begin{smallmatrix}1&0\\0&g
\end{smallmatrix}\right)} & \mathrm{Hom}_{R}(M,D)\oplus C \ar[d]_{\beta} \ar[r]^{(0,f)} & \mathrm{Im}\alpha\ar[d]_{0} &\\
& \mathrm{Hom}_{R}(M,\mathrm{Hom}_{R}(M,D)\oplus D)\ar[r]_{~\left(\begin{smallmatrix}1&0\\0&g
\end{smallmatrix}\right)_{\ast}} &  \mathrm{Hom}_{R}(M,\mathrm{Hom}_{R}(M,D)\oplus C) \ar[r]_{\quad\quad\quad(0,f)_{\ast}} & \mathrm{Hom}_{R}(M,\mathrm{Im}\alpha) & \\   }\vspace*{2mm}$$
where $\left(\begin{smallmatrix}1&0\\0&g\end{smallmatrix}\right)_{\ast}=\mathrm{Hom}_{R}(M,\left
(\begin{smallmatrix}1&0\\0&g\end{smallmatrix}\right))$ and $(0,f)_{\ast}=\mathrm{Hom}_{R}(M,(0,f))$. Thus there exists a short exact sequence
$$0\longrightarrow [\mathrm{Hom}_{R}(M,D)\oplus D,\left(\begin{smallmatrix}0&0\\1&0
\end{smallmatrix}\right)]\overset{\left(\begin{smallmatrix}1&0\\0&g\end{smallmatrix}\right)}\longrightarrow[\mathrm{Hom}_{R}(M,D)\oplus C,\beta]\overset{(0,f)}\longrightarrow[\mathrm{Im}\alpha,0]\longrightarrow0$$
in $R\ltimes M$-Mod. By Lemma 4.3 (1), we have
$$[\mathrm{Hom}
_{R}(M,D)\oplus D,\left(\begin{smallmatrix}0&0\\1&0\end{smallmatrix}\right)]=\mathbf{H}(D)\in\mathfrak{J}^{\mathcal{D}}\subseteq\mathfrak{J}
^{\mathcal{D}_{n-1}^{\wedge}}=(\mathfrak{J}^{\mathcal{D}})_{n-1}^{\wedge}\subseteq(\mathfrak{U}^{\mathcal{C}})
^{\bot}.$$
Since $\mathrm{Ext}_{R}^{1}(\mathrm{Hom}_{R}(M,\mathcal{D}_{n-1}^{\wedge}),\mathcal{D}
_{n-1}^{\wedge})=0$ and $\mathrm{Hom}_{R}(M,D)\in\mathrm{Hom}_{R}(M,\mathcal{D}_{n-1}^{\wedge})$, it follows from \cite[\text{Proposition}~2.5]{MOM2021} that $\mathrm{Hom}_{R}(M,D)\in\bigcap\limits_{i=1}\limits^{n}{^{\bot_{i}}
\mathcal{D}}=\mathcal{C}$. Then $\mathrm{Hom}_{R}(M,D)\oplus C\in\mathcal{C}$, and hence $[\mathrm{Hom}_{R}(M,D)
\oplus C,\beta]\in\mathfrak{U}^{\mathcal{C}}$. Therefore, $[\mathrm{Im}\alpha,0]$ admits a special $\mathfrak{U}
^{\mathcal{C}}$-precover. Since $(\mathcal{C},\mathcal{D})$ is a left $n$-cotorsion pair in $R$-Mod, there also exists a short exact sequence $0\rightarrow D_{0}\overset{g_{0}}\rightarrow C_{0}\overset{f_{0}}\rightarrow\mathrm{Ker}\alpha\rightarrow0$ in $R$-Mod with $C_{0}\in
\mathcal{C}$ and $D_{0}\in\mathcal{D}_{n-1}^{\wedge}$. Define $$\beta_{0}:\mathrm{Hom}_{R}(M,D_{0})
\oplus C_{0}\rightarrow\mathrm{Hom}_{R}(M,\mathrm{Hom}_{R}(M,D_{0})\oplus C_{0})$$
by $\beta_{0}(a,c)=
(0,{g_{0}}_{\ast}(a))$ for any $a\in\mathrm{Hom}_{R}(M,D_{0})$ and any $c\in C_{0}$, where ${g_{0}}_{\ast}=
\mathrm{Hom}_{R}(M,g_{0})$. Similarly, one can obtain that $[\mathrm{Ker}\alpha,0]$ also admits a special $\mathfrak{U}
^{\mathcal{C}}$-precover. That is, there exists a short exact sequence
$$0\longrightarrow [\mathrm{Hom}_{R}(M,D_{0})\oplus D_{0},\left(\begin{smallmatrix}0&0\\1&0
\end{smallmatrix}\right)]\longrightarrow[\mathrm{Hom}_{R}(M,D_{0})\oplus C_{0},\beta_{0}]\longrightarrow[\mathrm{Ker}\alpha,0]\longrightarrow0$$
in $R\ltimes M$-Mod with $[\mathrm{Hom}_{R}(M,D_{0})\oplus C_{0},\beta_{0}]\in\mathfrak{U}
^{\mathcal{C}}$ and $[\mathrm{Hom}_{R}(M,D_{0})\oplus D_{0},\left(\begin{smallmatrix}0&0\\1&0
\end{smallmatrix}\right)]\in(\mathfrak{J}^{\mathcal{D}})_{n-1}^{\wedge}$. By \cite[\text{Lemma}~5.4~(2)]{L2020}, there is the following commutative diagram with exact rows and columns:
$$\small\xymatrix@R=1cm@C=0.7cm{
   &0 \ar[d]_{}                       &0 \ar[d]_{}                         &0 \ar[d]_{} &   \\
  0  \ar[r]^{} &{[\mathrm{Hom}_{R}(M,D_{0})\oplus D_{0},\left(\begin{smallmatrix}0&0\\1&0
\end{smallmatrix}\right)]}  \ar[d]_{} \ar[r]^{}  &[Z,\psi] \ar[d]_{} \ar[r]^{}  &{[\mathrm{Hom}_{R}(M,D)\oplus D,\left(\begin{smallmatrix}0&0\\1&0\end{smallmatrix}\right)]} \ar[d]_{} \ar[r]^{}        &0 \\
 0 \ar[r]^{}      &[\mathrm{Hom}_{R}(M,D_{0})\oplus C_{0},\beta_{0}] \ar[d]_{} \ar[r]^{}  &[Y,\varphi] \ar[d]_{} \ar[r]^{}  &[\mathrm{Hom}_{R}(M,D)\oplus C,\beta] \ar[d]_{} \ar[r]^{}        &0  \\
  0  \ar[r]^{}     &[\mathrm{Ker}\alpha,0] \ar[d]_{} \ar[r]^{}   &[X,\alpha] \ar[d]_{} \ar[r]^{}   &[\mathrm{Im}\alpha,0] \ar[d]_{} \ar[r]^{}        &0  \\
  &0 &0 &0 &\\
   }\vspace*{2mm}$$
with $[Y,\varphi]\in\mathfrak{U}^{\mathcal{C}}$. Since $(\mathfrak{J}^{\mathcal{D}})_{n-1}^{\wedge}$ is closed under extensions, we have $[Z,\psi]\in(\mathfrak{J}^{\mathcal{D}})_{n-1}^{\wedge}$. Therefore, by
\cite[\text{Theorem}~2.7]{MOM2021}, $(\mathfrak{U}^{\mathcal{C}},\mathfrak{J}^{\mathcal{D}})$ is a left $n$-cotorsion pair in $R\ltimes M$-Mod.
\end{proof}

\begin{rem}\label{prop:2.4}{\rm{$\mathrm{(1)}$ In Theorem 4.4 (1), since  $\mathfrak{U}^{\mathcal{C}}$ is a resolving class, we have that the left $n$-cotorsion pair $(\mathfrak{U}^{\mathcal{C}},
\mathfrak{J}^{\mathcal{D}})$ in $R\ltimes M$-$\mathrm{Mod}$ is hereditary by Lemma 3.14 (1).

$\mathrm{(2)}$ In Theorem 4.4 (2), since $\mathfrak{U}^{\mathcal{D}}$ is a coresolving class, we have that the right $n$-cotorsion pair $(\mathfrak{B}^{\mathcal{C}},
\mathfrak{U}^{\mathcal{D}})$ in $R\ltimes M$-$\mathrm{Mod}$ is hereditary by Lemma 3.14 (2).
}}
\end{rem}

Next, we construct the left $n$-cotorsion pairs and the right $n$-cotorsion pairs in $R$-$\mathrm{Mod}$ through the left $n$-cotorsion pairs and the right $n$-cotorsion pairs in $R\ltimes M$-$\mathrm{Mod}$, respectively.

\begin{thm}\label{prop:2.4}{\it{Let $\mathcal{D}$ be a class of left $R$-modules. The following statements hold.

$\mathrm{(1)}$ Let $(\mathfrak{U}^{\mathcal{C}},\mathfrak{J}^{\mathcal{D}})$ be a left $n$-cotorsion pair in $R\ltimes M$-$\mathrm{Mod}$. Assume that $\mathrm{Ext}^{i}_{R}(M,D)=0$, $\mathrm{Ext}_{R}^{1}(\mathrm{Hom}_{R}(M,D),D)
=0$ and $\mathrm{Hom}_{R}(M,D)\in\mathcal{D}$ for each $1\leqslant i\leqslant n$ and any $D\in\mathcal{D}$, and $\mathcal{D}_{n-1}^{\wedge}$ is closed under extensions. Then $(\mathcal{C},\mathcal{D})$ is a left $n$-cotorsion pair in $R$-$\mathrm{Mod}$.

$\mathrm{(2)}$ Let $(\mathfrak{B}^{\mathcal{C}},\mathfrak{U}^{\mathcal{D}})$ be a right $n$-cotorsion pair in $R\ltimes M$-$\mathrm{Mod}$. Assume that $\mathrm{Tor}_{i}^{R}(M,C)=0$, $\mathrm{Ext}_{R}^{1}(C,M\otimes_{R}C)=0$ and $M\otimes_{R}C\in\mathcal{C}$ for each $1\leqslant i\leqslant n$ and any $C\in\mathcal{C}$, and $\mathcal{C}_{n-1}^{\vee}$ is closed under extensions. Then $(\mathcal{C},\mathcal{D})$ is a right $n$-cotorsion pair in $R$-$\mathrm{Mod}$.
}}
\end{thm}
\begin{proof} The proof of (1) is given below, and the proof of (2) is similar.

Let $C\in\mathcal{C}$ and $D\in\mathcal{D}$. For each $1\leqslant i\leqslant n$, it follows from Lemma 4.1 (2) that $\mathrm{Ext}_{R}^{i}(C,D)\cong\mathrm{Ext}_{R\ltimes M}^{i}([C,0],\mathbf{H}(D))=0$. Then $\mathcal{C}\subseteq\bigcap\limits_{i=1}\limits^{n}{^{\bot_{i}}\mathcal{D}}$. Let $N\in\bigcap\limits_{i=1}
\limits^{n}{^{\bot_{i}}\mathcal{D}}$. By Lemma 4.2 (1) and \cite[\text{Lemma}~4.6~(2)]{L2024}, we have $$[N,0]
\in\mathfrak{U}^{^{\bot_{[1,n]}}\mathcal{D}}=\bigcap\limits_{i=1}\limits^{n}{^{\bot_{i}}
(\mathbf{H}(\mathcal{D}))}=\bigcap\limits_{i=1}\limits^{n}{^{\bot_{i}}(\mathfrak{J}^{\mathcal{D}})}=
\mathfrak{U}^{\mathcal{C}},$$
which implies $N\in\mathcal{C}$. Therefore, $\mathcal{C}=\bigcap\limits_{i=1}
\limits^{n}{^{\bot_{i}}\mathcal{D}}$.

Let $X\in R$-$\mathrm{Mod}$. Since $(\mathfrak{U}^{\mathcal{C}},
\mathfrak{J}^{\mathcal{D}})$ is a left $n$-cotorsion pair in $R\ltimes M$-$\mathrm{Mod}$, there exists a short exact sequence
$$0\rightarrow[Z,\gamma]\rightarrow[Y,\beta]\rightarrow[X,0]\rightarrow0$$
in $R\ltimes M$-$\mathrm{Mod}$ with $[Y,\beta]\in\mathfrak{U}^{\mathcal{C}}$ and $[Z,\gamma]\in(\mathfrak{J}
^{\mathcal{D}})_{n-1}^{\wedge}$. Moreover, it follows from Lemma 4.3 (1) that $[Z,\gamma]\in(\mathfrak{J}
^{\mathcal{D}})_{n-1}^{\wedge}\subseteq\mathfrak{J}^{\mathcal{D}_{n-1}^{\wedge}}$. By \cite[\text{Lemma}~3.2~(4)]{L2024}, we obtain a short exact sequence
$$0\rightarrow\mathrm{Ker}\gamma\rightarrow Z\rightarrow\mathrm{Hom}_{R}(M,\mathrm{Ker}\gamma)\rightarrow0$$
in $R$-$\mathrm{Mod}$ with $\mathrm{Ker}\gamma\in\mathcal{D}_{n-1}^{\wedge}$. By assumption, \cite[\text{Lemma}~2.4]{MOM2021} and Proposition 3.8(2), we obtain $\mathrm{Hom}_{R}(M,\mathrm{Ker}\gamma)\in\mathcal{D}_{n-1}^{\wedge}$. Since $\mathcal{D}_{n-1}^{\wedge}$ is closed under extensions, we get $Z\in\mathcal{D}_{n-1}^{\wedge}$. Thus there exists a short exact sequence $0\rightarrow Z\rightarrow Y\rightarrow X\rightarrow0$ in $R$-$\mathrm{Mod}$ with $Y\in\mathcal{C}$ and $Z\in
\mathcal{D}_{n-1}^{\wedge}$. By \cite[\text{Theorem}~2.7]{MOM2021}, we obtain that $(\mathcal{C},\mathcal{D})$ is a left $n$-cotorsion pair in $R$-$\mathrm{Mod}$.
\end{proof}

\begin{rem}\label{prop:2.4}{\rm{In the case that keep the conditions as Theorem 4.4 or Theorem 4.6, according to \cite[\text{Proposition}~3.1]{MOM2021} and its dual, one obtains that $\mathfrak{U}^{\mathcal{C}}$, $\mathfrak{B}^{\mathcal{C}}$ and $\mathcal{C}$ are special precovering classes, and $\mathfrak{J}^{\mathcal{D}}$, $\mathfrak{U}^{\mathcal{D}}$ and $\mathcal{D}$ are special preenveloping classes.
}}
\end{rem}

Similarly, we know that in general, Theorem 4.4 seems to be more important than Theorem 4.6.
From now on, we characterize the heredity of left $n$-cotorsion pairs and right $n$-cotorsion pairs in $R\ltimes M$-$\mathrm{Mod}$ and $R$-$\mathrm{Mod}$.
In the following, we present a necessary and sufficient condition for $\mathfrak{U}^
{\mathcal{C}}$ to be a projectively resolving class and a necessary and sufficient condition for $\mathfrak{U}^
{\mathcal{D}}$ to be an injectively coresolving class, which is helpful for the study of the heredity.

\begin{lem}\label{prop:2.4}{\it{
$\mathrm{(1)}$ Let $(\mathcal{C},\mathcal{D})$ be a left $n$-cotorsion pair in $R$-$\mathrm{Mod}$. If $\mathrm{Ext}_{R}
^{i}(M,D)=0$ for each $1\leqslant i\leqslant n$ and any $D\in\mathcal{D}$, then $\mathfrak{U}^
{\mathcal{C}}$ is a projectively resolving class in $R\ltimes M$-$\mathrm{Mod}$ if and only if $(\mathcal{C},\mathcal{D})$ is hereditary.

$\mathrm{(2)}$ Let $(\mathcal{C},\mathcal{D})$ be a right $n$-cotorsion pair in $R$-$\mathrm{Mod}$. If $\mathrm{Tor}^{R}
_{i}(M,C)=0$ for each $1\leqslant i\leqslant n$ and any $C\in\mathcal{C}$, then $\mathfrak{U}^
{\mathcal{D}}$ is an injectively coresolving class in $R\ltimes M$-$\mathrm{Mod}$ if and only if $(\mathcal{C},\mathcal{D})$ is hereditary.
}}
\end{lem}
\begin{proof} The proof of (1) is given below, and the proof of (2) is similar.

``~$\Rightarrow$~'' Since $\mathfrak{U}^{\mathcal{C}}$ is a projectively resolving class, it follows from \cite[\text{Lemma}~4.3~(3)]{L2024} that $\mathcal{C}$ is closed under kernels of epimorphisms. By Lemma 3.14 (1), we obtain that $(\mathcal{C},\mathcal{D})$ is hereditary.

``~$\Leftarrow$~'' Since $(\mathcal{C},\mathcal{D})$ is a left $n$-cotorsion pair in $R$-$\mathrm{Mod}$, $\mathcal{C}$ is closed under extensions. Because $(\mathcal{C},\mathcal{D})$ is hereditary, we get that $\mathcal{C}$ is closed under kernels of epimorphisms by Lemma~3.14 (1). By \cite[\text{Lemma}~4.3~(3)]{L2024}, $\mathfrak{U}^{\mathcal{C}}$ is closed under both extensions and kernels of epimorphisms. Take an arbitrary projective module $(X,\alpha)$ in $R\ltimes M$-$\mathrm{Mod}$. According to \cite[\text{Corollary}~1.6~(c)]{RPI1975}, $\mathrm{Coker}\alpha$ is a projective left $R$-module and the sequence
$$M\otimes_{R}M\otimes_{R}X\overset{M\otimes_{R}\alpha}\longrightarrow M\otimes_{R}X\overset{\alpha}\longrightarrow X$$
is exact. Naturally, $\mathrm{Coker}\alpha\in\mathcal{C}$. By \cite[\text{Lemma}~3.2~(3)]{L2024}, consider the following short exact sequence
$$0\rightarrow M\otimes_{R}\mathrm{Coker}\alpha\rightarrow X\rightarrow\mathrm{Coker}\alpha\rightarrow0$$
in $R$-$\mathrm{Mod}$. For any $D\in\mathcal{D}$ and each $1\leqslant i\leqslant n$, by \cite[\text{Exercise}~9.20]{JJ2009}, we have
$$\mathrm{Ext}_{R}
^{i}(M\otimes_{R}\mathrm{Coker}\alpha,D)\cong\mathrm{Hom}_{R}(\mathrm{Coker}\alpha,\mathrm{Ext}_{R}
^{i}(M,D))=0.$$
Then $M\otimes_{R}\mathrm{Coker}\alpha\in\bigcap\limits_{i=1}\limits^{n}{^{\bot_{i}}
\mathcal{D}}=\mathcal{C}$. Since $\mathcal{C}$ is closed under extensions, we obtain $X\in\mathcal{C}$. This implies $(X,\alpha)\in\mathfrak{U}^{\mathcal{C}}$. So $\mathfrak{U}^{\mathcal{C}}$ is a projectively resolving class in $R\ltimes M$-$\mathrm{Mod}$.
\end{proof}

\begin{prop}\label{prop:2.4}{\it{$\mathrm{(1)}$ Let $(\mathcal{C},\mathcal{D})$ be a hereditary left $n$-cotorsion pair in $R$-$\mathrm{Mod}$. Assume that $\mathrm{Ext}^{i}_{R}(M,D)=0$ and $\mathrm{Ext}_{R}^{1}(\mathrm{Hom}_{R}(M,F),F)=0$ for each $1\leqslant i\leqslant n$, any $D\in\mathcal{D}$ and any $F\in\mathcal{D}_{n-1}^{\wedge}$, and $(\mathfrak{J}^{\mathcal{D}})_{n-1}^{\wedge}$ is closed under extensions. Then $(\mathfrak{U}^{\mathcal{C}},\mathfrak{J}^{\mathcal{D}})$ is a hereditary left $n$-cotorsion pair in $R\ltimes M$-$\mathrm{Mod}$.

$\mathrm{(2)}$ Let $(\mathcal{C},\mathcal{D})$ be a hereditary right $n$-cotorsion pair in $R$-$\mathrm{Mod}$. Assume that $\mathrm{Tor}_{i}^{R}(M,C)=0$ and $\mathrm{Ext}_{R}^{1}(E,M\otimes_{R}E)=0$ for each $1\leqslant i\leqslant n$, any $C\in\mathcal{C}$ and any $E\in\mathcal{C}_{n-1}^{\vee}$, and $(\mathfrak{B}^{\mathcal{C}})_{n-1}^{\vee}$ is closed under extensions. Then $(\mathfrak{B}^{\mathcal{C}},\mathfrak{U}^{\mathcal{D}})$ is a hereditary right $n$-cotorsion pair in $R\ltimes M$-$\mathrm{Mod}$.
}}
\end{prop}
\begin{proof} $\mathrm{(1)}$ It follows from Theorem 4.4 (1), Lemma 4.8 (1) and Remark 4.5 (1) that $(\mathfrak{U}^{\mathcal{C}},\mathfrak{J}^{\mathcal{D}})$ is a hereditary left $n$-cotorsion pair in $R\ltimes M$-$\mathrm{Mod}$.

$\mathrm{(2)}$ It follows from Theorem 4.4 (2), Lemma 4.8 (2) and Remark 4.5 (2) that $(\mathfrak{B}^{\mathcal{C}},\mathfrak{U}^{\mathcal{D}})$ is a hereditary right $n$-cotorsion pair in $R\ltimes M$-$\mathrm{Mod}$.
\end{proof}

\begin{prop}\label{prop:2.4}{\it{Let $\mathcal{D}$ be a class of left $R$-modules. The following statements hold.

$\mathrm{(1)}$ Let $(\mathfrak{U}^{\mathcal{C}},\mathfrak{J}^{\mathcal{D}})$ be a hereditary left $n$-cotorsion pair in $R\ltimes M$-$\mathrm{Mod}$. Assume that $\mathrm{Ext}^{i}_{R}(M,\\D)=0$, $\mathrm{Ext}_{R}^{1}(\mathrm{Hom}_{R}
(M,D),D)=0$ and $\mathrm{Hom}_{R}(M,D)\in\mathcal{D}$ for each $1\leqslant i\leqslant n$ and any $D\in\mathcal{D}$, and $\mathcal{D}_{n-1}^{\wedge}$ is closed under extensions. Then $(\mathcal{C},\mathcal{D})$ is a hereditary left $n$-cotorsion pair in $R$-$\mathrm{Mod}$.

$\mathrm{(2)}$ Let $(\mathfrak{B}^{\mathcal{C}},\mathfrak{U}^{\mathcal{D}})$ be a hereditary right $n$-cotorsion pair in $R\ltimes M$-$\mathrm{Mod}$. Assume that $\mathrm{Tor}_{i}^{R}(M,C)=0$, $\mathrm{Ext}_{R}^{1}(C,M\otimes_{R}C)=0$ and $M\otimes_{R}C\in\mathcal{C}$ for each $1\leqslant i\leqslant n$ and any $C\in\mathcal{C}$, and $\mathcal{C}_{n-1}^{\vee}$ is closed under extensions. Then $(\mathcal{C},\mathcal{D})$ is a hereditary right $n$-cotorsion pair in $R$-$\mathrm{Mod}$.
}}
\end{prop}
\begin{proof} $\mathrm{(1)}$ It immediately follows from Theorem 4.6 (1), Lemma 4.8 (1) and Lemma 3.14 (1).

$\mathrm{(2)}$ It immediately follows from Theorem 4.6 (2), Lemma 4.8 (2) and Lemma 3.14 (2).
\end{proof}

{\bf 4.2 Left (right) $n$-cotorsion pairs over Morita rings}

Morita rings originated from the theory of equivalences of module categories in \cite{K1958} and were formally introduced by Bass in \cite{H1962}. They are also called Morita context rings, and formal matrix rings. And many authors have carried out related studies on such rings (see \cite{H1962}, \cite{RPI1975}, \cite{EL1982}, \cite{PA2017}, \cite{K1958}). In this subsection, we directly apply the results obtained over trivial ring extensions to Morita rings with zero bimodule homomorphisms, since such rings constitute a special case of trivial ring extensions. Next, we recall some notions and basic facts about Morita rings.

Following \cite{PA2017}, let $A$ and $B$ be rings, $_{A}V_{B}$ an $A$-$B$-bimodule, $_{B}U_{A}$ a $B$-$A$-bimodule, $\phi:U\otimes_{A}V\rightarrow B$ a $B$-$B$-bimodule homomorphism, and $\psi:V\otimes_{B} U\rightarrow A$ an $A$-$A$-bimodule homomorphism. From the Morita context $\mathscr{M}=(A,V,U,B,\phi,\psi)$, we define the Morita ring
$$\Lambda_{(\phi,\psi)}(\mathscr{M})=\left(\begin{matrix}A&_{A}V_{B}\\_{B}U_{A}&B\end{matrix}
\right),$$
where addition of elements of $\Lambda_{(\phi,\psi)}(\mathscr{M})$ is componentwise and multiplication is given by
$$\left(\begin{matrix}a_1&v_1\\u_1&b_1\end{matrix}\right)\left(\begin{matrix}a_2&v_2\\u_2&b_2
\end{matrix}\right)=\left(\begin{matrix}a_{1}a_{2}+\psi(v_{1}\otimes u_2)&a_{1}v_{2}+v_{1}b_{2}
\\u_{1}a_{2}+b_{1}u_{2}&b_{1}b_{2}+\phi(u_{1}\otimes v_2)\end{matrix}\right).$$
We assume that $\phi(u\otimes v)u'=u\psi(v\otimes u')$ and $v\phi(u\otimes v')=\psi(v\otimes u)v'$ for any $u,~u'\in U$ and any $v,~v'\in V$. This condition ensures that $\Lambda_{(\phi,\psi)}(\mathscr{M})$ is an associative ring. For convenience, we use the notation $\Lambda_{(\phi,\psi)}$ instead of $\Lambda_{(\phi,\psi)}
(\mathscr{M})$. When $\phi=0=\psi$, the corresponding Morita ring is denoted by
$\Lambda_{(0,0)}=\left(\begin{smallmatrix}A&_{A}V_{B}\\_{B}U_{A}&B\end{smallmatrix}
\right)$, which is called the Morita ring with zero bimodule homomorphisms.

In \cite[\text{Theorem}~1.5]{EL1982}, Green proved that the category $\Lambda_{(\phi,\psi)}$-$\mathrm{Mod}$ is equivalent to the category $\Omega$, whose objects are tuples $(X,Y,f,g)$, where $X\in A$-$\mathrm{Mod}$, $Y\in B$-$\mathrm{Mod}$, $f\in\mathrm{Hom}_{B}(U\otimes_{A}X,Y)$ and $g\in\mathrm{Hom}_{A}(V\otimes_{B}Y,X)$ such that the following diagrams are commutative:
$$\xymatrix{
  V\otimes_B U\otimes_A X \ar[d]_{\psi\otimes_A X} \ar[r]^{\quad V\otimes_B f} & V\otimes_B Y \ar[d]^{g} \\
  A\otimes_A X \ar[r]^{\quad\cong} &~~~ X~~}
~~~~~~~~\xymatrix{
  U\otimes_A V\otimes_B Y \ar[d]_{\phi\otimes_B Y} \ar[r]^{\quad  U\otimes_A g} & U\otimes_A X \ar[d]^{f} \\
  B\otimes_B Y \ar[r]^{\quad\cong} &~~~ Y.}$$
Obviously, for $\phi=0=\psi$, one has $g(V\otimes_{B}f)=0$, $f(U\otimes_{A}g)=0$.
Morphisms from $(X,Y,f,g)$ to $(X_{1},Y_{1},f_{1},g_{1})$ in $\Omega$ are pairs $(\alpha,\beta)$ such that $\alpha\in\mathrm{Hom}_{A}(X,X_{1})$, $\beta\in\mathrm{Hom}_{B}(Y,Y_{1})$ and the following diagrams are commutative:
$$\small\xymatrix{
      & U\otimes_{A}X \ar[d]_{f} \ar[r]^{U\otimes_{A}\alpha} & U\otimes_{A}X_{1} \ar[d]^{f_{1}}  & \\
      & Y  \ar[r]^{\beta} &  Y_{1}  &  \\   }
\hspace{0.1em}
\small\xymatrix{
      & V\otimes_{B}Y \ar[d]_{g} \ar[r]^{V\otimes_{B}\beta} & V\otimes_{B}Y_{1} \ar[d]^{g_{1}}  & \\
      & X  \ar[r]^{\alpha} &  X_{1}.  &  \\   }\vspace*{2mm}
$$
A sequence
$$0\longrightarrow(X',Y',f',g')\overset{(\alpha^{'},\beta^{'})}\longrightarrow(X,Y,f,g)\overset{(\alpha^{''},
\beta^{''})}\longrightarrow(X'',Y'',f'',g'')\longrightarrow0$$
in $\Lambda_{(\phi,\psi)}$-$\mathrm{Mod}$ is exact if and only if both the sequence
$0\rightarrow X'\overset{\alpha^{'}}\rightarrow X\overset{\alpha^{''}}\rightarrow X''\rightarrow0$ in $A$-$\mathrm{Mod}$ and the sequence
$0\rightarrow Y'\overset{\beta^{'}}\rightarrow Y\overset{\beta^{''}}\rightarrow Y''\rightarrow0$ in $B$-$\mathrm{Mod}$ are exact.

By the adjunction isomorphism, the category $\Lambda_{(\phi,\psi)}$-$\mathrm{Mod}$ is also equivalent to the category $\Gamma$, whose objects are tuples $[X,Y,f,g]$, where $X\in A$-$\mathrm{Mod}$, $Y\in B$-$\mathrm{Mod}$, $f\in\mathrm{Hom}_{A}(X,\mathrm{Hom}_{B}$\\
$(U,Y))$ and $g\in\mathrm{Hom}_{B}(Y,\mathrm{Hom}_{A}(V,X))$ such that
the following diagrams are commutative:
$$\xymatrix@R=2.8em@C=2.8em{
  X \ar[d]_{f} \ar[r]^{\mathrm{Hom}_{A}(\psi,X)h_{A,X}\quad\quad\quad\quad} & \mathrm{Hom}_{A}(V\otimes_{B}U,X) \ar[d]^{\cong} \\
  \mathrm{Hom}_{B}(U,Y) \ar[r]^{\mathrm{Hom}_{B}(U,g)\quad\quad} &\mathrm{Hom}_{B}(U,\mathrm{Hom}_{A}(V,X))}
~~~~~~~~\xymatrix@R=2.8em@C=2.8em{
  Y \ar[d]_{g} \ar[r]^{\mathrm{Hom}_{B}(\phi,Y)h_{B,Y}\quad\quad\quad\quad} & \mathrm{Hom}_{B}(U\otimes_{A}V,Y) \ar[d]^{\cong} \\
  \mathrm{Hom}_{A}(V,X) \ar[r]^{\mathrm{Hom}_{A}(V,f)\quad\quad} &\mathrm{Hom}_{A}(V,\mathrm{Hom}_{B}(U,Y)).}$$
Here, $h_{A,X}:X\rightarrow\mathrm{Hom}_{A}(A,X)$ and $h_{B,Y}:Y\rightarrow\mathrm{Hom}_{B}(B,Y)$ are the canonical isomorphisms.
Obviously, for $\phi=0=\psi$, one has $\mathrm{Hom}_{B}(U,g)f=0$, $\mathrm{Hom}_{A}(V,f)g=0$.
Morphisms from $[X,Y,f,g]$ to $[X_{1},Y_{1},f_{1},g_{1}]$ in $\Gamma$ are pairs $[\alpha,\beta]$ such that $\alpha\in\mathrm{Hom}_{A}(X,X_{1})$, $\beta\in\mathrm{Hom}_{B}(Y,Y_{1})$ and the following diagrams are commutative:
$$\small\xymatrix@R=2.8em@C=2.8em{
      & X \ar[d]_{f} \ar[r]^{\alpha} & X_{1} \ar[d]^{f_{1}}  & \\
      & \mathrm{Hom}_{B}(U,Y)  \ar[r]^{\mathrm{Hom}_{B}(U,\beta)} &  \mathrm{Hom}_{B}(U,Y_{1})  &  \\   }
\hspace{0.1em}
\small\xymatrix@R=2.8em@C=2.8em{
      & Y \ar[d]_{g} \ar[r]^{\beta} & Y_{1} \ar[d]^{g_{1}}  & \\
      & \mathrm{Hom}_{A}(V,X)  \ar[r]^{\mathrm{Hom}_{A}(V,\alpha)} &  \mathrm{Hom}_{A}(V,X_{1}).  &  \\   }\vspace*{2mm}
$$

We can identify the category $\Lambda_{(\phi,\psi)}$-$\mathrm{Mod}$ with the category $\Omega$ and the category $\Gamma$. In the rest of this subsection, we assume that $\phi=0=\psi$.

Note that the ring homomorphisms $A\times B\rightarrow A$ and $A\times B\rightarrow B$ equip $U\oplus V$ with a left $A\times B$-right $A\times B$-bimodule structure. It is known that a left $A\times B$-module is an ordered pair $(X,Y)$ with $X\in A$-$\mathrm{Mod}$ and $Y\in B$-$\mathrm{Mod}$. Similarly, a right $A\times B$-module is an ordered pair $(W_1,W_2)$ with $W_{1}\in\mathrm{Mod}$-$A$ and $W_{2}\in\mathrm{Mod}$-$B$. Therefore, we have two isomorphisms $(U\oplus V)\otimes_{A\times B}(X,Y)\cong(V\otimes_{B}Y,U\otimes_{A}X)$ and $\mathrm{Hom}_{A
\times B}(U\oplus V,(X,Y))\cong(\mathrm{Hom}_{B}(U,Y),\mathrm{Hom}_{A}(V,X))$. By \cite{RPI1975}, the Morita ring $\Lambda_{(0,0)}=\left(\begin{smallmatrix}A&_{A}V_{B}\\_{B}U_{A}&B\end{smallmatrix}\right)$ is isomorphic to the trivial ring extension $(A\times B)\ltimes(U\oplus V)$ under the correspondence $\left(\begin
{smallmatrix}a&v\\u&b\end{smallmatrix}\right)\rightarrow((a,b),(u,v))$. Thus $\Lambda_{(0,0)}$-$\mathrm{Mod}$ is isomorphic to $(A\times B)\ltimes(U\oplus V)$-$\mathrm{Mod}$ by the functor $$\Theta:\Lambda
_{(0,0)}\text{-}\mathrm{Mod}\rightarrow(A\times B)\ltimes(U\oplus V)\text{-}\mathrm{Mod}$$
given by $\Theta(X,Y,f,g)=((X,Y),(g,f))$ or by $\Theta[X,Y,f,g]=[(X,Y),(f,g)]$.

Let $\mathcal{C}$ be a class of left $A$-modules and $\mathcal{D}$ a class of left $B$-modules. Now, we define several special classes in $\Lambda_{(0,0)}$-$\mathrm{Mod}$ as follows.

(1) $\mathbf{T}_{A}(\mathcal{C}):=\{(X,U\otimes_{A}X,1,0)\mid X\in\mathcal{C}\}$.

(2) $\mathbf{T}_{B}(\mathcal{D}):=\{(V\otimes_{B}Y,Y,0,1)\mid Y\in\mathcal{D}\}$.

(3) $\mathbf{H}_{A}(\mathcal{C}):=\{(X,\mathrm{Hom}_{A}(V,X),0,1)\mid X\in\mathcal{C}\}$.

(4) $\mathbf{H}_{B}(\mathcal{D}):=\{(\mathrm{Hom}_{B}(U,Y),Y,1,0)\mid Y\in\mathcal{D}\}$.

(5) $\mathfrak{U}_{\mathcal{D}}^{\mathcal{C}}:=\{(X,Y,f,g)\mid X\in\mathcal{C},~Y\in\mathcal{D}\}=\{[X,Y,f,g]\mid X\in
\mathcal{C},~Y\in\mathcal{D}\}$.

(6) $\mathfrak{B}_{\mathcal{D}}^{\mathcal{C}}:=
\{(X,Y,f,g)\mid\mathrm{Coker}f\in\mathcal{D},~\mathrm{Coker}g\in
\mathcal{C},~\text{the sequences}~V\otimes_{B}U\otimes_{A}X\overset{V\otimes_{B}f}\longrightarrow V\otimes_{B}Y
\overset{g}\longrightarrow X~\text{and}~U\otimes_{A}V\otimes_{B}Y\overset{U\otimes_{A}g}\longrightarrow U\otimes_{A}X\overset{f}\longrightarrow Y~\text{are exact}\}$.

(7) $\mathfrak{J}_{\mathcal{D}}^{\mathcal{C}}:=
\{[X,Y,f,g]\mid\mathrm{Ker}f\in\mathcal{C},~\mathrm{Ker}g\in
\mathcal{D},~\text{the sequences}~X\overset{f}\longrightarrow\mathrm{Hom}_{B}(U,Y)\overset{\mathrm{Hom}_{B}(U,g)}
\longrightarrow\mathrm{Hom}_{B}(U,\mathrm{Hom}_{A}(V,X))~\text{and}~Y\overset{g}\longrightarrow
\mathrm{Hom}_{A}(V,X)\overset{\mathrm{Hom}_{A}(V,f)}\longrightarrow\mathrm{Hom}_{A}(V,
\mathrm{Hom}_{B}(U,Y))~\text{are exact}\}$.

It is easy to see that $\Theta(\mathbf{T}_{A}(\mathcal{C}))=\mathbf{T}(\mathcal{C},0)$, $\Theta(\mathbf{T}_{B}
(\mathcal{D}))=\mathbf{T}(0,\mathcal{D})$, $\Theta(\mathbf{H}_{A}(\mathcal{C}))=\mathbf{H}(\mathcal{C},0)$, $\Theta(\mathbf{H}_{B}(\mathcal{D}))=\mathbf{H}(0,\mathcal{D})$, $\Theta(\mathfrak{U}_{\mathcal{D}}
^{\mathcal{C}})=\mathfrak{U}^{(\mathcal{C},\mathcal{D})}$, $\Theta(\mathfrak{B}_{\mathcal{D}}^{\mathcal{C}})
=\mathfrak{B}^{(\mathcal{C},\mathcal{D})}$, $\Theta(\mathfrak{J}_{\mathcal{D}}^{\mathcal{C}})=\mathfrak{J}^{(
\mathcal{C},\mathcal{D})}$. These identities provide a bridge between $(A\times B)\ltimes(U\oplus V)$-$\mathrm{Mod}$ and $\Lambda_{(0,0)}$-$\mathrm{Mod}$, and play an essential role in the subsequent proofs.

In the following, we turn to the discussion of $n$-cotorsion pairs in $\Lambda_{(0,0)}$-$\mathrm{Mod}$.

\begin{prop}\label{prop:2.4}{\it{$\mathrm{(1)}$ Let $(\mathcal{C}_{1},\mathcal{C}_{2})$ and $(\mathcal{D}_{1},\mathcal{D}_{2})$ be left $n$-cotorsion pairs in $A$-$\mathrm{Mod}$ and $B$-$\mathrm{Mod}$, respectively. Assume that $\mathrm{Ext}^{i}_{A}(V,\mathcal{C}_
{2})=0=\mathrm{Ext}^{i}_{B}(U,\mathcal{D}_{2})$ for each $1\leqslant i\leqslant n$, $\mathrm{Ext}_{A}^{1}(\mathrm{Hom}_{B}(U,(\mathcal{D}_{2}
)_{n-1}^{\wedge}),(\mathcal{C}_{2})_{n-1}^{\wedge})=0=\mathrm{Ext}_{B}^{1}(\mathrm{Hom}_{A}
(V,(\mathcal{C}_{2})_{n-1}^{\wedge}),(\mathcal{D}_{2})_{n-1}^{\wedge})$, $\mathfrak{U}^{\mathcal{C}_{1}}_{\mathcal{D}_{1}}$ is a resolving class, and $(\mathfrak{J}^{\mathcal{C}_{2}}
_{\mathcal{D}_{2}})_{n-1}^{\wedge}$ is closed under extensions. Then $(\mathfrak{U}^{\mathcal{C}_{1}}_{\mathcal{D}
_{1}},\mathfrak{J}^{\mathcal{C}_{2}}_{\mathcal{D}_{2}})$ is a hereditary left $n$-cotorsion pair in $\Lambda_{(0,0)}$-$\mathrm{Mod}$.

$\mathrm{(2)}$ Let $(\mathcal{C}_{1},\mathcal{C}_{2})$ and $(\mathcal{D}_{1},\mathcal{D}_{2})$ be right $n$-cotorsion pairs in $A$-$\mathrm{Mod}$ and $B$-$\mathrm{Mod}$, respectively. Assume that $\mathrm{Tor}_{i}^{A}(U,\mathcal{C}_{1})=0=\mathrm{Tor}_{i}^{B}(V,\mathcal{D}_{1})$ for each $1\leqslant i\leqslant n$, $\mathrm{Ext}_{A}^{1}
((\mathcal{C}_{1})_{n-1}^{\vee},V\otimes_{B}(\mathcal{D}_{1})_{n-1}^{\vee})=0=\mathrm{Ext}_{B}^{1}
((\mathcal{D}_{1})_{n-1}^{\vee},
U \otimes_{A}(\mathcal{C}_{1})_{n-1}^{\vee})$, $\mathfrak{U}^{\mathcal{C}_{2}}_{\mathcal{D}_{2}}$ is a coresolving class, and $(\mathfrak{B}^{\mathcal{C}_{1}}_{\mathcal{D}_{1}})_{n-1}^{\vee}$ is closed under extensions. Then $(\mathfrak{B}^{\mathcal{C}_{1}}_{\mathcal{D}_{1}},\mathfrak{U}^{\mathcal{C}_{2}}_{\mathcal{D}_{2}})$ is a hereditary right $n$-cotorsion pair in $\Lambda_{(0,0)}$-$\mathrm{Mod}$.
}}
\end{prop}

\begin{proof}
Since $(\mathcal{C}_{1},\mathcal{C}_{2})$ and $(\mathcal{D}_{1},\mathcal{D}_{2})$ are left (resp. right) $n$-cotorsion pairs in $A$-$\mathrm{Mod}$ and $B$-$\mathrm{Mod}$, respectively, the pair $((\mathcal{C}_{1},
\mathcal{D}_{1}),(\mathcal{C}_{2},\mathcal{D}_{2}))$ is a left (resp. right) $n$-cotorsion pair in $A\times B$-$\mathrm{Mod}$.

$\mathrm{(1)}$ By Theorem 4.4 (1) and Remark 4.5 (1), $(\mathfrak{U}^{(\mathcal{C}_{1},\mathcal{D}_{1})},\mathfrak{J}
^{(\mathcal{C}_{2},\mathcal{D}_{2})})$ is a hereditary left $n$-cotorsion pair in $(A\times B)\ltimes(U\oplus V)$-$\mathrm{Mod}$. Consequently, $(\mathfrak{U}^{\mathcal{C}_{1}}_{\mathcal{D}_{1}},\mathfrak{J}^{\mathcal{C}
_{2}}_{\mathcal{D}_{2}})$ is a hereditary left $n$-cotorsion pair in $\Lambda_{(0,0)}$-$\mathrm{Mod}$.

$\mathrm{(2)}$ By Theorem 4.4 (2) and Remark 4.5 (2), $(\mathfrak{B}^{(\mathcal{C}_{1},\mathcal{D}_{1})},\mathfrak{U}
^{(\mathcal{C}_{2},\mathcal{D}_{2})})$ is a hereditary right $n$-cotorsion pair in $(A\times B)\ltimes(U\oplus V)$-$\mathrm{Mod}$. Consequently, $(\mathfrak{B}^{\mathcal{C}_{1}}_{\mathcal{D}_{1}},\mathfrak{U}^{\mathcal{C}
_{2}}_{\mathcal{D}_{2}})$ is a hereditary right $n$-cotorsion pair in $\Lambda_{(0,0)}$-$\mathrm{Mod}$.
\end{proof}

\begin{prop}\label{prop:2.4}{\it{Let $\mathcal{C}_{1}$ and $\mathcal{C}_{2}$ be classes of left $A$-modules, and $\mathcal{D}_{1}$ and $\mathcal{D}_{2}$ be classes of left $B$-modules. The following statements hold.

$\mathrm{(1)}$ Let $(\mathfrak{U}^{\mathcal{C}_{1}}_{\mathcal{D}_{1}},\mathfrak{J}^{\mathcal{C}
_{2}}_{\mathcal{D}_{2}})$ be a left $n$-cotorsion pair in $\Lambda_{(0,0)}$-$\mathrm{Mod}$. Assume that $\mathrm{Ext}_{A}^{i}(V,\mathcal{C}_{2})=0=\mathrm{Ext}_{B}
^{i}(U,\mathcal{D}_{2})$ for each $1\leqslant i\leqslant n$, $\mathrm{Ext}_{A}^{1}(\mathrm{Hom}_{B}(U,\mathcal{D}_{2}),\mathcal{C}_{2})=
0=\mathrm{Ext}_{B}^{1}(\mathrm{Hom}
_{A}(V,\mathcal{C}_{2}),\mathcal{D}_{2})$, $\mathrm{Hom}_{A}(V,\mathcal{C}_{2})
\subseteq\mathcal{D}_{2}$, $\mathrm{Hom}_{B}(U,\mathcal{D}_{2})\subseteq\mathcal{C}_{2}$, and $(\mathcal{C}_{2})_{n-1}^{\wedge}$ and $(\mathcal{D}_{2})_{n-1}^{\wedge}$ are closed under extensions. Then $(\mathcal{C}_{1},
\mathcal{C}_{2})$ and $(\mathcal{D}_{1},\mathcal{D}_{2})$ are left $n$-cotorsion pairs in $A$-$\mathrm{Mod}$ and $B$-$\mathrm{Mod}$, respectively.

$\mathrm{(2)}$ Let $(\mathfrak{B}^{\mathcal{C}_{1}}_{\mathcal{D}_{1}},\mathfrak{U}^{\mathcal{C}
_{2}}_{\mathcal{D}_{2}})$ be a right $n$-cotorsion pair in $\Lambda_{(0,0)}$-$\mathrm{Mod}$. Assume that $\mathrm{Tor}_{i}^{A}(U,\mathcal{C}_{1})=0=\mathrm{Tor}_{i}^{B}(V,\mathcal{D}_{1})$ for each $1\leqslant i\leqslant n$, $\mathrm{Ext}
_{A}^{1}(\mathcal{C}_{1},V\otimes_{B}\mathcal{D}_{1})=
0=\mathrm{Ext}_{B}^{1}(\mathcal{D}_{1},U\otimes_{A}
\mathcal{C}_{1})$, $V\otimes_{B}\mathcal{D}_{1}\subseteq\mathcal{C}_{1}$, $U\otimes_{A}\mathcal{C}_{1}
\subseteq\mathcal{D}_{1}$, and $(\mathcal{C}_{1})_{n-1}^{\vee}$ and $(\mathcal{D}_{1})_{n-1}^{\vee}$ are closed under extensions. Then $(\mathcal{C}_{1},\mathcal{C}_{2})$ and $(\mathcal{D}_{1},
\mathcal{D}_{2})$ are right $n$-cotorsion pairs in $A$-$\mathrm{Mod}$ and $B$-$\mathrm{Mod}$, respectively.
}}
\end{prop}

\begin{proof}
$\mathrm{(1)}$ Since $(\mathfrak{U}^{\mathcal{C}_{1}}_{\mathcal{D}_{1}},\mathfrak{J}^{\mathcal{C}_{2}}
_{\mathcal{D}_{2}})$ is a left $n$-cotorsion pair in $\Lambda_{(0,0)}$-$\mathrm{Mod}$, the pair $(\mathfrak{U}^
{(\mathcal{C}_{1},\mathcal{D}_{1})},\mathfrak{J}^{(\mathcal{C}_{2},\mathcal{D}_{2})})$ is a left $n$-cotorsion pair in $(A\times B)\ltimes(U\oplus V)$-$\mathrm{Mod}$. It follows from Theorem 4.6 (1) that $((\mathcal{C}_{1},
\mathcal{D}_{1}),(\mathcal{C}_{2},\mathcal{D}_{2}))$ is a left $n$-cotorsion pair in $A\times B$-$\mathrm{Mod}$. Therefore, $(\mathcal{C}_{1},\mathcal{C}_{2})$ and $(\mathcal{D}_{1},
\mathcal{D}_{2})$ are left $n$-cotorsion pairs in $A$-$\mathrm{Mod}$ and $B$-$\mathrm{Mod}$, respectively.

$\mathrm{(2)}$ Since $(\mathfrak{B}^{\mathcal{C}_{1}}_{\mathcal{D}_{1}},\mathfrak{U}^{\mathcal{C}
_{2}}_{\mathcal{D}_{2}})$ is a right $n$-cotorsion pair in $\Lambda_{(0,0)}$-$\mathrm{Mod}$, the pair $(\mathfrak{B}^{(\mathcal{C}_{1},\mathcal{D}_{1})},\mathfrak{U}^{(\mathcal{C}_{2},\mathcal{D}_{2})})$ is a right $n$-cotorsion pair in $(A\times B)\ltimes(U\oplus V)$-$\mathrm{Mod}$. It follows from Theorem 4.6 (2) that $((\mathcal{C}_{1},\mathcal{D}_{1}),(\mathcal{C}_{2},\mathcal{D}_{2}))$ is a right $n$-cotorsion pair in $A\times B$-$\mathrm{Mod}$. Therefore, $(\mathcal{C}_{1},\mathcal{C}_{2})$ and $(\mathcal{D}_{1},
\mathcal{D}_{2})$ are right $n$-cotorsion pairs in $A$-$\mathrm{Mod}$ and $B$-$\mathrm{Mod}$, respectively.
\end{proof}

The following results are immediate from Proposition 4.11, Proposition 4.12, Lemma 4.8 and Lemma 3.14.

\begin{prop}\label{prop:2.4}{\it{$\mathrm{(1)}$ Let $(\mathcal{C}_{1},\mathcal{C}_{2})$ and $(\mathcal{D}_{1},\mathcal{D}_{2})$ be hereditary left $n$-cotorsion pairs in $A$-$\mathrm{Mod}$ and $B$-$\mathrm{Mod}$, respectively. Assume that $\mathrm{Ext}^{i}_{A}(V,\mathcal{C}_{2})=0=\mathrm{Ext}^{i}_{B}(U,\mathcal{D}_{2})$ for each $1\leqslant i\leqslant n$, $\mathrm{Ext}_{A}^{1}(\mathrm{Hom}_{B}(U,
(\mathcal{D}_{2})_{n-1}^{\wedge}),(\mathcal{C}_{2})_{n-1}^{\wedge})=
0=\mathrm{Ext}_{B}^{1}(\mathrm{Hom}_{A}
(V,(\mathcal{C}_{2})_{n-1}^{\wedge}),(\mathcal{D}_{2})_{n-1}^{\wedge})$, and $(\mathfrak{J}^{\mathcal{C}_{2}}_{\mathcal{D}_{2}})_{n-1}^{\wedge}$ is closed under extensions. Then $(\mathfrak{U}^{\mathcal{C}_{1}}_{\mathcal{D}_{1}},\mathfrak{J}^{\mathcal{C}_{2}}_{\mathcal{D}_{2}})$ is a hereditary left $n$-cotorsion pair in $\Lambda_{(0,0)}$-$\mathrm{Mod}$.

$\mathrm{(2)}$ Let $(\mathcal{C}_{1},\mathcal{C}_{2})$ and $(\mathcal{D}_{1},\mathcal{D}_{2})$ be hereditary right $n$-cotorsion pairs in $A$-$\mathrm{Mod}$ and $B$-$\mathrm{Mod}$, respectively. Assume that $\mathrm{Tor}_{i}^{A}(U,\mathcal{C}_{1})=0=\mathrm{Tor}_{i}^{B}(V,\mathcal{D}_{1})$ for each $1\leqslant i\leqslant n$, $\mathrm{Ext}_{A}^{1}((\mathcal{C}_{1})_{n-1}^{\vee},
V\otimes_{B}(\mathcal{D}_{1})_{n-1}^{\vee})=0=
\mathrm{Ext}_{B}^{1}((\mathcal{D}_{1})_{n-1}^{\vee},U\otimes_{A}(\mathcal{C}_{1})_{n-1}^{\vee})$, and $(\mathfrak{B}^{\mathcal{C}_{1}}_{\mathcal{D}_{1}})_{n-1}^{\vee}$ is closed under extensions. Then $(\mathfrak{B}^{
\mathcal{C}_{1}}_{\mathcal{D}_{1}},\mathfrak{U}^{\mathcal{C}_{2}}_{\mathcal{D}_{2}})$ is a hereditary right $n$-cotorsion pair in $\Lambda_{(0,0)}$-$\mathrm{Mod}$.
}}
\end{prop}

\begin{prop}\label{prop:2.4}{\it{Let $\mathcal{C}_{1}$ and $\mathcal{C}_{2}$ be classes of left $A$-modules, and $\mathcal{D}_{1}$ and $\mathcal{D}_{2}$ be classes of left $B$-modules. The following statements hold.

$\mathrm{(1)}$ Let $(\mathfrak{U}^{\mathcal{C}_{1}}_{\mathcal{D}_{1}},\mathfrak{J}^{\mathcal{C}
_{2}}_{\mathcal{D}_{2}})$ be a hereditary left $n$-cotorsion pair in $\Lambda_{(0,0)}$-$\mathrm{Mod}$. Assume that $\mathrm{Ext}_{A}^{i}(V,\mathcal{C}_{2})\\=0=\mathrm{Ext}_{B}^{i}
(U,\mathcal{D}_{2})$ for each $1\leqslant i\leqslant n$, $\mathrm{Ext}_{A}^{1}(\mathrm{Hom}_{B}(U,\mathcal{D}_{2}),\mathcal{C}_{2})=
0=\mathrm{Ext}_{B}^{1}(\mathrm{Hom}
_{A}(V,\mathcal{C}_{2}),\mathcal{D}_{2})$, $\mathrm{Hom}_{A}(V,\mathcal{C}_{2})
\subseteq\mathcal{D}_{2}$, $\mathrm{Hom}_{B}(U,\mathcal{D}_{2})\subseteq\mathcal{C}_{2}$, and $(\mathcal{C}_{2})_{n-1}^{\wedge}$ and $(\mathcal{D}_{2})_{n-1}^{\wedge}$ are closed under extensions. Then $(\mathcal{C}
_{1},\mathcal{C}_{2})$ and $(\mathcal{D}_{1},\mathcal{D}_{2})$ are hereditary left $n$-cotorsion pairs in $A$-$\mathrm{Mod}$ and $B$-$\mathrm{Mod}$, respectively.

$\mathrm{(2)}$ Let $(\mathfrak{B}^{\mathcal{C}_{1}}_{\mathcal{D}_{1}},\mathfrak{U}^{\mathcal{C}_{2}}
_{\mathcal{D}_{2}})$ be a hereditary right $n$-cotorsion pair in $\Lambda_{(0,0)}$-$\mathrm{Mod}$. Assume that $\mathrm{Tor}_{i}^{A}(U,\mathcal{C}_{1})=0=\mathrm{Tor}_{i}^{B}(V,
\mathcal{D}_{1})$ for each $1\leqslant i\leqslant n$, $\mathrm{Ext}_{A}^{1}(\mathcal{C}_{1},V\otimes_{B}\mathcal{D}_{1})=0
=\mathrm{Ext}_{B}^{1}(\mathcal{D}_{1},U\otimes_{A}\mathcal{C}_{1})$, $V\otimes_{B}\mathcal{D}_{1}\subseteq\mathcal{C}_{1}$, $U\otimes
_{A}\mathcal{C}_{1}\subseteq\mathcal{D}_{1}$, and $(\mathcal{C}_{1})_{n-1}^{\vee}$ and $(\mathcal{D}_{1})
_{n-1}^{\vee}$ are closed under extensions. Then $(\mathcal{C}_{1},\mathcal{C}_{2})$ and $(\mathcal{D}_{1},
\mathcal{D}_{2})$ are hereditary right $n$-cotorsion pairs in $A$-$\mathrm{Mod}$ and $B$-$\mathrm{Mod}$, respectively.
}}
\end{prop}

\section{some applications}

To generalize the concept of trivial extension categories of abelian categories, Marmaridis pioneered introduction of $\eta$-extension categories and $\zeta$-coextension categories of abelian categories in \cite{N1993}, and proved that these two categories are isomorphic. According to \cite{RPI1975}, comma categories are regarded as an example of trivial extension categories of abelian categories. Meanwhile, Marmaridis also defined $\theta$-extension rings in \cite{N1993}, where trivial ring extensions constitute a special case of $\theta$-extension rings. Note that formal triangular matrix rings are special cases of trivial ring extensions by Subsection 4.2. Also, the categories of modules over formal triangular matrix rings are comma categories by Example 2.5. We summarize relations between the aforementioned categories and rings as follows:
\begin{itemize}
    \item $\eta$-extension category of an abelian category \quad $\xrightarrow{\text{special case}}$ \quad trivial extension category of an abelian category \quad $\xrightarrow{\text{special case}}$ \quad comma category \quad $\xrightarrow{\text{special case}}$ \quad category of modules over a formal triangular matrix ring
    \item $\theta$-extension ring \quad $\xrightarrow{\text{special case}}$ \quad trivial ring extension \quad $\xrightarrow{\text{special case}}$ \quad Morita ring with zero bimodule homomorphisms \quad $\xrightarrow{\text{special case}}$ \quad formal triangular matrix ring
\end{itemize}
Naturally, the above ring correspondences induce the following relations among their module categories:
\begin{itemize}
    \item category of modules over a $\theta$-extension ring \quad $\xrightarrow{\text{special case}}$ \quad category of modules over a trivial ring extension \quad $\xrightarrow{\text{special case}}$ \quad category of modules over a Morita ring with zero bimodule homomorphisms \quad $\xrightarrow{\text{special case}}$ \quad category of modules over a formal triangular matrix ring
\end{itemize}
Therefore, whether from the perspective of $\eta$-extension categories or categories of modules over $\theta$-extension rings, or more specifically, whether from the perspective of comma categories or categories of modules over trivial ring extensions, we observe that all can eventually return to categories of modules over formal triangular matrix rings. Hence, this section applies the conclusions derived in Sections 3 and 4 to the framework of formal triangular matrix rings.

First, we review the concept of formal triangular matrix rings.
Let $R$ and $S$ be rings, and $M$ an $R$-$S$-bimodule. Then we obtain the formal triangular matrix ring $\Lambda=\left(\begin{smallmatrix}R&M\\0&S\end{smallmatrix}\right)$ with usual matrix addition and multiplication. By \cite[\text{Theorem}~1.5]{EL1982}, the category $\Lambda$-$\mathrm{Mod}$ is equivalent to the category $\Delta$, whose objects are triples $\left(\begin{smallmatrix}X\\Y\end{smallmatrix}\right)_{\varphi}$, where $X\in R$-$\mathrm{Mod}$, $Y\in S$-$\mathrm{Mod}$, and $\varphi:M\otimes_{S}Y\rightarrow X$ is an $R$-module homomorphism, and whose morphisms from $\left(\begin{smallmatrix}X\\Y\end{smallmatrix}\right)_{\varphi}$ to $\left(\begin{smallmatrix}X_{1}\\Y_{1}\end{smallmatrix}\right)_{\varphi_{1}}$ are morphism pairs $\left(\begin{smallmatrix}f\\g\end{smallmatrix}\right)$ with $f\in\mathrm{Hom}_{R}(X,X_{1})$ and $g\in\mathrm{Hom}_{S}(Y,Y_{1})$ such that the following diagrams are commutative
$$\small\xymatrix{
      & M\otimes_{S}Y \ar[d]_{\varphi} \ar[r]^{M{\otimes_{S}}g} & M\otimes_{S}Y_{1} \ar[d]^{\varphi_{1}}  & \\
      & X  \ar[r]^{f} &  X_{1}.  &  \\   }\vspace*{2mm}
$$
We can identify the category $\Lambda$-$\mathrm{Mod}$ with the category $\Delta$. Given a left $\Lambda$-module $\left(\begin{smallmatrix}X\\Y\end{smallmatrix}\right)_{\varphi}$, we will denote by $\varphi^{M}$ the $S$-morphism from $Y$ to $\mathrm{Hom}_{R}(M,X)$ given by $\varphi^{M}(y)(m)=\varphi(m\otimes y)$ for any $y\in Y$ and $m\in M$.

Throughout this section, let $\Lambda=\left(\begin{smallmatrix}R&M\\0&S\end{smallmatrix}
\right)$ be a formal triangular matrix ring, $\mathcal{X}$ and $\mathcal{M}$ be classes of left $R$-modules, and $\mathcal{Y}$ and $\mathcal{N}$ be classes of left $S$-modules. Now, we define several special classes in $\Lambda$-$\mathrm{Mod}$ as follows.

(1) $\left(\begin{smallmatrix}\mathcal{X}\\ \mathcal{Y}\end{smallmatrix}\right):
=\{\left(\begin{smallmatrix}X\\Y\end{smallmatrix}\right)_{\varphi}\in
\Lambda\text{-}\mathrm{Mod}\mid X\in\mathcal{X},~Y\in\mathcal{Y}\}$.

(2) $\left(\begin{smallmatrix}\mathcal{M}\\ \mathcal{N}\end{smallmatrix}\right):=\{\left(\begin{smallmatrix}Z\\N
\end{smallmatrix}\right)_{\varphi}\in\Lambda\text{-}\mathrm{Mod}\mid Z\in\mathcal{M},~N\in\mathcal{N}\}$.

(3) $\mathfrak{B}_{\mathcal{Y}}^{\mathcal{X}}:=
\{\left(\begin{smallmatrix}X\\Y\end{smallmatrix}\right)_{\varphi}\in
\Lambda\text{-}\mathrm{Mod}\mid Y\in\mathcal{Y},~\mathrm{Ker}\varphi
=0~\text{and}~\mathrm{Coker}\varphi\in\mathcal{X}\}$.

(4) $\mathfrak{D}_{\mathcal{N}}^{\mathcal{M}}:
=\{\left(\begin{smallmatrix}Z\\N\end{smallmatrix}\right)_{\varphi}\in
\Lambda\text{-}\mathrm{Mod}\mid Z\in\mathcal{M},~\mathrm{Ker}\varphi^{M}
\in\mathcal{N}~\text{and}~\mathrm{Coker}\varphi^{M}=0\}$.

We now apply Theorem 3.11, Theorem 3.12, Proposition 3.16 and Proposition 3.17 from Section 3 to formal triangular matrix rings. It is worth noting that we refine the condition ``~$1\leqslant j\leqslant n+1$~'' appearing in \cite[\text{Theorem}~3.7]{TX2025}, \cite[\text{Theorem}~3.8]{TX2025} and \cite[\text{Theorem}~3.13]{TX2025}.

\begin{prop}\label{prop:2.4}{\it{$(1)$ Let $(\mathcal{X},\mathcal{M})$ be a left $n$-cotorsion pair in $R$-$\mathrm{Mod}$ and $(\mathcal{Y},\mathcal{N})$ a left $n$-cotorsion pair in $S$-$\mathrm{Mod}$. Assume that $\mathrm{Ext}_{R}^{j}(M,\mathcal{M})=0$ for each $1\leqslant j\leqslant n$, $\left(\begin{smallmatrix}\mathcal{X}\\ \mathcal{Y}\end{smallmatrix}\right)$ is a resolving class, and $(\mathfrak{D}_{\mathcal{N}}^{\mathcal{M}})_{n-1}^{\wedge}$ is closed under extensions. Then $(\left(\begin{smallmatrix}\mathcal{X}\\\mathcal{Y}\end{smallmatrix}\right),\mathfrak{D}_{\mathcal{N}}
^{\mathcal{M}})$ is a hereditary left $n$-cotorsion pair in $\Lambda$-$\mathrm{Mod}$.

$(2)$ Let $(\mathcal{X},\mathcal{M})$ be a right $n$-cotorsion pair in $R$-$\mathrm{Mod}$ and $(\mathcal{Y},\mathcal{N})$ a right $n$-cotorsion pair in $S$-$\mathrm{Mod}$. Assume that $\mathrm{Tor}_{j}^{S}(M,\mathcal{Y})=0$ for each $1\leqslant j\leqslant n$, $\left(\begin{smallmatrix}\mathcal{M}\\ \mathcal{N}\end{smallmatrix}\right)$ is a coresolving class, and $(\mathfrak{B}_{\mathcal{Y}}^{\mathcal{X}})_{n-1}^{\vee}$ is closed under extensions. Then $(\mathfrak{B}_{\mathcal{Y}}^{\mathcal{X}},\left(\begin{smallmatrix}\mathcal{M}\\ \mathcal{N}\end{smallmatrix}\right))$ is a hereditary right $n$-cotorsion pair in $\Lambda$-$\mathrm{Mod}$.
}}
\end{prop}

\begin{prop}\label{prop:2.4}{\it{
$(1)$ Assume that $(\left(\begin{smallmatrix}\mathcal{X}\\ \mathcal{Y}\end{smallmatrix}\right),\mathfrak{D}
_{\mathcal{N}}^{\mathcal{M}})$ is a left $n$-cotorsion pair in $\Lambda$-$\mathrm{Mod}$. The following statements hold.

$(a)$ If $\mathrm{Ext}_{R}^{j}(M,\mathcal{M})=0$ for each $1\leqslant j\leqslant n$, then $(\mathcal{X},
\mathcal{M})$ is a left $n$-cotorsion pair in $R$-$\mathrm{Mod}$.

$(b)$ If $\mathrm{Ext}_{R}^{j}(M,\mathcal{M})=0$ for each $1\leqslant j\leqslant n$, $\mathrm{Hom}_{R}(M,
\mathcal{X})\subseteq\mathcal{N}_{n-1}^{\wedge}$, and $\mathcal{N}_{n-1}^{\wedge}$ is closed under extensions, then $(\mathcal{Y},\mathcal{N})$ is a left $n$-cotorsion pair in $S$-$\mathrm{Mod}$.

$(2)$ Assume that $(\mathfrak{B}_{\mathcal{Y}}^{\mathcal{X}},\left(\begin{smallmatrix}\mathcal{M}\\ \mathcal{N}\end{smallmatrix}\right))$ is a right $n$-cotorsion pair in $\Lambda$-$\mathrm{Mod}$. The following statements hold.

$(a)$ If $\mathrm{Tor}_{j}^{S}(M,\mathcal{Y})=0$ for each $1\leqslant j\leqslant n$, then $(\mathcal{Y},
\mathcal{N})$ is a right $n$-cotorsion pair in $S$-$\mathrm{Mod}$.

$(b)$ If $\mathrm{Tor}_{j}^{S}(M,\mathcal{Y})=0$ for each $1\leqslant j\leqslant n$, $M\otimes_{S}\mathcal{N}
\subseteq\mathcal{X}_{n-1}^{\vee}$, and $\mathcal{X}_{n-1}^{\vee}$ is closed under extensions, then $(\mathcal{X},\mathcal{M})$ is a right $n$-cotorsion pair in $R$-$\mathrm{Mod}$.
}}
\end{prop}

\begin{prop}\label{prop:2.4}{\it{$(1)$ Let $(\mathcal{X},\mathcal{M})$ be a hereditary left $n$-cotorsion pair in $R$-$\mathrm{Mod}$ and $(\mathcal{Y},\mathcal{N})$ a hereditary left $n$-cotorsion pair in $S$-$\mathrm{Mod}$. Assume that $\mathrm{Ext}_{R}^{j}(M,\mathcal{M})=0$ for each $1\leqslant j\leqslant n$, and $(\mathfrak{D}
_{\mathcal{N}}^{\mathcal{M}})_{n-1}^{\wedge}$ is closed under extensions. Then $(\left(\begin{smallmatrix}
\mathcal{X}\\ \mathcal{Y}\end{smallmatrix}\right),\mathfrak{D}_{\mathcal{N}}^{\mathcal{M}})$ is a hereditary left $n$-cotorsion pair in $\Lambda$-$\mathrm{Mod}$.

$(2)$ Let $(\mathcal{X},\mathcal{M})$ be a hereditary right $n$-cotorsion pair in $R$-$\mathrm{Mod}$ and $(\mathcal{Y},\mathcal{N})$ a hereditary right $n$-cotorsion pair in $S$-$\mathrm{Mod}$. Assume that $\mathrm{Tor}_{j}^{S}(M,\mathcal{Y})=0$ for each $1\leqslant j\leqslant n$, and $(\mathfrak{B}_{\mathcal{Y}}
^{\mathcal{X}})_{n-1}^{\vee}$ is closed under extensions. Then $(\mathfrak{B}_{\mathcal{Y}}^{\mathcal{X}},
\left(\begin{smallmatrix}\mathcal{M}\\ \mathcal{N}\end{smallmatrix}\right))$ is a hereditary right $n$-cotorsion pair in $\Lambda$-$\mathrm{Mod}$.
}}
\end{prop}

\begin{prop}\label{prop:2.4}{\it{
$(1)$ Assume that $(\left(\begin{smallmatrix}\mathcal{X}\\ \mathcal{Y}\end{smallmatrix}\right),\mathfrak{D}
_{\mathcal{N}}^{\mathcal{M}})$ is a hereditary left $n$-cotorsion pair in $\Lambda$-$\mathrm{Mod}$. The following statements hold.

$(a)$ If $\mathrm{Ext}_{R}^{j}(M,\mathcal{M})=0$ for each $1\leqslant j\leqslant n$, then $(\mathcal{X},
\mathcal{M})$ is a hereditary left $n$-cotorsion pair in $R$-$\mathrm{Mod}$.

$(b)$ If $\mathrm{Ext}_{R}^{j}(M,\mathcal{M})=0$ for each $1\leqslant j\leqslant n$, $\mathrm{Hom}_{R}(M,
\mathcal{X})\subseteq\mathcal{N}_{n-1}^{\wedge}$, and $\mathcal{N}_{n-1}^{\wedge}$ is closed under extensions, then $(\mathcal{Y},\mathcal{N})$ is a hereditary left $n$-cotorsion pair in $S$-$\mathrm{Mod}$.

$(2)$ Assume that $(\mathfrak{B}_{\mathcal{Y}}^{\mathcal{X}},\left(\begin{smallmatrix}\mathcal{M}\\ \mathcal{N}\end{smallmatrix}\right))$ is a hereditary right $n$-cotorsion pair in $\Lambda$-$\mathrm{Mod}$. The following statements hold.

$(a)$ If $\mathrm{Tor}_{j}^{S}(M,\mathcal{Y})=0$ for each $1\leqslant j\leqslant n$, then $(\mathcal{Y},
\mathcal{N})$ is a hereditary right $n$-cotorsion pair in $S$-$\mathrm{Mod}$.

$(b)$ If $\mathrm{Tor}_{j}^{S}(M,\mathcal{Y})=0$ for each $1\leqslant j\leqslant n$, $M\otimes_{S}\mathcal{N}
\subseteq\mathcal{X}_{n-1}^{\vee}$, and $\mathcal{X}_{n-1}^{\vee}$ is closed under extensions, then $(\mathcal{X},\mathcal{M})$ is a hereditary right $n$-cotorsion pair in $R$-$\mathrm{Mod}$.
}}
\end{prop}

\begin{rem}\label{prop:2.4}{\rm{For the condition ``~$1\leqslant j\leqslant n+1$~'' appearing in \cite[\text{Theorem}~3.7]{TX2025}, \cite[\text{Theorem}~3.8]{TX2025} and \cite[\text{Theorem}~3.13]{TX2025}, we weaken this condition to ``~$1\leqslant j\leqslant n$~'' in Proposition 5.1, Proposition 5.2 and Proposition 5.3. In addition, we also show that the left (resp. right) $n$-cotorsion pair in $\Lambda$-$\mathrm{Mod}$ obtained in Proposition 5.1 is hereditary, but this point was not mentioned in \cite[\text{Theorem}~3.7]{TX2025}.
}}
\end{rem}

We may also apply Theorem 4.4, Theorem 4.6, Proposition 4.9 and Proposition 4.10 from Section 4 to formal triangular matrix rings. Next, we take Theorem 4.4 or Proposition 4.11 as an illustration and retain the notations introduced for Morita rings in Section 4.

Let $T=\left(\begin{smallmatrix}A&_{A}V_{B}\\0&B\end{smallmatrix}\right)$ be a formal triangular matrix ring, $\mathcal{C}_{1}$ and $\mathcal{C}_{2}$ be classes of left $A$-modules, and $\mathcal{D}_{1}$ and $\mathcal{D}_{2}$ be classes of left $B$-modules. The following statements hold.

$(1)$ Let $(\mathcal{C}_{1},\mathcal{C}_{2})$ and $(\mathcal{D}_{1},\mathcal{D}_{2})$ be left $n$-cotorsion pairs in $A$-$\mathrm{Mod}$ and $B$-$\mathrm{Mod}$, respectively. Assume that $\mathrm{Ext}^{i}_{A}(V,\mathcal{C}_{2})
=0$ for each $1\leqslant i\leqslant n$, $\mathrm{Ext}_{B}^{1}(\mathrm{Hom}_{A}(V,(\mathcal{C}_{2}
)_{n-1}^{\wedge}),(\mathcal{D}_{2})_{n-1}^{\wedge})=0$, $\mathfrak{U}^{\mathcal{C}_{1}}_{\mathcal{D}_{1}}$ is a resolving class, and $(\mathfrak{J}^{\mathcal{C}_{2}}_{\mathcal{D}_{2}})_{n-1}^{\wedge}$ is closed under extensions. Then $(\mathfrak{U}^{\mathcal{C}_{1}}_{\mathcal{D}_{1}},\mathfrak{J}^{\mathcal{C}_{2}}
_{\mathcal{D}_{2}})$ is a hereditary left $n$-cotorsion pair in $T$-$\mathrm{Mod}$.

$(2)$ Let $(\mathcal{C}_{1},\mathcal{C}_{2})$ and $(\mathcal{D}_{1},\mathcal{D}_{2})$ be right $n$-cotorsion pairs in $A$-$\mathrm{Mod}$ and $B$-$\mathrm{Mod}$, respectively. Assume that $\mathrm{Tor}_{i}^{B}(V,\mathcal{D}_{1})
=0$ for each $1\leqslant i\leqslant n$, $\mathrm{Ext}_{A}^{1}((\mathcal{C}_{1})_{n-1}^{\vee},V\otimes_{B}
(\mathcal{D}_{1})_{n-1}^{\vee})=0$, $\mathfrak{U}^{\mathcal{C}_{2}}_{\mathcal{D}_{2}}$ is a coresolving class, and $(\mathfrak{B}^{\mathcal{C}_{1}}_{\mathcal{D}_{1}})_{n-1}^{\vee}$ is closed under extensions. Then $(\mathfrak{B}^{\mathcal{C}_{1}}_{\mathcal{D}_{1}},\mathfrak{U}^{\mathcal{C}_{2}}_{\mathcal{D}_{2}})$ is a hereditary right $n$-cotorsion pair in $T$-$\mathrm{Mod}$.

\begin{rem}\label{prop:2.4}{\rm{From the perspective of categories of modules over trivial ring extensions, compared with \cite[\text{Theorem}~3.7]{TX2025}, we improve the range of the index ``~$i$~'' to ``~$1\leqslant i\leqslant n$~''. Unfortunately, our conclusion requires the extra conditions ``~$\mathrm{Ext}_{B}^{1}(\mathrm{Hom}_{A}(V,(\mathcal{C}_{2})
_{n-1}^{\wedge}),(\mathcal{D}_{2})_{n-1}^{\wedge})=0$ and $\mathrm{Ext}_{A}^{1}((\mathcal{C}_{1})
_{n-1}^{\vee},V\otimes_{B}(\mathcal{D}_{1})_{n-1}^{\vee})=0$~'' which do not appear in
\cite[\text{Theorem}~3.7]{TX2025}. Of course, these same conditions are also additional assumptions compared with Proposition 5.1. Therefore, for the further study of $n$-cotorsion pairs in categories of modules over trivial ring extensions, removing these seemingly redundant conditions would greatly make the results obtained in Section 4 more concise.
}}
\end{rem}

\begin{center}{\bf{Acknowledgments}}\end{center}
The authors sincerely thank the referee for the very helpful suggestions and comments. This work was supported by the National Natural Science Foundation of China (Nos. 11901463, 12361007) and Funds for Innovative Fundamental Research Group Project of Gansu Province (No. 23JRRA684).


\begin{thebibliography}{99}
\small \setlength{\parskip}{1pt}
\bibitem{H1962} Bass, H.: The Morita Theorems. University of Oregon, New York (1962)

\bibitem{WJK2024} Cao, W.Q., Wei, J.Q., Wu, K.L.: Recollements and $n$-cotorsion pairs. J. Algebra Appl. \textbf{26}, 2750224 (2027).


\bibitem{XJ2022} Chen, X.W., Le, J.: Recollements, comma categories and morphic enhancements. Proc. R. Soc. Edinb. A: Math. \textbf{152}(3), 567-591 (2022).

\bibitem{EEE-2000} Enochs, E.E., Jenda, O.M.G.: Relative Homological Algebra. Walter de Gruyter, Berlin (2000)

\bibitem{RPI1975} Fossum, R.M., Griffith, P.A., Reiten, I.: Trivial Extension of Abelian Categories. Springer-Verlag, Berlin (1975)

\bibitem{RJ2012} G\"{o}bel, R., Trlifaj, J.: Approximations and Endomorphism Algebras of Modules. Walter De Gruyter Incorporated, Berlin (2012)

\bibitem{EL1982} Green, E.L.: On the representation theory of rings in matrix form. Pac. J. Math. \textbf{100}(1), 123-138 (1982).

\bibitem{JJ2025} He, J., He, J.: $n$-Cotorsion pairs and recollements of extriangulated categories. J. Algebra Appl. \textbf{25}(11), 2650129 (2026).

\bibitem{JP2022} He, J., Zhou, P.Y.: On the relation between $n$-cotorsion pairs and ($n$+1)-cluster tilting subcategories. J. Algebra Appl. \textbf{21}(1), 2250011 (2022).

\bibitem{HP2019} Holm, H., J{\o}rgensen, P.: Cotorsion pairs in categories of quiver representations. Kyoto J. Math. \textbf{59}(3), 575-606 (2019).

\bibitem{M2002} Hovey, M.: Cotorsion pairs, model category structures, and representation theory. Math. Z. \textbf{241}(3), 553-592 (2002).

\bibitem{JH2022} Hu, J.S., Zhu, H.Y.: Special precovering classes in comma categories. Sci. China Math. \textbf{65}(5), 933-950 (2022).

\bibitem{MOM2021} Huerta, M., Mendoza, O., P\'{e}rez M.A.: $n$-Cotorsion pairs. J. Pure Appl. Algebra \textbf{225}(5), 106556 (2021).

\bibitem{PA2017} Krylov, P., Tuganbaev, A.: Formal Matrices. Springer International Publishing AG, Switzerland (2017)

\bibitem{TX2025} Long, T.L., Zhang, X.X.: $n$-Cotorsion pairs over formal triangular matrix rings. Mediterr. J. Math. \textbf{22}(6), 142 (2025).

\bibitem{YDR2024} Ma, Y.J., Sun, D.D., Zhu, R.M., Hu, J.S.: Recollements induced by left Frobenius pairs. Bull. Malays. Math. Sci. Soc. \textbf{47}(1), 31 (2024).

\bibitem{L2020} Mao, L.X.: Cotorsion pairs and approximation classes over formal triangular matrix rings. J. Pure Appl. Algebra \textbf{224}(6), 106271 (2020).

\bibitem{L2023} Mao, L.X.: Silting and cosilting modules over trivial ring extensions. Commun. Algebra \textbf{51}(4), 1532-1550 (2023).

\bibitem{L2024} Mao, L.X.: Cotorsion pairs and Hovey triples over trivial ring extensions. Commun. Algebra \textbf{52}(1), 148-171 (2024).

\bibitem{N1983} Marmaridis, N.: Comma categories in representation theory.
Commun. Algebra \textbf{11}(17), 1919-1943 (1983).

\bibitem{N1993} Marmaridis, N.: On extensions of abelian categories with applications to ring theory. J. Algebra \textbf{156}(1), 50-64 (1993).

\bibitem{HK2020} Minamoto, H., Yamaura, K.: Homological dimension formulas for trivial extension algebras. J. Pure Appl. Algebra \textbf{224}(8), 106344 (2020).

\bibitem{K1958} Morita, K.: Duality for modules and its applications to the theory of rings with minimum condition. Sci. Rep. Tokyo Kyoiku Diagaku Sect. A \textbf{6}(150), 83-142 (1958).

\bibitem{IJ1973} Palm\'{e}r, I., Roos, J.E.: Explicit formulae for the global homological dimensions of trivial extensions of rings. J. Algebra \textbf{27}(2), 380-413 (1973).

\bibitem{C2014} Psaroudakis, C.: Homological theory of recollements of abelian categories. J. Algebra \textbf{398}, 63-110 (2014).

\bibitem{JJ2009} Rotman, J.J.: An Introduction to Homological Algebra. Springer Science+Business Media LLC, New York (2009)

\bibitem{L1979} Salce, L.: Cotorsion theories for abelian groups.
Symposia Math. \textbf{23}, 11-32 (1979).

\bibitem{WZ2022} Wang, Z.P., Liu, Z.K.: Recollements induced by monomorphism categories. J. Algebra \textbf{594}, 614-635 (2022).

\bibitem{YJD2024} Yuan, Y., He, J., Wu, D.J.: Cotorsion pairs in comma categories.
Czechoslovak Math. J. \textbf{74}(3), 715-734 (2024).

\bibitem{YJDY2025} Yuan, Y., He, J., Wu, D.J., Wang, Y.D.: $n$-Tilting pairs and $n$-cotilting subcategories over comma categories. Georgian Math. J. \textbf{32}(4), 709-720 (2025).

\end{thebibliography}
\end{document}